\documentclass[12pt]{amsart}
\usepackage{amsmath, amssymb, amsthm, amscd, amsfonts}
\usepackage{mathrsfs}
\usepackage{graphicx}
\usepackage{verbatim}
\usepackage{enumerate}

\numberwithin{equation}{section}

\renewcommand{\implies}{\subseteq}
\newcommand{\butnot}{\setminus}
\newcommand{\id}{\mathrm{id}}
\newcommand{\one}{\mathbf{1}}% this symbol isn't quite ideal, but it will work
\newcommand{\points}{\epsilon}
\newcommand{\dist}{d}

\DeclareMathOperator{\diam}{diam}
\newcommand{\del}{\partial}% boundary
\newcommand{\Kin}{\subset\subset}% relatively compact in
\newcommand{\cl}{\overline}% closure
\newcommand{\connected}{\textbf{CC}}% can't be \CC because that means the set of continuous functions from a space
\newcommand{\tendstoc}{{\xrightarrow[c\rightarrow 0]{}}}
\newcommand{\tendstoi}{{\xrightarrow[i]{}}}

\newcommand{\tendston}{{\xrightarrow[n]{}}}
\newcommand{\tendstom}{{\xrightarrow[m]{}}}
\newcommand{\tendstoeps}{{\xrightarrow[\varepsilon]{}}}
\newcommand{\osc}{\mathrm{osc}}

\newcommand{\AL}{{\alpha,l}}

\DeclareMathOperator{\Supp}{Supp}
\renewcommand{\d}{\mathrm{d}}
\newcommand{\RP}{\mathrm{RP}}
\newcommand{\BP}{\mathrm{BP}}
\newcommand{\FP}{\mathrm{FP}}
\DeclareMathOperator{\mult}{mult}
\DeclareMathOperator{\RE}{Re}

\DeclareMathOperator{\Res}{Res}% this is really algebraic geometry
\DeclareMathOperator{\Div}{Div}
\newcommand{\ZW}{{z,w}}
\newcommand{\IETA}{{i,\eta}}
\newcommand{\IXY}{{i,x,y}}
\newcommand{\QEDmod}{\renewcommand{\qedsymbol}{$\triangleleft$}}
\newcommand{\what}{\widehat}
\newcommand{\wtilde}{\widetilde}
\newcommand{\all}{\;\;\forall}% Appropriate amount of space
\newcommand{\on}{\upharpoonleft}
\newcommand{\inc}{b}% this letter wasn't being used for anything else
\newcommand{\A}{\mathbb{A}}
\renewcommand{\AA}{\mathscr{A}}
\newcommand{\Par}{\mathcal{A}}
\newcommand{\B}{\Delta}% unit ball
\newcommand{\BB}{\mathbb{B}}
\newcommand{\BBB}{\mathscr{B}}
\newcommand{\ParB}{\mathcal{B}}
\newcommand{\complexplane}{\mathbb{C}}
\newcommand{\C}{\what{\complexplane}}
\newcommand{\CC}{\mathcal{C}}
\newcommand{\E}{\mathbb{E}}
\newcommand{\EE}{\mathscr{E}}
\newcommand{\F}{\mathscr{F}}
\newcommand{\FF}{\mathcal{F}}
\newcommand{\G}{\mathbb{G}}
\newcommand{\I}{\mathscr{I}}

\newcommand{\JJ}{\mathcal{J}}
\newcommand{\K}{\mathscr{K}}
\renewcommand{\L}{\mathscr{L}}
\newcommand{\LL}{\mathbb{L}}
\newcommand{\M}{\mathcal{M}}
\newcommand{\N}{\mathbb{N}}
\newcommand{\powerset}{\mathbb{P}}% It's good that all four of these look significantly different.
\newcommand{\basemeasure}{\textup{\textbf{P}}}
\newcommand{\pr}{\mathscr{P}}%
\newcommand{\poly}{\mathcal{P}}%

\newcommand{\R}{\mathbb{R}}
\newcommand{\RR}{\mathscr{R}}
\renewcommand{\SS}{\mathcal{S}}
\newcommand{\T}{\mathbb{T}}

\newcommand{\V}{\mathbb{V}}

\newcommand{\X}{\mathbb{X}}

\newcommand{\Z}{\mathbb{Z}}
\newcommand{\ZZ}{\mathcal{Z}}

\newtheorem{theorem}{Theorem}[section]
\newtheorem{lemma}[theorem]{Lemma}
\newtheorem{corollary}[theorem]{Corollary}
\newtheorem{proposition}[theorem]{Proposition}
\newtheorem{claim}[theorem]{Claim}
\newtheorem{construction}[theorem]{Construction}

\theoremstyle{definition}
\newtheorem{definition}[theorem]{Definition}

\newtheorem{event}[theorem]{Event}

\theoremstyle{remark}
\newtheorem{remark}[theorem]{Remark}
\newtheorem{example}[theorem]{Example}
\newtheorem{questions}[theorem]{Questions}

\begin{document}
\author{David Simmons}
\address{University of North Texas, Department of Mathematics, 1155 Union Circle \#311430, Denton, TX 76203-5017, USA}
\email{DavidSimmons@my.unt.edu}
\title{Random Iteration of Rational Functions}
\begin{abstract}
It is a theorem of Denker and Urba\'nski \cite{DU} that if $T:\C\rightarrow\C$ is a rational map of degree at least two and if $\phi:\C\rightarrow\R$ is H\"older continuous and satisfies the ``thermodynamic expanding'' condition $P(T,\phi) > \sup(\phi)$, then there exists exactly one equilibrium state $\mu$ for $T$ and $\phi$, and furthermore $(\C,T,\mu)$ is metrically exact. We extend these results to the case of a holomorphic random dynamical system on $\C$, using the concepts of relative pressure and relative entropy of such a system, and the variational principle of Bogensch\"utz \cite{Bo}. Specifically, if $(T,\Omega,\basemeasure,\theta)$ is a holomorphic random dynamical system on $\C$ and $\phi:\Omega\rightarrow H_\alpha(\C)$ is a H\"older continuous random potential function satisfying one of several sets of technical but reasonable hypotheses, then there exists a unique equilibrium state of $(\X,\T,\phi)$ over $(\Omega,\basemeasure,\theta)$.

Also included is a general (non-thermodynamic) discussion of random dynamical systems acting on $\C$, generalizing several basic results from the deterministic case.% At the same time, we indicate two errors, one substantial and one less substantial, of the paper \cite{DU}, which are found also earlier in Ma\~n\'e's paper \cite{Ma2}. We give a new method which fixes these errors.
\end{abstract}
\maketitle

\begin{section}{Overview} \label{sectionoverview}
Let $T = (T_\omega)_{\omega\in\Omega}$ be a collection of continuous endomorphisms of a topological space $X$ parameterized by a standard Borel probability space $(\Omega,\basemeasure)$, such that the map $\omega\mapsto T_\omega$ is Borel measurable. Let $\theta:\Omega\rightarrow\Omega$ be an ergodic invertible measure-preserving transformation. We call the tuple $(T,\Omega,\basemeasure,\theta)$ a \emph{random dynamical system on $X$}. The dynamics of this system are given by the \emph{pseudo-iterates}
\[
T_\omega^n(x) := T_{\theta^{n-1}\omega}\circ\ldots\circ T_\omega(x).
\]
Random dynamical systems have been studied by several authors, including Kifer \cite{Ki} and Arnold \cite{Ar}. The crucial ergodic theory concepts of entropy and pressure were defined for random dynamical systems in \cite{AR} and \cite{Bo}, respectively. The variational principle is generalized to random dynamical systems in \cite{Bo} under very general hypotheses.

If $X$ is the Riemann sphere $\C$ and the maps $(T_\omega)_\omega$ are all rational functions, we say that $(T,\Omega,\basemeasure,\theta)$ is a \emph{holomorphic random dynamical system}.

% The only relevant article that I know of is \cite{FS}. The setup there can be described in our terminology by assuming that the sequence $(T_{\theta^n\omega})_n$ is independently and identically distributed with respect to $\basemeasure$, and that the distribution for a single map $T_\omega$ is the image of the uniform distribution on the unit disk $\B$ under a holomorphic curve $\zeta:\B\rightarrow\RR_d$, where $\RR_d$ is the set of rational functions of degree $d$; $\RR_d$ is a complex manifold of degree $2d + 1$. In this case it is proven that for $\basemeasure$-almost every $\omega\in\Omega$ and Lebesgue-almost every $x\in\C$, the pseudo-iterates $(T_\omega^n(x))_n$ converge to the attractive cycles of the map $\zeta(0)$.

In this paper we develop the thermodynamic formalism for holomorphic random dynamical systems. Our aim is to generalize several results from the theory of deterministic rational maps $T:\C\rightarrow\C$. For our purposes the first main result is due to Gromov \cite{Gr}, who proved that
\begin{equation}
\label{Gromov}
h_{\text{top}}(T) = \ln(\deg(T)).
\end{equation}
This paper remained unpublished for a long time; the first published proof of (\ref{Gromov}) is found in Lyubich's paper \cite{Ly}. Lyubich also proved the existence of a measure of maximal entropy for $T$, constructed as the limiting distribution of the preimages of some fixed point not in the exceptional set of $T$. The uniqueness of this measure was proven by Ma\~n\'e \cite{Ma1}, who used Ruelle's inequality to show that any measure of positive entropy has a generating partition of finite entropy. All of these results concern only the topological entropy and not the topological pressure; the first result concerning the pressure was given by Denker and Urba\'nski \cite{DU}, who proved the following theorem:

\begin{theorem}[Denker and Urba\'nski, '91] \label{theoremdenkerurbanski}
Suppose that $T$ is a rational map of degree at least two and suppose that $\phi:\C\rightarrow\R$ is H\"older continuous and satisfies
\begin{equation}
\label{DU}
P(\phi) > \sup(\phi).
\end{equation}
Then there is a unique equilibrium state for $(T,\phi)$.
\end{theorem}
% cite changed to ref \replace
In a general dynamical system, (\ref{DU}) might be an unreasonable condition since the pressure is defined as a limit of limits and can rarely be calculated explicitly. However, because of (\ref{Gromov}), (\ref{DU}) follows from the easy to check condition
\[
\sup(\phi) - \inf(\phi) < \ln(\deg(T)).
\]
Theorem \ref{theoremdenkerurbanski} was proven independently by Przytycki \cite{Prz}.

Finally, we will discuss the following result due to Jonsson \cite{Jo}. Fix $d\geq 2$, and let $\RR_d$ be the set of rational functions of degree $d$, endowed with the compact-open topology.
\begin{theorem}[Jonsson, '00] \label{theoremjonsson}
Suppose that $(T,\Omega,\basemeasure,\theta)$ is a holomorphic random dynamical system, such that $\Omega$ is a compact metric space and such that the maps $\theta:\Omega\rightarrow\Omega$ and $T:\Omega\rightarrow\RR_d$ are continuous. Suppose that $h_\basemeasure(\theta) < \infty$. Then there exists a unique measure of maximal relative entropy of $(\X,\T)$ over $(\Omega,\basemeasure,\theta)$. Furthermore
\begin{equation}
\label{Jonsson}
\sup_{\sigma\in \M(\X,\T,\basemeasure)}h_\sigma(\T\on\theta) = \ln(d).
\end{equation}
\end{theorem}
\begin{remark}
Precise definitions of the appropriate generalizations of the notions of pressure, entropy, and equilibria to the setting of holomorphic random dynamical systems are given in Section \ref{sectionergodic}.
\end{remark}
\begin{remark}
The left hand side of (\ref{Jonsson}) is equal to
\[
h_{\text{top},\basemeasure}(\T\on\theta)
\]
by Bogensch\"utz's random variational principle (Theorem \ref{theoremvarprinciple}). Thus (\ref{Jonsson}) generalizes the deterministic equation \textup{(\ref{Gromov})}.
\end{remark}
\begin{proof}[Proof of Theorem \ref{theoremjonsson}]
This theorem is [\cite{Jo} Theorem B(ii,iii)], reformulated using the equation
\[
h_\mu(\T\on\theta) = h_\mu(\T) - h_\basemeasure(\theta),
\]
which is valid since $h_\basemeasure(\theta) < \infty$. 
\end{proof}
Note that the proof of Theorem \ref{theoremjonsson} relies heavily on the use of potential theory. It turns out that this technique is essentially useless when considering a nonzero potential function. Thus new techniques are needed to consider the case $\phi\not\equiv 0$. These techniques come from the Denker-Urbanski paper \cite{DU}; however, some care is needed to make these techniques generalize to the random setting.

The goal of this paper is to prove several generalizations of Theorem \ref{theoremdenkerurbanski}. One of these theorems (Theorem \ref{theorem1}) will also be a generalization of Theorem \ref{theoremjonsson}.

For the remainder of this section, fix a holomorphic random dynamical system $(T,\Omega,\basemeasure,\theta)$ on $\C$ and a random potential function $\phi:\Omega\rightarrow H_\alpha(\C)$ (here $\alpha > 0$ is fixed). Assume that the set
\[
\{\deg(T_\omega):\omega\in\Omega\}
\]
is bounded and does not contain $0$ or $1$. Also assume that the integrability condition
\[
\int \ln\sup_{x\in\C}((T_\omega)_*(x))\d\basemeasure(\omega) < \infty
\]
is satisfied. (Here and elsewhere $(T_\omega)_*(x)$ is the derivative of $T_\omega$ at $x$ with respect to the spherical metric.) In particular, this assumption is satisfied if $T(\Omega)$ is relatively compact.

For each $\omega\in\Omega$ and $n\in\N$, we define the \emph{Perron-Frobenius operator $L_\omega^n:\CC(\C)\rightarrow\CC(\C)$} via the equation
\begin{equation*}
L_\omega^n[f](p) := \sum_{x\in (T_\omega^n)^{-1}(p)}\exp\left(\sum_{j = 0}^{n - 1}\phi_{\theta^j(\omega)}(T_\omega^j(x))\right)f(x).
\end{equation*}
(The sum is counted with multiplicity.)

Our first result is a generalization of Theorem \ref{theoremjonsson}. The strongest hypothesis in this theorem is the fact that $\one$ is an eigenfunction of the Perron-Frobenius operator.

\begin{theorem}
\label{theorem1}
Fix $\alpha > 0$. Suppose that the integrability condition
\begin{align*}
\int\|\phi_\omega\|_\alpha\d\basemeasure(\omega) &< \infty
\end{align*}
holds, and suppose that for each $\omega\in\Omega$, there exists $\lambda_\omega > 0$ so that $L_\omega[\one] = \lambda_\omega \one$.
Then there exists a unique equilibrium state of $(\X,\T,\phi)$ over $(\Omega,\basemeasure,\theta)$. Furthermore
\[
P_{\phi,\basemeasure}(\T\on\theta) = \int \ln(\lambda_\omega)\d\basemeasure(\omega).
\]
\end{theorem}

\begin{corollary}
\label{corollary2}
There exists a unique measure of maximal relative entropy of $(\X,\T)$ over $(\Omega,\basemeasure,\theta)$. Futhermore
\[
h_{\text{top},\basemeasure}(\T\on\theta) := P_{0,\basemeasure}(\T\on\theta) = \int \ln(\deg(T_\omega))\d\basemeasure(\omega),
\]
generalizing \textup{(\ref{Jonsson})} and \textup{(\ref{Gromov})}.
\end{corollary}
\begin{proof}
If $\phi = 0$, then $L_\omega[\one] = \deg(T_\omega)\one$.
\end{proof}

\begin{remark}
Theorem \ref{theoremjonsson} is a corollary of Corollary \ref{corollary2}.
\end{remark}
\begin{proof}
The conclusion of Corollary \ref{corollary2} is the same as the conclusion of Theorem \ref{theoremjonsson}, and the hypotheses are much weaker (in particular, the hypothesis of the compactness of $\Omega$, and therefore of $T(\Omega)$, is replaced by a much milder integrability hypothesis).
\end{proof}

The next theorem concerns random holomorphic dynamical systems which come from perturbing a deterministic dynamical system.

\begin{theorem}
\label{theorem3}
Fix $\alpha > 0$ and $0 \leq \tau < 1$. For every rational function $T_0$ of degree at least two, there exists a neighborhood $\BBB$ of $T_0$ in the compact-open topology such that the following holds:
If $(T,\Omega,\basemeasure,\theta)$ is a holomorphic random dynamical system on $\C$ with $T(\Omega)\implies\BBB$, if $\phi:\Omega\rightarrow\CC(\C,\R)$ is a random potential function, and if:
\begin{align} \label{supphiomega}
\sup_{\omega\in\Omega}\|\phi_\omega\|_\alpha &< \infty\\ \label{suptauinf}
\sup(e^{\phi_\omega}) &\leq \tau \inf(L_\omega[\one]) \all \omega\in\Omega,
\end{align}
then there exists a unique equilibrium state of $(\X,\T,\phi)$ over $(\Omega,\basemeasure,\theta)$.
\end{theorem}

\begin{remark}
(\ref{suptauinf}) follows from the stronger hypothesis
\[
\sup(\phi_\omega) - \inf(\phi_\omega) \leq \deg(T_\omega) - \varepsilon,
\]
where $\varepsilon := -\ln(\tau) > 0$. In particular, this condition is satisfied when $\phi$ is close to $0$.
\end{remark}

\begin{theorem}
\label{theorem4}
Fix $\alpha > 0$, $n\in\N$, and $0 \leq \tau < 1$. For almost every set $A\implies\RR$ of cardinality $n$, in both the topological and the measure-theoretic sense, there exists a neighborhood $\BBB$ of $A$ in the compact-open topology such that the following holds:
If $(T,\Omega,\basemeasure,\theta)$ is a holomorphic random dynamical system on $\C$ with $T(\Omega)\implies\BBB$, if $\phi:\Omega\rightarrow\CC(\C,\R)$ is a random potential function, and if \textup{(\ref{supphiomega})} and \textup{(\ref{suptauinf})} are satisfied, then there exists a unique equilibrium state of $(\X,\T,\phi)$ over $(\Omega,\basemeasure,\theta)$.
\end{theorem}

\begin{remark}
\label{remarkproofs}~
\begin{itemize}
\item Theorem \ref{theorem1} is proved directly from Remark \ref{remarkprobabilistic}, Proposition \ref{propositionnonsingular}, Theorem \ref{maintheorem}, Theorem \ref{theoremmanepartition}, and Theorem \ref{theoremequilibrium}.
\item Theorem \ref{theorem3} is proved directly from Remark \ref{remarkdeterministic}, Theorem \ref{theoremcondition}, Corollary \ref{corollarysufficient}, Theorem \ref{maintheorem}, Theorem \ref{theoremmanepartition}, and Theorem \ref{theoremequilibrium}.
\item Theorem \ref{theorem4} is proved directly from Remark \ref{remarksufficient}, Theorem \ref{theoremcondition}, Corollary \ref{corollarysufficient}, Theorem \ref{maintheorem}, Theorem \ref{theoremmanepartition}, and Theorem \ref{theoremequilibrium}.
\end{itemize}
\end{remark}
\end{section}
\begin{section}{Introduction and fundamental definitions}\label{sectionintroduction}

For every integer $d\geq 1$, let $\RR_d$ be the set of all complex rational maps of degree $d$, endowed with the compact-open topology (equivalently, the uniform topology). Let $\RR=\coprod_{d = 1}^\infty \RR_d$ be the set of all (non-constant) complex rational maps. Note that the map $\circ:\RR\times\RR\rightarrow\RR$, $\circ(f,g) = f\circ g$ is continuous.

There are three ways to iterate rational maps:
\begin{itemize}
\item deterministically, using the same rational map every time (autonomous case)
\item deterministically, using possibly different rational maps each time (non-autonomous case)
\item randomly.
\end{itemize}

The first of these has been well studied; in this paper we are interested in the latter two. We begin by setting up some notation which describes non-autonomous deterministic dynamics:

Fix $m,n\in\Z$ with $m \leq n$. If $(T_j)_{j=m}^{n-1}$ is a finite sequence of rational maps (possibly part of a larger sequence), we denote the composition of its elements by
\[
T_m^n := T_{n-1}\circ T_{n-2}\circ\cdots\circ T_m.
\]
The contravariant map on sets we denote $T_n^m := (T_m^n)^{-1}:2^{\C}\rightarrow 2^{\C}$. The covariant map on measures we denote $T_m^n = (\sigma\mapsto\sigma\circ T_n^m):\M(\C)\rightarrow\M(\C)$. We call the map $T_m^n$ a \emph{pseudo-iterate} of the sequence $(T_j)_j$. Here the set of values for $j$ is left deliberately unspecified. A basic property is that $T_m^n\circ T_j^m = T_j^n$ whenever $j\leq m\leq n$. This is also true when $j\geq m\geq n$, but the meaning is different. If $\AA\implies\RR$ and $n\in\N$, then $\AA^n:=\{T_0^n: (T_j)_{0\leq j < n}\text{ is a sequence in }\AA\}$.

These conventions can be thought of in the following way: For each $n\in\Z$, there exists a different universe $X_n := \C$; the fact that these Riemann surfaces are all conformally equivalent is incidental. For each $m,n\in\Z$ with $m\leq n$, $T_m^n$ denotes a holomorphic map from $X_m$ to $X_n$. It does not make sense to compose two maps $T_{m_1}^{n_1}$ and $T_{m_2}^{n_2}$ unless $m_1 = n_2$, because otherwise the domain and codomain are mismatched.

Our general philosophy will be to put a subscript on every object that lives in a particular universe $X_n$. For objects which move other objects between different universes, the subscript indicates the domain and the superscript indicates the codomain. For example, in the thermodynamic formalism $L_m^n$ will indicate the Perron-Frobenius operator acting as a map from $\CC(X_m)$ to $\CC(X_n)$; see Section \ref{sectionPF}.

We move on to random dynamics. We begin by introducing a notion related to the notion of a random dynamical system, which is also often called a random dynamical system. We distinguish it by calling it a \emph{relative dynamical system}:

\begin{definition}
A \emph{(measurable) relative dynamical system} consists of
\begin{itemize}
\item A probability space $(\Omega,\basemeasure)$
\item An ergodic invertible measure-preserving transformation $\theta:\Omega\rightarrow\Omega$ [This map will usually be notated without parentheses i.e. $\theta\omega := \theta(\omega)$]
\item A measurable space $\X$
\item A measurable transformation $\T:\X\rightarrow \X$
\item A measurable map $\pi:\X\rightarrow\Omega$ such that the diagram commutes, i.e. $\pi\circ\T=\theta\circ\pi$.
\end{itemize}
\end{definition}

The interpretation is that if $\X$ is decomposed into fibers $\X_\omega := \pi^{-1}(\omega)$, then $\T$ becomes a collection of maps between them: $\T_\omega:\X_\omega\rightarrow \X_{\theta\omega}$. Thus $(\X_\omega)_\omega$ are the spaces on which the dynamical maps $(\T_\omega)_\omega$ are acting. Let us make a further interpretation in the special case $\X = \Omega\times X$, where $X$ is a topological space and $\pi = \pi_1$ is projection onto the first coordinate. For each point $(\omega,p)\in \X$, the value $\omega$ tells you what sequence of maps $(\T_{\theta^j\omega})_{j\in\N}$ will be applied to the point $p\in X$. Now, let us suppose we know $p$, but not $\omega$. The point $\omega$ can then be chosen randomly according to the measure $\basemeasure$, and this yields a probability distribution on the set of potential forward orbits of $p$ (a probability measure on $X^\N$); i.e. it tells us where the point might be sent if we decide to iterate randomly. However, it is not always possible to recover the original relative dynamical system based on these probability distributions.

In our case we take $X = \C$. We have the following definition:

\begin{definition}
\label{definitionrandomaction}
A \emph{holomorphic random dynamical system on $\C$} consists of
\begin{itemize}
\item A probability space $(\Omega,\basemeasure)$
\item An ergodic invertible measure-preserving transformation $\theta:\Omega\rightarrow\Omega$
\item A Borel measurable map $T:\Omega\rightarrow\RR$ [This map will usually be denoted in subscript i.e. $T_\omega := T(\omega)$]
\end{itemize}
The triple $(\Omega,\basemeasure,\theta)$ is called the \emph{base system} and the map $T$ is called the \emph{action on $\C$}.

If $(T,\Omega,\basemeasure,\theta)$ is a holomorphic random dynamical system on $\C$, we construct a relative dynamical system in a natural way as a skew-product: Let $\X = \Omega \times \C$, and let $\T:\X\rightarrow\X$ be defined by $\T(\omega,p) := (\theta\omega,T_\omega(p))$. The sextuple $(\Omega,\basemeasure,\theta,\X,\T,\pi_1)$ is called the \emph{relative dynamical system associated with the random dynamical system $(T,\Omega,\basemeasure,\theta)$}.
\end{definition}

Note that $T_*[\basemeasure]$ is the distribution of an individual rational function.
\begin{comment}
The following will be our primary method to construct examples and counterexamples:

\begin{example}
\label{examplepower}
Suppose that $\mu$ is any probability measure on $\RR$. Let $\Omega = \RR^\Z$, let $\basemeasure = \mu^\Z$, let $\theta$ be the shift map, and let $T$ be projection onto the zeroth coordinate. Then $(T,\Omega,\basemeasure,\theta)$ is a holomorphic random dynamical system on $\C$.
\end{example}
\end{comment}

In several sections of this paper, we will deal with a fixed holomorphic random dynamical system $(T,\Omega,\basemeasure,\theta)$. We will almost always avoid mentioning explicitly the dependence of objects on a fixed element $\omega\in\Omega$. In particular, for all $j\in\Z$ we will use $j$ as shorthand for $\theta^j\omega$, i.e. $T_j := T_{\theta^j\omega}$, $\phi_j := \phi_{\theta^j\omega}$, and so on. Thus for each $\omega\in\Omega$ we have a doubly infinite sequence of rational functons $(T_j)_{j\in\Z}$. Conversely, if some object (such as the Julia set) depends on the sequence $(T_{\theta^j\omega})_j$, this can be indicated by simply using a subscript of $\omega$, i.e. $\JJ_\omega := \JJ_0\on_\omega$ is the Julia set of the sequence $(T_{\theta^j\omega})_j$. In fact we will use this convention even when talking about deterministic autonomous sequences i.e. $\JJ_m\on_{(T_j)_j} := \JJ_0\on_{(T_{m + j})_j}$.

The relation between the pseudo-iterates of the sequence $(T_j)_j = (T_{\theta^j\omega})_j$ and the relative dynamical system $(\Omega,\basemeasure,\theta,\X,\T,\pi_1)$ is given by the formula
\[
T_m^n(p) = \pi_2(\T^{n-m}(\theta^m\omega,p))
\]
This motivates the following notation: $T_\omega^n := T_{\T^{n-1}\omega}\circ\ldots\circ T_\omega = T_0^n\on_\omega$. Note that we could not have written $\T^n\omega$ as the superscript, as the map $n\mapsto\T^n\omega$ is not necessarily injective. Even in the case that this map is injective, the map $(\omega,n)\mapsto T_\omega^n$ is measurable whereas the map $(\omega,\T^n\omega)\mapsto T_\omega^n$ is not.

We define an \emph{event} to be a proposition whose truth value depends on $\omega$; the \emph{probability} of a measurable event, denoted $\pr(\mathrm{event})$, is the $\basemeasure$-measure of the set of all $\omega$ which satisfy the event e.g. $\pr(\omega\in A) = \basemeasure(A)$. Similarly, a \emph{random variable} is a real number which depends on $\omega$, and the \emph{expected value} of a integrable random variable, denoted $\EE[\text{random variable}]$, is the integral against $\basemeasure$ of the function which sends $\omega$ to the value which the random variable takes on corresponding to that value of $\omega$. If necessary, for non-measurable events, the phrase ``the probability of the event is at least $x$" means that the inner measure of the event is at least $x$.

A fundamental property is \emph{translation-invariance}: if $A$ is a measurable event, and $B$ is the event obtained by translating each index occuring in the statement of $A$ by some fixed amount, then $\pr(B) = \pr(A)$. A similar statement holds for the expected values of random variables.

\begin{lemma}
\label{lemmapoincare}
If $A$ is an event of positive probability, then it is almost certain that infinitely many translates of $A$ will occur in both directions.
\end{lemma}
\begin{proof}
This is a corollary of the Poincar\'e recurrence theorem, together with the fact that $\basemeasure$ is ergodic.
\end{proof}

\begin{lemma}
\label{lemmaflip}
If $(A_n)_{n\in\N}$ is a sequence of measurable events with
\[
\pr(\text{$A_n$ is satisfied by $\omega$ for all $n\in\N$ sufficiently large}) = 1
\]
then
\[
\pr(\exists n\in\N\text{ such that $A_n$ is satisfied by $\theta^{-n}\omega$}) = 1.
\]
\end{lemma}
\begin{proof}
Fix $\varepsilon > 0$. Continuity of measures gives $N\in\N$ such that
\[
\pr(\text{$A_N$ is satisfied by $\omega$})
\geq \pr(\text{$A_n$ is satisfied by $\omega$ for all $n\geq N$})
\geq 1 - \varepsilon.
\]
Translation invariance of probabilities gives
\[
\pr(\exists n\in\N\text{ such that $A_n$ is satisfied by $\theta^{-n}\omega$})
\geq \pr(\text{$A_N$ is satisfied by $\theta^{-N}\omega$})
\geq 1 - \varepsilon.
\]
Taking the supremum over all $\varepsilon > 0$ yields the lemma.
\end{proof}

\begin{remark}
The assumption here that the $A_n$ are measurable is crucial. (Indeed, this assumption is crucial whenever a continuity of measures argument is invoked.) In general we will not verify this assumption directly, but refer to the Appendix (Section \ref{sectionappendix}).
\end{remark}

We include a section on notational conventions (Section \ref{sectionnotation}) in case there is any confusion. Some conventions we list here:

$0\in\N$.

Unless explicitly stated, variables are allowed to take on the value $\infty$. However, they are nonnegative unless otherwise stated.

Much of the time, we work with multisets rather than sets. If $A$ and $B$ are multisets, $f$ is a function, and $C$ is a standard set, then the expressions $A\cup B$, $A\cap C$, $\#(A)$, $f(A)$, $\sum_{x\in A}f(x)$, and $\one_A$ should all be interpreted in the multiset-theoretical sense.\footnote{For those familiar with algebraic geometry: A multiset can be thought of as an effective divisor on $\C$. The operations of union, cardinality, forward image, inverse image, and intersection with a standard set correspond to the divisor concepts of sum, degree, push-forward, pullback, and restriction to a subdomain of $\C$, respectively. The operation of summation over a multiset has no direct analogue in the theory of divisors. The characteristic function of a multiset is merely the divisor associated with it, interpreted as a map from $\C$ to $\N$.} If there is a star i.e. $\#^*(A)$ or $\sum_{x\in A}^* f(x)$, then the expression should be interpreted in the regular set-theoretic sense. If $T$ is a holomorphic map of Riemann surfaces, and if $p$ is in the codomain of $T$, then by $T^{-1}(p)$, $\RP_T$, $\BP_T$, and $\FP_T$ we mean the multisets consisting of all preimages of $p$, ramification points, branch points, and fixed points, respectively, counting multiplicity, so that $\BP_T = T(\RP_T)$, and
\begin{align*}
\#(T^{-1}(p) &= \deg(T)\\
\#(\BP_T) &= \#(\RP_T) = 2\deg(T) - 2\\
\#(\FP_T) &= \deg(T) + 1
\end{align*}
assuming $\deg(T)\geq 2$. For example, $\#^*(T^{-1}(p))$ is the absolute number of preimages of $p$, not counting multiplicity.

In some cases, we allow maps between multisets that send two copies of the same point to different places; see Lemmas \ref{lemmacomplex} and \ref{lemmamorecomplex}. 

The word ``multiplicity'' is used in three senses in this paper. The multiplicity of a multiset is the maximum of its characteristic function; its multiplicity at a point is its characteristic function evaluated at that point. The multiplicity of a point $p$ relative to a rational map $T$ will be denoted $\mult_T(x)$ or just $\mult(x)$. If $T$ is a rational map and $U\implies\C$ is open, the multiplicity $\mult(V)$ of a connected component $V\in\connected(T^{-1}(U))$ is the degree of the map $T\on V:V\rightarrow U$ as a proper map between Riemann surfaces. Alternatively, $\mult(V)=\#(T^{-1}(x)\cap V)$ for all $x\in U$.

% A way to prove the legitimacy of concept (3) of multiplicity:
% Let $(S_i)_{i<q}$ be the connected components of $f^{-1}(S)$. Fix $i<q$. Then the function $\#(T^{-1}(\cdot)\cap I^{-1}(S_i)):S\rightarrow\N$ is continuous and integer valued, and thus locally constant. (See for example [\cite{B} p.87, last paragraph]) Since $S$ is connected, it follows that $\#(T^{-1}(y)\cap I^{-1}(S_i))=\#(T^{-1}(x)\cap I^{-1}(S_i))$.
% \bibitem[Be91]{B} A. Beardon, \emph{Iteration of rational maps}, Complex analytic dynamical systems. Graduate Texts in Mathematics, 132. Springer-Verlag, New York, 1991.
\end{section}

\begin{section}{Definitions}\label{sectionbasic}
We begin this section by studying non-autonomous sequences of rational maps, and end by studying holomorphic random dynamical systems.

Suppose that $(T_j)_{j\in\N}$ is a sequence of rational maps. We define the Fatou and Julia sets of the sequence $(T_j)_{j\in\N}$ as follows: A point $x\in\C$ is \emph{Fatou} [with respect to $(T_j)_j$] if it has a neighborhood $U$ such that the sequence $(T_0^n\on U)_{n\in\N}$ is a normal family. The set of Fatou points is called the \emph{Fatou set} and is denoted by $\FF_0$, and the set of non-Fatou points is called the \emph{Julia set} and is denoted by $\JJ_0$. Since normality is a local property, it follows that the sequence $(T_0^n\on \FF_0)_{n\in\N}$ is a normal family. This definition is a clear analogue of the definition of the Fatou set in the deterministic case.

If $T$ is a rational function, denote the set of its totally ramified points [points with ramification degree $\deg(T) - 1$] by $\SS_T$. The \emph{exceptional set} of a sequence $(T_j)_{j\in\N}$ is the set
\[
\SS_0 := \{x\in\C: x\in\bigcap_{n\in\N}\SS_{T_0^n}\text{ for all }n\in\N\}
\]
i.e. $\SS_0$ is the set of points whose iterates are all totally ramified. The exceptional set of a constant sequence $(T)_j$ is equal to the exceptional set of $T$ defined in the standard way. Note that $T_n\on\SS_n$ is always injective.

\begin{remark}
\label{remarkinvariant}
For any $m,n\in\Z$, $m<n$, we have $T_n^m(\FF_n)=\FF_m$ and $T_n^m(\JJ_n)=\JJ_m$. In other words, the Fatou and Julia sequences $(\FF_n)_n$ and $(\JJ_n)_n$ are fully $(T_n)_n$-invariant. However, the exceptional set $(\SS_n)_n$ is not in general fully invariant; it is only forward invariant. See also Remark \ref{remarkinvariantexceptional} below.
\end{remark}

\begin{definition}
\label{definitionlinear}
A sequence $(T_j)_{j\in\N}$ is \emph{linear} if $\deg(T_j) = 1$ for all $j\in\N$, and \emph{quasilinear} if $\deg(T_j) = 1$ for all but finitely many $j\in\N$.
\end{definition}

The assumption of nonlinearity or of non-quasilinearity has some immediate applications:

\begin{remark}
\label{remarkatmosttwo}
If $(T_j)_{j\in\N}$ is nonlinear then $\#(\SS_0)\leq 2$. If $(T_j)_{j\in\N}$ is linear then $\SS_0 = \C$.
\end{remark}
\begin{proof}
As in the deterministic case, this follows from the Riemann-Hurwitz formula for the number of ramification points of a rational map.
\end{proof}

\begin{remark}
\label{remarknonempty}
If $(T_n)_n$ is a sequence of rational functions such that $\deg(T_n)\tendston \infty$, then $(T_n)_n$ is not a normal family. In particular, if $(T_j)_j$ is not quasilinear, then $\JJ_0 \neq \emptyset$.
\end{remark}
\begin{proof}
The first assertion follows from the fact that the map $T\mapsto\deg(T)$ is continuous from the compact-open topology, and never takes infinity as a value. The second assertion follows from the first plus the fact that the degree is multiplicative.
\end{proof}

\begin{lemma}
\label{lemmafatoucompact}
Suppose that $(T_j)_{j\in\N}$ is not quasilinear. For every $\kappa > 0$ such that $B_s(\SS_0,\kappa) \Kin \FF_0$, and for every $\kappa_2 > 0$, there exists $N\in\N$ such that for all $n\geq N$,
\[
T_0^n(B_s(\SS_0,\kappa))\implies B_s(\SS_n,\kappa_2).
\]
\end{lemma}
\begin{proof}
Fix $\kappa > 0$ such that $B_s(\SS_0,\kappa) \Kin \FF_0$, and fix $\kappa_2 > 0$. For each $x\in\SS_0$,
\begin{equation}
\label{sequenceconvergent}
(T_0^n\on \cl{B_s}(x,\kappa))_{n\in\N}
\end{equation}
is a normal family. Now
\begin{equation*}
\mult_{T_0^n}(x) 
= \deg(T_0^n)
\tendston \infty.
\end{equation*}
Thus no subsequence of (\ref{sequenceconvergent}) can converge to a non-constant map, because multiplicity at a point is upper-semicontinuous in the compact-open topology. It follows that $\diam_s(T_0^n(B_s(x,\kappa)))\tendston 0$. Thus for sufficiently large $n$,
\begin{equation*}
T_0^n(B_s(x,\kappa)) \implies B_s(T_0^n(x),\kappa_2).
\end{equation*}
But $T_0^n(x)\in\SS_n$. Since $\#(\SS_0) < \infty$ by Remark \ref{remarkatmosttwo}, we have that for sufficiently large $n$,
\begin{equation*}
T_0^n(B_s(\SS_0,\kappa)) \implies B_s(\SS_n,\kappa_2).
\end{equation*}
\end{proof}

Of course, this lemma is moot if $\SS_0 \implies \JJ_0$. In the deterministic case, this is fine, since $\SS_0$ is always a subset of the Fatou set. However, in the random case, this is not true; see Proposition \ref{examplesingularnew}.

\begin{definition}
\label{definitionsingular}
A sequence of rational functions $(T_j)_j$ is \emph{singular} if $\SS_0\cap\JJ_0\neq\emptyset$.
\end{definition}

This concludes our study of non-autonomous non-thermodynamic dynamics. We move on to random dynamics:

\begin{comment}
\begin{definition}
\label{definitiontotal}
If $(T,\Omega,\basemeasure,\theta)$ is a holomorphic random dynamical system on $\C$, we define the \emph{total Fatou set}, the \emph{total Julia set}, and the \emph{total exceptional set} by 
\begin{align*}
\FF &:= \{(\omega,p)\in\X:p\in \FF_\omega\}\\
\JJ &:= \{(\omega,p)\in\X:p\in \JJ_\omega\}\\
\SS &:= \{(\omega,p)\in\X:p\in \SS_\omega\}.
\end{align*}
(Recall the conventions established in Section 1 for dealing with objects dependent on elements of $\Omega$.)
\end{definition}
\end{comment}

\begin{remark}
\label{remarkmeasurable}
Suppose that $(T,\Omega,\basemeasure,\theta)$ is a holomorphic random dynamical system on $\C$. The maps $\omega\mapsto\JJ_\omega$ and $\omega\mapsto\SS_\omega$ are Effros measurable, i.e. Borel measurable when the codomain $\K(\C)$ is given the Vietoris topology. (For more information see the Appendix (Section \ref{sectionappendix}).)
\end{remark}
\begin{proof}
Fix $\delta > 0$. We have
\begin{align*}
&\{\omega\in\Omega: \cl{B_s}(x,\delta) \cap\JJ_\omega \neq\emptyset\}\\
&= \left\{\omega\in\Omega: \forall \delta_2 > \delta,H < \infty\text{ rational }\exists n\in\N\text{ such that}\sup_{\cl{B_s}(x,\delta_2)}(T_0^n)_* \geq H\right\}% see appendix
\end{align*}
which is measurable by Theorem \ref{theorempropositions}. By a standard criterion for Effros measurability, the map $\omega\mapsto\JJ_\omega$ is Effros measurable.

Now
\begin{align*}
\SS_\omega
&= \left\{x: \forall n\in\N\;\;\exists p\in\C\text{ such that }T_n^0(p)\implies\{x\}\right\}
\end{align*}
which depends measurably on $\omega$ by Theorem \ref{theorempropositions}.% see appendix
\end{proof}

\begin{corollary}
\label{corollarymeasurable}
The maps $\omega\mapsto\#(\SS_\omega)$ and $\omega\mapsto\#(\SS_\omega\cap\JJ_\omega)$ are measurable. Thus the sets $\{\omega\in\Omega:(T_j)_j\text{ is quasilinear}\}$ and $\{\omega\in\Omega:(T_j)_j\text{ is singular}\}$ are measurable.
\end{corollary}
\begin{proof}
This follows directly from Corollary \ref{corollaryconventions}.% see appendix
\end{proof}

\begin{remark}
\label{remarkinvariantexceptional}
If $(T,\Omega,\basemeasure,\theta)$ is a holomorphic random dynamical system on $\C$, then $\#(\SS_\omega)$ is independent of $\omega$ for $\basemeasure$-almost every $\omega\in\Omega$. In particular, $\pr((\SS_n)_n\text{ is fully invariant}) = 1$.
\end{remark}
\begin{proof}
For all $\omega\in\Omega$, $T_\omega(\SS_\omega)\implies \SS_{\theta\omega}$, so $\#(\SS_\omega)\leq\#(\SS_{\theta\omega})$. Since $\basemeasure$ is ergodic, it follows that there exists a constant $m = 0,1,2,\infty$ [see Remark \ref{remarkatmosttwo}] such that $\#(\SS_\omega) = m$ almost surely.

Fix $\omega\in\Omega$ such that $\#(\SS_n) = m$ for all $n\in\Z$. If $m = \infty$, this means that $T_j$ is degree one for all $j\in\Z$; thus $\SS_n = \C$ for all $n$, and $(\SS_n)_n$ is fully invariant. If $m < \infty$, then for each $j\in\Z$, we have 
\[
T_j(\SS_j) \implies \SS_{j + 1}
\]
Since $T_j\on\SS_j$ is injective, it follows that both sides have cardinality $m$. Thus we have equality. Since each point of $\SS_j$ is totally ramified, we have
\[
\SS_j = (T_j)^{-1}(T_j(\SS_j)) = (T_j)^{-1}(\SS_{j + 1}),
\]
i.e. $(\SS_n)_n$ is fully invariant.
\end{proof}

Definitions \ref{definitionlinear} and \ref{definitionsingular} generalize straightforwardly to the random setting:

\begin{definition}
A holomorphic random dynamical system $(T,\Omega,\basemeasure,\theta)$ is \emph{linear} if $\deg(T_0) = 1$ almost surely, \emph{antilinear} if $\deg(T_0)\geq 2$ almost surely, and \emph{singular} if $(T_j)_j$ is almost certainly singular.
\end{definition}

\begin{remark}
\label{remarkpoincare}
If $(T,\Omega,\basemeasure,\theta)$ is nonlinear, then $(T_j)_j$ is almost certainly not quasilinear. If $(T,\Omega,\basemeasure,\theta)$ is nonsingular, then $(T_j)_j$ is almost certainly nonsingular.
\end{remark}
\begin{proof}
Lemma \ref{lemmapoincare} and Remark \ref{remarkmeasurable}, together with the observation that for all $m\in\N$, if $(T_{m + j})_j$ is nonsingular then $(T_j)_j$ is nonsingular.
\end{proof}

Thus by Remark \ref{remarknonempty}, if $(T,\Omega,\basemeasure,\theta)$ is a nonlinear holomorphic action on $\C$, then $\JJ_\omega \neq \emptyset$ for almost all $\omega\in\Omega$.

To show that nonsingularity is a nontrivial requirement, we give an example where it fails. In this example, the dynamics are not destroyed completely, so that suggests that there may be some interest in investigating singular actions. However we also show that under reasonable hypotheses nonsingularity holds.

\begin{comment}
\begin{example}
\label{examplesingularold}
Let $\poly\implies\RR$ be the set of all non-constant polynomials, taken with the compact-open topology as a set of maps from $\complexplane$ to $\complexplane$. (This is different from the subspace topology, in that the set of polynomials of a given degree do not form a clopen set.) For any positive measure $\wtilde{\mu}$ on $\poly$ with $\wtilde{\mu}(\poly) < 1$, there exists a probability measure $\mu\geq\wtilde{\mu}$ on $\poly$ such that if $(T,\Omega,\basemeasure,\theta)$ are as in Example \ref{examplepower}, then almost certainly $\JJ_0 = \SS_0 = \{\infty\}$.
\end{example}
\end{comment}
\begin{proposition}
\label{examplesingularnew}
There exists $(T,\Omega,\basemeasure,\theta)$ a holomorphic random dynamical system on $\C$ such that
\[
\pr(\SS_0 = \{\infty\}\implies \JJ_0\text{ and }\mathrm{HD}(\JJ_0)\geq 1) = 1.
\]
\end{proposition}
\begin{proposition}
\label{propositionnonsingular}
Suppose that $(T,\Omega,\basemeasure,\theta)$ is antilinear. If
\[
\EE[\ln\sup(T_*)] < \infty,
\]
then $(T,\Omega,\basemeasure,\theta)$ is nonsingular.
\end{proposition}
In particular, if $T_*[\basemeasure]$ is a point measure or is supported on a compact set, then $(T,\Omega,\basemeasure,\theta)$ is nonsingular. (However, the proofs of these special cases could be simplified.)

To prove Proposition \ref{examplesingularnew}, we will use the following lemma:
\begin{lemma}
\label{lemmasudden}
For each $n\in\N$, suppose that $f_n:\N^n\rightarrow\N$. Then there exists a probability measure $\sigma$ on $\N$ such that
\begin{equation}
\label{sudden}
\sigma^\N[(k_n)_{n\in\N}:\exists\text{ infinitely many }n\in\N\text{ such that }k_n \geq f_n(k_0,\ldots,k_{n - 1})] = 1.
\end{equation}
\end{lemma}
The idea is that according to $\sigma^\N$, the odds are that every once in a while, something will happen which is much more significant than anything which has happened before; informally we could write
\[
\sigma^\N[(k_n)_{n\in\N}:\exists\text{ infinitely many }n\in\N\text{ such that }k_n >> k_0,\ldots,k_{n - 1}] = 1.
\]
$(f_n)_n$ specifies exactly what we mean by ``$>>$''.
\begin{proof}[Proof of Lemma \ref{lemmasudden}:]
Let
\[
\mu = \sum_{\ell\in\N}2^{-(\ell + 1)}\delta_\ell
\]
be the geometric distribution. An elementary exercise in probability shows that $\mu^\N$ satisfies
\[
\mu^\N[(\ell_n)_{n\in\N}:\exists\text{ infinitely many }n\in\N\text{ such that }\ell_n > \ell_j\all j < n\text{ and such that }n\leq 3^{\ell_n}] = 1.
\]
Define the sequence $(k_\ell)_{\ell\in\N}$ by induction:
\begin{equation}
\label{kell}
k_\ell := \max_{n\leq 3^\ell}\max_{\substack{(\ell_j)_{j = 0}^{n - 1} \\ \ell_j < \ell\all j < n}}f_n(k_{\ell_0},\ldots,k_{\ell_{n - 1}})
\end{equation}
and let $\sigma = (\ell\mapsto k_\ell)_*[\mu]$. Fix $(\ell_n)_{n\in\N}$ such that there exist infinitely many $n\in\N$ such that $\ell_n > \ell_j$ for all $j < n$ and such that $n\leq 3^{\ell_n}$. For each such $n$, by (\ref{kell}) we have
\[
k_{\ell_n}\geq f_n(k_{\ell_0},\ldots,k_{\ell_{n - 1}}).
\]
Thus
\[
\mu^\N[(\ell_n)_{n\in\N}:\exists\text{ infinitely many }n\in\N\text{ such that }k_{\ell_n} \geq f_n(k_{\ell_0},\ldots,k_{\ell_{n - 1}})] = 1.
\]
which clearly implies (\ref{sudden}).
\end{proof}
\begin{proof}[Proof of Proposition \ref{examplesingularnew}:]

Define $Q:\complexplane\rightarrow\RR$ by
\begin{equation}
\label{Qdef}
Q_c(z) := 3z^2 - 2z^3 + c z^2 (z - 1)^2.
\end{equation}
Note that for each $c\in\complexplane$, the points $0$, $1$, and $\infty$ are all fixed ramification points, but only $\infty$ is totally ramified.

Now, $\deg(Q_c) = 4$ for $c\neq 0$, but $\deg(Q_0) = 3$. Clearly $Q_c\tendstoc Q_0$ locally uniformly on $\complexplane$, but we cannot have $Q_c\tendstoc Q_0$ uniformly on $\C$, because $\deg:\RR\rightarrow\N$ is continuous. Thus $(Q_c)_{0<c\leq 1}$ is not a normal family in any neighborhood of $\infty$. So for all $k\in\N$ there exists $0 < c_k \leq 1$ such that
\[
Q_{c_k}(B_s(\infty,2^{-k}))\nsubseteq \C\butnot B_e(0,1).
\]
For each $n\in\N$ and for each $n$-tuple $(k_j)_{j = 0}^{n - 1}$, let $f = f_n(k_j)_j\geq n$ be large enough so that
\[
Q_{c_{k_{n - 1}}}\circ\ldots\circ Q_{c_{k_0}}(B_s(\infty,2^{-k}))\supseteq B_s(\infty,2^{-f}).
\]
Let $\sigma\in\M(\N)$ be the probability measure given by Lemma \ref{lemmasudden}. Let $\Omega = \N^\Z$, let $\basemeasure = \sigma^\Z$, let $\theta$ be the shift map, and let $T:\Omega\rightarrow\RR$ be given by
\[
T(k_n)_n := Q_{c_{k_0}}.
\]
Fix $\omega = (k_n)_n\in\Omega$ and assume that there exist infinitely many $n\in\N$ such that $k_n \geq f_n(k_0,\ldots,k_{n - 1})$; by Lemma \ref{lemmasudden}, this assumption is almost certainly valid. For each $n\in\N$ such that $k_n > f_n(k_0,\ldots,k_{n - 1})$, we have
\begin{align*}
T_0^{n + 1}(B_s(\infty,2^{-k_n}))
&= Q_{c_{k_n}}\circ\ldots\circ Q_{c_{k_0}}(B_s(\infty,2^{-k_n}))\\
&\supseteq Q_{c_{k_n}}(B_s(\infty,2^{-f_n(k_j)_{j = 0}^{n - 1}})\\
&\supseteq Q_{c_{k_n}}(B_s(\infty,2^{-k_n})\\
&\nsubseteq \C\butnot B_e(0,1).
\end{align*}
But since $k_n \geq n$, we have
\[
T_0^{n + 1}(B_s(\infty,2^{-n}))\nsubseteq \C\butnot B_e(0,1).
\]
Since this is true for all $n\in\N$, it is clear that $\infty\in\JJ_0$. Based on (\ref{Qdef}), it is clear that $\SS_0 = \{\infty\}$. Thus, all that remains is to show that $\mathrm{HD}(\JJ_0)\geq 1$.

First, we show that $0,1\in\FF_0$. To see this, note that the family $(Q_c)_{0 < c \leq 1}$ is normal on $\complexplane$, and $0$ and $1$ are superattracting fixed points of this family. The result then follows from an elementary calculation.

Since $\mult_{T_0^n}(0),\mult_{T_0^n}(1) \geq 2^n\tendston\infty$, we have that $T_0^n\tendston 0$ locally uniformly on the connected component of $\FF_0$ containing $0$, and $T_0^n\tendston 1$ locally uniformly on the connected component of $\FF_0$ containing $1$. Thus $\FF_0$ has two distinct connected components, so $\mathrm{HD}(\JJ_0)\geq 1$.
\end{proof}
To prove Proposition \ref{propositionnonsingular}, we will use the following lemma:
\begin{lemma}
\label{lemmaC10}
Suppose that $T\in\RR$, and that $x,y\in\C$. Let
\[
\wtilde{H} = \frac{32}{\pi^2} \sup(T_*)^2.
\]
Then
\begin{align} \label{C9}
\|(T_*)_*(x)\|_s^e \leq \wtilde{H}\\ \label{C10}
|T_*(x) - T_*(y)| &\leq \wtilde{H}\dist_s(x,y)\\ \label{C11}
\dist_s(T(x),T(y)) &\leq \dist_s(x,y)[T_*(x) + \frac{\wtilde{H}}{2}\dist_s(x,y)]
\end{align}
\end{lemma}
\begin{proof}
Integration along the geodesic connecting $x$ and $y$ yields the sequence of implications
\[
(\ref{C9}) \Rightarrow (\ref{C10}) \Rightarrow (\ref{C11}).
\]
Thus we are reduced to proving (\ref{C9}). To this end, fix $x\in\C$. Without loss of generality, we may assume that $x = T(x) = 0$. Then
\begin{align*}
T_*(z) &= |T'(z)|\frac{1 + |z|^2}{1 + |T(z)|^2}\\
(T_*)_*(0) &= |T''(0)|,
\end{align*}
since the conversion factor is equal to one up to second order. Now, the mean value inequality gives
\begin{align*}
T\left(\cl{B_s}\left(0,\frac{\pi}{4H}\right)\right) &\implies \cl{B_s}\left(0,\frac{\pi}{4}\right)\\
T\left(\cl{B_e}\left(0,\tan\left(\frac{\pi}{4H}\right)\right)\right) &\implies \cl{B_e}\left(0,\tan\left(\frac{\pi}{4}\right)\right) = \cl{B_e}(0,1)\\
|T| &\leq 1 \hspace{.5 in} \left[\text{on }\cl{B_e}\left(0,\tan\left(\frac{\pi}{4H}\right)\right)\right]
\end{align*}
We can now give an elementary bound from the Cauchy integral formula:
\begin{align*}
T(z) &= \frac{1}{2\pi\imath}\int_{w\in S_e(0,\frac{\pi}{4H})}\frac{T(w)\d w}{w - z}\\
T''(z) &= \frac{2!}{2\pi\imath}\int_{w\in S_e(0,\frac{\pi}{4H})}\frac{T(w)\d w}{(w - z)^3}\\
|T''(0)| &\leq \frac{2}{2\pi}\int_{w\in S_e(0,\frac{\pi}{4H})}\frac{|\d w|}{|w|^3}\\
&= \frac{2}{(\pi/(4H))^2} = \frac{32}{\pi^2}H^2.
\end{align*}
\end{proof}
\begin{proof}[Proof of Proposition \ref{propositionnonsingular}]
For each $\omega\in\Omega$, define
\[
H_\omega := \frac{32}{\pi^2}\sup(T_*)^2.
\]
Clearly, $\EE[\ln(H)] < \infty$.

Fix $\omega\in\Omega$ and assume that $(T_j)_j$ is non-quasilinear, and that there exists $C < \infty$ such that
\begin{equation}
\label{ergodic}
\sum_{j = 0}^{n - 1}\ln(H_j) \leq Cn
\end{equation}
for all $n\in\N$. By the Birkhoff ergodic theorem, this assumption is almost certainly valid.

Let
\[
\delta = e^{-C}.
\]
\begin{claim}
For all $x\in\SS_0$, for all $y\in B_s(x,\delta)$, and for all $n\in\N$
\[
\dist_s(T_0^n(x),T_0^n(y)) \leq \delta\prod_{j = 0}^{n - 1}\frac{H_j}{2e^C} \leq 2^{-n}.
\]
\end{claim}
\begin{proof}
First note that (\ref{ergodic}) implies
\begin{equation}
\label{useagain}
\delta\prod_{j = 0}^{n - 1}\frac{H_j}{2e^C}
\leq 2^{-n}\delta
\leq \min(2^{-n},e^{-C}),
\end{equation}
proving the right hand inequality.

The proof of the left hand inequality is by induction on $n$:

Base case $n = 0$: By hypothesis.

Inductive step: Assume the claim is true for $n$. Since $x\in\SS_0$ and since $(T_j)_j$ is non-quasilinear, we have $T_0^n(x)\in\RP_{T_n}$ i.e. $(T_n)_*\circ T_0^n(x) = 0$. By Lemma \ref{lemmaC10},
\begin{align*}
\dist_s(T_0^{n + 1}(x),T_0^{n + 1}(y)) &\leq \dist_s(T_0^n(x),T_0^n(y))\left((T_n)_*\circ T_0^n(x) + \frac{H_n}{2}\dist_s(T_0^n(x),T_0^n(y))\right)\\
&\leq \left(\delta\prod_{j = 0}^{n - 1}\frac{H_j}{2e^C}\right)^2\frac{H_n}{2}\\
&\leq \delta\prod_{j = 0}^{n - 1}\frac{H_j}{2e^C}\frac{H_n}{2e^C}
= \delta\prod_{j = 0}^n\frac{H_j}{2e^C};
\end{align*}
the last inequality following from (\ref{useagain}).
\QEDmod\end{proof}
It follows that $\diam(T_0^n(B_s(x,\delta)))$ tends to zero as $n$ approaches infinity. Thus $B_s(x,\delta)\implies \FF_0$; in particular $x\in \FF_0$. Since this is true for all $x\in\SS_0$, we have $\SS_0\implies \FF_0$.
\end{proof}
In deterministic dynamics, if $U$ is an open set such that some subsequence of $(T^n\on U)_n$ is a normal family, then $U$ is Fatou i.e. the entire sequence is normal. We give a counterexample to a similar claim in random dynamics, although it is in a sense cheating since the example is conjugate to a deterministic action.
\begin{example}
\label{exampleclusterdivergent}
There exists a nonlinear nonsingular holomorphic random dynamical system $(T,\Omega,\basemeasure,\theta)$ on $\C$ such that the following event is almost certain to occur:
\begin{quote}
\begin{event}
\label{eventclusterdivergent}
There exists an open set $U$ intersecting the Julia set and an increasing sequence $(n_i)_i$ in $\N$ such that
\[
\diam(T_0^{n_i}(U))\tendstoi 0.
\]
In particular, $(T_0^{n_i}\on U)_i$ is a normal family.
\end{event}
\end{quote}
\end{example}
\begin{proof}
Fix $S\in\RR\butnot\RR_1$ with no totally ramified points and with a geometrically attracting fixed point $p$, i.e. $0 < S_*(p) < 1$. Let $B$ be a neighborhood of $p$ which is relatively compact in the attracting basin of $p$, so that $\diam(S^n(B))\tendston 0$.

For each $j\in\N$, fix $\phi_j\in\RR_1$ such that $(\phi_j)_*(p) = j$. For each $n\in\N$, let
\[
f_n(k_0,\ldots,k_{n - 1}) := f(n) := n\lceil S_*(p)^{-n}\rceil.
\]
Let $\sigma\in\M(\N)$ be the measure guaranteed by Lemma \ref{lemmasudden}. Let
\begin{align*}
\Omega &= \N^\Z\\
\basemeasure &= \sigma^\Z\\
\theta(k_j)_j &= (k_{j + 1})_j\\
T(k_j)_j &= \phi_{k_1}\circ S\circ \phi_{k_0}^{-1}
\end{align*}
Clearly, $(T,\Omega,\basemeasure,\theta)$ is nonlinear and nonsingular; furthermore we have
\[
T_0^n(k_j)_j = \phi_{k_n}\circ S^n\circ \phi_{k_0}^{-1}.
\]
Fix $\omega = (k_j)_j\in\Omega$ and assume that there exist infinitely many $n\in\N$ with $k_n\geq f(n)$ and that $k_n\underset{n}{\nrightarrow}\infty$; by Lemma \ref{lemmasudden}, this assumption is almost certainly valid.
We have
\begin{align*}
(T_0^n)_*(\phi_{k_0}(p)) &= \frac{(\phi_{k_n})_*(p)(S_*(p))^n}{(\phi_{k_0})_*(p)}\\
&= \frac{k_n (S_*(p))^n}{k_0}
\end{align*}
Thus for infinitely many $n\in\N$, we have
\[
(T_0^n)_*(\phi_{k_0}(p)) \geq \frac{f(n) (S_*(p))^n}{k_0} \geq \frac{n}{k_0}\tendston \infty,
\]
so $\phi_{k_0}(p)\in\JJ_0$.

Let $U = \phi_{k_0}(B)$, so that $U\cap \JJ_0\neq\emptyset$. Since $k_n\underset{n}{\nrightarrow}\infty$, there exist $k\in\N$ and an increasing sequence $(n_i)_i$ such that $k_{n_i} = k$ for all $i\in\N$. Now
\begin{align*}
T_0^{n_i}(U) &= \phi_{k_{n_i}}\circ S^{n_i}(B) = \phi_k \circ S^{n_i}(B)\\
\diam(T_0^{n_i}(U)) &\leq \sup(\phi_k)_* \diam(S^{n_i}(B))\tendstoi 0.
\end{align*}
\end{proof}

\end{section}

\begin{section}{Topological exactness}\label{sectionexact}
In this section, we fix a nonlinear holomorphic random dynamical system $(T,\Omega,\basemeasure,\theta)$.

We begin by considering the following event:

\begin{quote}
\begin{event}
\label{eventexactJJ}
Suppose $U\implies\C$ is open with $U\cap \JJ_0\neq\emptyset$. Then there exists $n\in\N$ such that $T_0^n(U)\supseteq \JJ_n$. (In other words, $(T_j)_j$ is topologically exact on the Julia set.)
\end{event}
\end{quote}

We would like to prove that Event \ref{eventexactJJ} is almost certain to occur under reasonable assumptions. The most obvious assumptions are nonlinearity and nonsingularity, but these are insufficient to ensure Event \ref{eventexactJJ}. Indeed, consider Example \ref{exampleclusterdivergent}, and fix $\omega\in\Omega$ satisfying Event \ref{eventclusterdivergent}. If $T_0^n(U)\supseteq \JJ_n$ for some $n$, then we would have that $(T_n^{n_i})_i$ was normal on both $\FF_n$ and $\JJ_n$, and thus on all of $\C$, contradicting Remark \ref{remarknonempty}. Thus $T_0^n(U)\nsupseteq \JJ_n$ for all $n\in\N$ i.e. the sequence $(T_j)_j$ is not topologically exact on $(\JJ_j)_j$.

One option would be to introduce stronger hypotheses and prove that Event \ref{eventexactJJ} holds under these hypotheses. We will take a different approach. Note that Event \ref{eventexactJJ} implies a distinction between open sets $U$ on which the sequence $(T_0^n\on U)_n$ is normal and those on which some subsequence is normal. We define the \emph{uniform Julia set} to be the set of all points $x\in\JJ_0$ such that for every increasing sequence of integers $(n_i)_i$, the sequence $(T_0^{n_i}\on U)_i$ is not a normal family. The uniform Julia set is denoted $\JJ_0'$. Like the Julia set and the exceptional set, the uniform Julia set enjoys total invariance $T_n^m(\JJ_n') = \JJ_m'$. A priori, it is not clear that $\JJ_0'\neq\emptyset$; this will follow from our hypotheses of nonlinearity and nonsingularity. The main results of this section are that $(T_j)_j$ is topologically exact on $(\JJ_j')_j$ almost surely (Proposition \ref{propositionexactUJ}), and that $\JJ_0'$ is almost certainly uncountable and perfect (Theorem \ref{theoremperfect}).

As an intermediate step, we prove a property weaker than exactness for the actual Julia set. Essentially we replace in the definition of exactness ``For all $U$'' by ``There exist arbitrarily small $U$''.

The idea of the proof is that if a rational map $T$ is ``mixing'', if an open set $U$ is ``large'', and if a compact set $K$ is ``far away from the exceptional set'', then $T(U)\supseteq K$. The details are to specify what these concepts mean, and to prove that they happen some of the time.

The ``mixing'' of a rational map $T$ can be measured by the complexity of a preimage of an arbitrary point far from the exceptional set. The exceptional set must be ignored since the preimage of any exceptional point is a singleton, which is trivial. The exceptional set does not mix; instead, we measure the degree of mixing outside the exceptional set. We measure the complexity of the set $T^{-1}(p)$ first in terms of cardinality, and secondly in terms of the concept of ``$m$-diameter'' defined below.

\begin{lemma}
\label{lemmasix}
Suppose $(T_j)_{j\in\N}$ is a non-quasilinear sequence of rational functions. For all $m\in\N$, there exists $L\in\N$ such that for all $\ell\geq L$ and for all $p\in\C\butnot\SS_\ell$ we have $\#^*(T_\ell^0(p))\geq m$.
\end{lemma}
\begin{proof}
The sequence of sets
\[
(K_\ell)_{\ell\in\N} := \left(\{p\in\C:\#^*(T_\ell^0(p)) < m\}\right)_{\ell\in\N}
\]
is backward invariant. Thus the function $\ell\mapsto\#(K_\ell)$ is nonincreasing; since $(T_j)_j$ is non-quasilinear this function is eventually finite, and therefore eventually constant, say for $\ell\geq L$. But then for all $\ell\geq L$, we have
\begin{align*}
T_{\ell + 1}^\ell(K_{\ell + 1}) &\implies K_\ell\\
\#^*(T_{\ell + 1}^\ell(K_{\ell + 1})) &\geq \#^*(K_{\ell + 1}) = \#^*(K_\ell)
\end{align*}
Thus $T_{\ell + 1}^\ell(K_{\ell + 1}) = K_\ell$, and $(K_\ell)_{\ell\geq L}$ is fully invariant. Thus every point in $K_\ell$ is totally ramified. Thus we have $K_\ell\implies \SS_\ell$, proving the lemma.
\end{proof}

We next want to rephrase Lemma \ref{lemmasix} in a quantitative way. To do this we will need a generalization of diameter:

For each (multi)set $K\implies\C$ and for each $m\in\N$ define the $m$-diameter of $K$, denoted $\diam_m(K)$, by
\begin{equation}
\label{mdiameter}
\diam_m(K)=\sup\left\{\min_{i\neq j}\dist_s(x_i,x_j):(x_i)_{i = 1}^m\in K^m\right\}.
\end{equation}
We will only use the $m$-diameter of the spherical metric. The case $m = 2$ gives the ordinary diameter. Note that if $K$ is compact, then the supremum in (\ref{mdiameter}) is actually achieved. Also note that

\begin{itemize}
\item[1)] $\diam_m(K)>0$ if and only if $\#^*(K)\geq m$
\item[2)] $m_1\leq m_2$ implies $\diam_{m_2}(K)\leq\diam_{m_1}(K)$
\item[3)] $K_1\implies K_2$ implies $\diam_m(K_1)\leq\diam_m(K_2)$
\item[4)] $\diam_m:2^{\C}\rightarrow\R$ is Lipschitz continuous with a corresponding constant of $2$, if $2^{\C}$ has the Hausdorff metric and $\R$ has the standard metric
\end{itemize}

\begin{corollary}
\label{corollarysixquantitative}
Let $(T,\Omega,\basemeasure,\theta)$ be a nonlinear holomorphic random dynamical system on $\C$. For all $m\in\N$ and for all $\varepsilon,\kappa > 0$ there exist $\ell\in\N$ and $\delta > 0$ so that the following event is true with probability at least $1 - \varepsilon$:
\begin{quote}
\begin{event}
\label{eventsixquantitative}
For all $p\in\C\butnot B_s(\SS_\ell,\kappa)$, we have $\diam_m(T_\ell^0(p))\geq \delta$.
\end{event}
\end{quote}
\end{corollary}
\begin{proof}
It is enough to show: % see appendix
\[
\pr(\exists L\in\N\text{ such that }\forall \ell\geq L, \exists\delta > 0\text{ such that Event \ref{eventsixquantitative} is satisfied}) = 1.
\]
Fix $\omega\in\Omega$ and assume that $(T_j)_j$ is not quasilinear; by Remark \ref{remarkpoincare}, this assumption is almost certainly valid. By Lemma \ref{lemmasix}, there exists $L\in\N$ such that for all $\ell\geq L$ and for all $p\in\C\butnot\SS_\ell$, we have $\#^*(T_\ell^0(p))\geq m$. Fix $\ell\geq L$. We have $\diam_m(T_\ell^0(p)) > 0$ for all $p\in\C\butnot\SS_\ell$. Since $\C\butnot B_s(\SS_\ell,\kappa)$ is compact and since $p\mapsto \diam_m(T_\ell^0(p))$ is continuous, there exists $\delta > 0$ such that $\diam_m(T_\ell^0(p)) \geq \delta$ for all $p\in\C\butnot B_s(\SS_\ell,\kappa)$. Thus we are done.
\end{proof}

We now are ready to give a precise meaning to the ``largeness'' quality of an open set $U$ discussed in the paragraph preceding Lemma \ref{lemmasix}. For each $\delta>0$, we define the set $\G_\delta\implies 2^{\C}$ by
\[
\G_\delta := \{U\implies\C\text{ open}: \diam_3(\C\butnot U) \geq \delta\}
\]
Open sets $U\in \G_\delta$ we consider ``small''. The set $\G_\delta$ has two important properties, which can be stated deterministically:

\begin{lemma}
\label{lemmaspillover}
Fix $\delta > 0$ and $K\implies\C$, and suppose that $T$ is a rational map. If $\diam_3(T^{-1}(p))\geq\delta$ for all $p\in K$, then for all $U\notin\G_\delta$ we have $T(U)\supseteq K$.
\end{lemma}
\begin{lemma}
\label{lemmajuliainfinite}
Fix $\delta > 0$, and suppose that $(T_j)_{j\in\N}$ is a sequence of rational maps. If $T_n(U)\in \G_\delta$ for all $n\in\N$, then $(T_n\on U)_{n\in\N}$ is a normal family.
\end{lemma}
\begin{proof}[Proof of Lemma \ref{lemmaspillover}:]
By contradiction, suppose that there exists $p\in K\butnot T(U)$. Then $T^{-1}(p)\implies \C\butnot U$. Thus 
\begin{equation*}
\diam_3(\C\butnot U) \geq \diam_3(T^{-1}(p)) \geq \delta.
\end{equation*}
i.e. $U\in\G_\delta$, contradicting our hypothesis.
\end{proof}
\begin{proof}[Proof of Lemma \ref{lemmajuliainfinite}:]
For each $n\in\N$, we have $\diam_3(\C\butnot T_n(U))\geq\delta$. Since $\C\butnot T_n(U)$ is compact, there exist points $(p_i^{(n)})_{i=0}^2$ in $\C\butnot T_n(U)$ which are $\delta$-separated.

Let $a_0=0$, $a_1=1$, $a_2=\infty$, and let $\pi:\{(x_i)_{i=0}^2\in\C^3:x_i\text{ are distinct}\}\rightarrow\RR_1$ be the map which sends each triple $(x_i)_{i=0}^2$ to the unique M\"obius transformation $\phi$ such that $\phi(a_i)=x_i$ for each $i=0,1,2$. Now let $K=\{(x_i)_{i=0}^2\in\C^3:x_i\text{ are $\delta$-separated}\}$; $K$ is compact. Since $\pi$ is an algebraic morphism, it follows that $\pi$ is continuous. Thus $\K:=\pi(K)$ is compact.

For each $n\in\N$ let $\phi_n=\pi(p_i^{(n)})_{i=0}^2$. Then $\phi_n\in\K$ for all $n\in\N$; thus $(\phi_n)_{n\in\N}$ is a normal family. Now consider the family $(S_n)_{n\in\N}:=(\phi_n^{-1}\circ T_n)_{n\in\N}$. Note that $S_n(U)\implies\complexplane\butnot\{0,1\}$. By Montel's theorem, $(S_n\on U)_{n\in\N}$ is a normal family. Since the composition of two normal families is normal, it follows that $(T_n\on U)_{n\in\N}$ is a normal family.
\end{proof}

We are now ready to prove

\begin{theorem}
\label{theoremexact}
Suppose that $(T,\Omega,\basemeasure,\theta)$ is nonlinear and nonsingular. Then the following event is almost certain to occur:
\begin{quote}
\begin{event}
\label{eventexact}
Fix $\kappa,\delta > 0$. Then there exist $N\in\N$ and $x\in \JJ_0$ such that $T_0^n(B_s(x,\delta))\supseteq\C\butnot B_s(\SS_n,\kappa)$ and $T_0^n(B_s(x,\delta))\supseteq \JJ_n$ for all $n\geq N$.
\end{event}
\end{quote}
\end{theorem}
\begin{proof}
Fix $\kappa > 0$ small enough so that $\pr(\dist_s(\JJ_0,\SS_0) > \kappa)\geq 2/3$. By Corollary \ref{corollarysixquantitative}, there exist $\delta_2 > 0$ and $\ell\in\N$ such that the probability of Event \ref{eventsixquantitative} occurs with $m = 3$ and $\delta = \delta_2$ is at least $2/3$. Fix $\omega\in\Omega$ and assume that $(T_j)_j$ non-quasilinear and that there exists an increasing sequence $(n_k)_{k\in\N}$ such that for all $k\in\N$,
\begin{itemize}
\item Event \ref{eventsixquantitative} occurs for $\theta^{n_k}\omega$ and $\delta = \delta_2$
\item $\dist_s(\JJ_{n_k},\SS_{n_k}) > \kappa$
\end{itemize}
By Lemma \ref{lemmapoincare}, this assumption is almost certainly valid.

By contradiction, suppose that $T_0^{n_k + \ell}(B_s(x,\delta))\nsupseteq \C\butnot B_s(\SS_{n_k + \ell},\kappa)$ for all $k\in\N$ and for all $x\in \JJ_0$. Fix $x\in\C$.
\begin{itemize}
\item If $x\in \JJ_0$, then by Lemma \ref{lemmaspillover}, $T_0^{n_k}(B_s(x,\delta))\in \G_{\delta_2}$ for all $k\in\N$. By Lemma \ref{lemmajuliainfinite}, $(T_0^{n_k}\on B_s(x,\delta))_k$ is a normal family.
\item If $x\in \FF_0$, then there exists a neighborhood $U_x$ of $x$ such that the sequence $(T_0^n\on U_x)_n$ is a normal family; in particular the subsequence $(T_0^{n_k}\on U_x)_k$ is a normal family.
\end{itemize}
Since normality is a local property, it follows that $(T_0^{n_k})_{k\in\N}$ is a normal family. But this is impossible by Proposition \ref{remarknonempty}.

Thus there exists $k\in\N$ and $x\in \JJ_0$ so that $T_0^{n_k + \ell}(B_s(x,\delta))\supseteq \C\butnot B_s(\SS_{n_k + \ell},\kappa)$. Now
\begin{equation*}
\C\butnot T_0^{n_k + \ell}(B_s(x,\delta))
\implies B_s(\SS_{n_k + \ell},\kappa)
\Kin \FF_0.
\end{equation*}
Thus by Lemma \ref{lemmafatoucompact}, we have that Event \ref{eventexact} occurs almost certainly for fixed $\delta$, $\kappa$ sufficiently small. Since $\delta$ and $\kappa$ can be quantified countably, we are done.
\end{proof}

As a corollary, we obtain that the uniform Julia set is nonempty:

\begin{lemma}
\label{lemmauniformnonempty}
Event \ref{eventexact} implies that $\JJ_0'\neq\emptyset$, assuming $(T_j)_j$ is not quasilinear.
\end{lemma}
\begin{proof}
Let $\kappa = 1 > 0$. For each $k\in\N$, there exist $N_k\in\N$ and $x_k\in\JJ_0$ satisfying $T_0^n(B_s(x_k,2^{-k}))\supseteq \JJ_n$ for all $n\geq N_k$. Let $x\in\JJ_0$ be a cluster point of the sequence $(x_k)_k$; we claim that $x\in\JJ_0'$. If $U$ is any neighborhood of $x$, then there exists $k\in\N$ so that $B_s(x_k,2^{-k})\implies U$, and thus $T_0^n(U)\supseteq \JJ_n$, where $n = N_k$. By contradiction, suppose that there exists an increasing sequence $(n_i)_i$ such that $(T_0^{n_i}\on U)_i$ is a normal family. Then $(T_n^{n_i}\on T_0^n(U))_i$ is also a normal family. But then $(T_n^{n_i})_i$ is normal both on $\FF_n$ and $\JJ_n$, contradicting Remark \ref{remarknonempty}. Thus there exists no such sequence $(n_i)_i$, and $x\in\JJ_0'$.
\end{proof}

We now go on to prove our main results about the uniform Julia set:

\begin{proposition}
\label{propositionexactUJ}
Suppose that $(T,\Omega,\basemeasure,\theta)$ is nonlinear and nonsingular. Then the following event is almost certain to occur:
\begin{quote}
\begin{event}
\label{eventexactUJ}
Suppose $U\implies\C$ is open with $U\cap \JJ_0'\neq\emptyset$. Then there exists $n\in\N$ such that $T_0^n(U)\supseteq \JJ_n\supseteq\JJ_n'$.
\end{event}
\end{quote}
\end{proposition}
\begin{proof}
Fix $\omega\in\Omega$ and assume that there exist an increasing sequence $(n_i)_i$, an integer $\ell\in\N$, and $\kappa,\delta > 0$ such that for all $i\in\N$,
\begin{itemize}
\item $\diam_3(T_{n_i + \ell}^{n_i}(p))\geq \delta$ for all $p\in\C\butnot B_s(\SS_{n_i + \ell},\kappa)$
\item $\JJ_{n_i + \ell}\implies \C\butnot B_s(\SS_{n_i + \ell},\kappa)$
\end{itemize}
By Corollary \ref{corollarysixquantitative} and Lemma \ref{lemmapoincare}, this assumption is almost certainly valid. Suppose that $U\implies\C$ is open with $U\cap \JJ_0'\neq\emptyset$. Then $(T_0^{n_i}\on U)_i$ is not a normal family. By Lemma \ref{lemmajuliainfinite}, there exists $i\in\N$ with $T_0^{n_i}(U)\notin\G_\delta$. By Lemma \ref{lemmaspillover}, $T_0^n(U) = T_{n_i}^n(T_0^{n_i}(U)) \supseteq \C\butnot B_s(\SS_n,\kappa)\supseteq\JJ_n$, where $n = n_i + \ell$. Thus we are done.
\end{proof}

\begin{theorem}
\label{theoremperfect}
Suppose that $(T,\Omega,\basemeasure,\theta)$ is nonlinear and nonsingular. Then the following event is almost certain to occur:
\begin{quote}
\begin{event}
\label{eventperfect}
The set $\JJ_0'$ is uncountable and perfect.
\end{event}
\end{quote}
\end{theorem}
\begin{proof}
First, note that by Lemmas \ref{lemmauniformnonempty} and \ref{lemmasix}, we have that $\#(\JJ_0')\geq 2$ almost surely. Fix $\omega\in\Omega$ and assume that Event \ref{eventexactUJ} occurs, and that $\#(\JJ_n')\geq 2$ for all $n\in\N$; this assumption is almost certainly valid. Since $\JJ_0'$ is nonempty, to show that $\JJ_0'$ is uncountable it suffices to show that $\JJ_0'$ is perfect. To this end, by contradiction suppose that $x\in\JJ_0'$ is an isolated point of $\JJ_0'$, i.e. there exists $\delta > 0$ so that $\JJ_0' \cap B_s(x,\delta) = \{x\}$. By Event \ref{eventexactUJ}, there exists $n\in\N$ with $T_0^n(B_s(x,\delta))\supseteq\JJ_n'$. But then 
\[
\JJ_n' = T_0^n(B_s(x,\delta)) \cap \JJ_n' = T_0^n(B_s(x,\delta) \cap \JJ_0') = \{T_0^n(x)\},
\]
contradicting that $\#(\JJ_n')\geq 2$. Thus $x$ is not isolated, and $\JJ_0'$ is perfect.
\end{proof}
\end{section}

\begin{section}{Perron-Frobenius Operator: Definition, Notation, and Fundamental Lemma}\label{sectionPF}

Suppose that $T$ is a rational map and that $\phi\in\CC(\C)$ is a potential function. We define the \emph{Perron-Frobenius operator $L:\CC(\C)\rightarrow\CC(\C)$ associated with $(T,\phi)$} via the equation
\begin{equation}
\label{Ldef}
L[f](p) := \sum_{x\in T^{-1}(p)}\exp(\phi(x))f(x).
\end{equation}
(Here, finally, we are using the conventions about multiplicity established in Section \ref{sectionintroduction}.)

If $(T_j)_j$ is a sequence of rational maps and if $(\phi_j)_j$ is a sequence of potential functions, then $L_j$ denotes the Perron-Frobenius operator associated with $(T_j,\phi_j)$. We denote the pseudo-iterates by $L_m^n := L_{n-1} \circ\cdots\circ L_m$, and the Birkhoff sums by $\phi_m^n := \sum_{j=m}^{n-1}\phi_j\circ T_m^j$ [$m < n$ in both cases]. It is an easy exercise to show that $L_m^n$ is the Perron-Frobenius operator associated with $(T_m^n,\phi_m^n)$, i.e.
\begin{equation}
\label{Ldefiterated}
L_m^n[f](p) = \sum_{x\in T_n^m(p)}\exp(\phi_m^n(x))f(x).
\end{equation}
We denote the dual operator by exchanging indices, so that
\[
L_n^m[\nu] := \int \left[\sum_{x\in T_n^m(p)}\exp(\phi_m^n(x))\delta_x\right]\d\nu(p).
\]
If $(T,\Omega,\basemeasure,\theta)$ is a holomorphic random dynamical system on $\C$, we define a \emph{random potential function on $(T,\Omega,\basemeasure,\theta)$} to be a measurable map $\phi:\Omega\rightarrow\CC(\C)$. As in Section \ref{sectionintroduction}, we shorten $\phi_j := \phi_j(\omega) := \phi(\theta^j\omega)$ for $\omega\in\Omega$ fixed, and $\phi := \phi_0$.

The following lemma will be used repeatedly. In words, it says that if $f$ is a function and $g>0$ is a test function, then the convex hull of the range of $L[f]/L[g]$ is a subinterval of the convex hull of the range of $f/g$.
\begin{lemma}
\label{lemmanonincreasing}
Suppose that $T$ is a rational map, and suppose that $\phi$ is a potential function. For any $f,g\in\CC(\C)$ with $g>0$ and for any $K \implies \C$, we have
\begin{align}\label{suplessthan}
\sup_K \frac{L[f]}{L[g]} &\leq \sup_{T^{-1}(K)}\frac{f}{g} \\ \label{inflessthan}
\inf_K \frac{L[f]}{L[g]} &\geq \inf_{T^{-1}(K)}\frac{f}{g} \\ \label{osclessthan}
\left\|\frac{L[f]}{L[g]}\right\|_{\osc,K} &\leq \left\|\frac{f}{g}\right\|_{\osc,T^{-1}(K)}
\end{align}
\end{lemma}
\begin{proof}
Fix $p\in K$; for all $x\in T^{-1}(p)$, $f(x)\leq g(x) \sup_{T^{-1}(K)}(f/g)$. Summing over all $x\in T^{-1}(p)$, dividing by $L[g](p)$, and taking the supremum over all $p\in K$ yields (\ref{suplessthan}). A similar argument yields (\ref{inflessthan}). Subtracting (\ref{inflessthan}) from (\ref{suplessthan}) yields (\ref{osclessthan}).
\end{proof}

\end{section}
\begin{section}{Statement of Key Theorems}\label{sectionmaintheorems}

\begin{definition}
\label{definitionboundeddistortion}
Fix a holomorphic random dynamical system $(T,\Omega,\basemeasure,\theta)$ on $\C$, and a random potential function $\phi:\Omega\rightarrow\CC(\C)$. We say that $X\implies\C$ has the \emph{bounded distortion property} if
\begin{itemize}
\item[A)] $X$ is closed, connected, contains at least three points, and its complement $B := \C\butnot X$ satisfies
\[
T(B)\Kin B
\]
almost surely.
\item[B)] There exists $M<\infty$ so that for all $j\in\N$,
\begin{equation}
\label{Misaboundzero}
\|\ln(L_0^j[\one])\|_{\osc,X}\leq M
\end{equation}
almost surely. Equivalently, for all $n,j\in\Z$ with $j\leq n$,
\begin{equation}
\label{Misabound}
\|\ln(L_j^n[\one])\|_{\osc,X}\leq M
\end{equation}
almost surely.
\end{itemize}

$X$ has the \emph{equicontinuity property} if $X$ satisfies (A) and if
\begin{itemize}
\item[C)] There exists $\gamma$ a modulus of continuity such that for all $n\in\N$
\begin{equation}
\label{gammaisaboundzero}
\rho_{\ln(L_0^n[\one])}^{(X)}\leq \gamma
\end{equation}
almost surely. Equivalently, for all $n,j\in\Z$ with $j\leq n$,
\begin{equation}
\label{gammaisabound}
\rho_{\ln(L_j^n[\one])}^{(X)}\leq \gamma
\end{equation}
almost surely.
\end{itemize}
\end{definition}

Clearly, the equicontinuity property implies the bounded distortion property.

\begin{remark}
\label{remarkprobabilistic}
If $\one$ is a pseudo-eigenvalue of the Perron-Frobenius operator, i.e. $\pr(L[\one]\text{ is constant}) = 1$, then $\C$ has the equicontinuity property. For example, this is true if $\phi = 0$. Another sufficient condition is given below in Theorem \ref{theoremcondition}.
\end{remark}

\begin{theorem}
\label{maintheorem}
Fix $\alpha,\beta > 0$. Suppose that $(T,\Omega,\basemeasure,\theta)$ is a nonsingular holomorphic random dynamical system on $\C$ with a potential function $\phi:\Omega\rightarrow\CC(\C)$, suppose that $X\implies\C$ has the bounded distortion property, and suppose that
\begin{align} \label{Disaboundintegral}
\EE[\deg(T)^\beta] &< \infty\\ \label{C1isaboundintegral}
\EE[\|\phi\|_\AL^\beta] &< \infty\\ \label{tauisaboundintegral}
\EE[\sup(\phi)] &< \EE[\ln\inf(L[\one])].
\end{align}
Then the following event is almost certain to occur:
\begin{quote}
\begin{event}
\label{mainevent}
Suppose that $\gamma_1$ and $\gamma_2$ are moduli of continuity. For all $\varepsilon,\kappa > 0$ there exists $N\in\N$ such that for all $n\geq N$ and for all $f,g\in\CC(\C)$ with $g>0$, $\rho_{f/g}^{(X)}\leq\gamma_1$, and $\rho_{\ln(g)}^{(X)}\leq \gamma_2$,
\begin{equation}
\label{mainequation}
\left\|\frac{L_0^n[f]}{L_0^n[g]}\right\|_{\osc,X\butnot B_s(\SS_n,\kappa)\cup \JJ_n} \leq \varepsilon.
\end{equation}
\end{event}
\end{quote}
\end{theorem}

\begin{remark}
The ``thermodynamic expanding'' condition (\ref{tauisaboundintegral}) appears to be different from the condition $P > \sup(\phi)$ used in \cite{Prz} and \cite{DU}, however they are not so different. In \cite{Prz} it is stated that the main advantage of the condition $P > \sup(\phi)$ is that it is checkable; in fact due to the fact that $h_{\text{top}}(T) = \ln(\deg(T))$ it is sufficient to check $\sup(\phi) - \inf(\phi) < \ln(\deg(T))$. But this condition also implies (\ref{tauisaboundintegral}) (or its deterministic counterpart). Furthermore, in the random setting the equation $h_{\text{top}}(T) = \ln(\deg(T))$ has not been proven; in fact one of the results of this paper is a generalization of this equation (Corollary \ref{corollaryhequalslndeg}). It would be silly to assume what we are trying to prove.
\end{remark}

\begin{remark}
\label{remarkfullstrength}
For most purposes, it suffices to consider (\ref{mainequation}) with $X\butnot B_s(\SS_n,\kappa)\cup \JJ_n$ replaced by just $\JJ_n$. However, we will use the full strength in proving Theorem \ref{theoremequilibrium} (uniqueness of equilibrium states). Similarly, although the statement becomes simpler if we move the $f$ and $g$ quantifiers outside of the $\varepsilon$ and $\kappa$ quantifiers, and from there delete $\gamma_1$ and $\gamma_2$ from the statement entirely, and from there implicitize $\varepsilon$, $N$, and $n$ by replacing (\ref{mainequation}) with an equation about limits, the full strength is needed to prove Corollary \ref{corollaryeigenvalues}.
\end{remark}

\begin{remark}
\label{remarkthermodynamic}
(\ref{tauisaboundintegral}) implies that $(T,\Omega,\basemeasure,\theta)$ is nonlinear. Thus in the proof of Theorem \ref{maintheorem} we may use Theorem \ref{theoremexact} and Corollary \ref{corollarysixquantitative}.
\end{remark}

The following lemma explains the use of the bounds (\ref{Disaboundintegral}) - (\ref{tauisaboundintegral}). It will be used in Corollary \ref{corollaryZAB} and Lemma \ref{lemmahyperbolic}.

\begin{lemma}
\label{lemmamotivation}
Suppose that $(T,\Omega,\basemeasure,\theta)$ satisfies the hypotheses of Theorem \ref{maintheorem}. Fix $\varepsilon > 0$. Then there exist $D,C_1,C_2 < \infty$ and $\tau < 1$ so that for each $n\in\Z$, the probability that for all $j \leq n - 1$
\begin{align}\label{Disaboundmod}
\deg(T_j) &\leq D(n - j)^{1/\beta}\\ \label{C1isaboundmod}
\|\phi_j\|_\AL &\leq C_1(n - j)^{1/\beta}\\ \label{tauisaboundmod}
e^{\sup(\phi_j^n)} &\leq e^{C_2}\tau^{n-j}\inf(L_j^n[\one])
\end{align}
is at least $1 - \varepsilon$.
\end{lemma}
\begin{proof}
Choose
\begin{align*}
D_0 &> \EE[\deg(T)^\beta]\\
C_0 &> \EE[\|\phi\|_\AL^\beta]\\
\ln(\tau) &> \EE[\sup(\phi) - \ln\inf(L[\one])].
\end{align*}
By the Birkhoff ergodic theorem, for each $\omega\in\Omega$ there exists $C < \infty$ so that for al $m\leq n$,
\begin{align} \label{BETstart}
\sum_{j = m}^n\deg(T_j)^\beta &\leq (n - m)D_0 + C\\
\sum_{j = m}^n\|\phi_j\|_\AL^\beta &\leq (n - m)C_0 + C\\ \label{BETend}
\sum_{j = m}^n\ln(\sup(\phi_j) - \inf(L_j[\one])) &\leq (n - m)\ln(\tau) + C.
\end{align}
Thus by continuity of measures, there exists $C$ such that (\ref{BETstart}) - (\ref{BETend}) are satisfied with probability at least $1 - \varepsilon$. Let
\begin{align*}
D &:= (D_0 + C)^{1/\beta}\\
C_1 &:= (C_0 + C)^{1/\beta}\\
C_2 &:= C,
\end{align*}
so that (\ref{BETstart}) - (\ref{BETend}) imply (\ref{Disaboundmod}) - (\ref{tauisaboundmod}). 

\end{proof}

The following theorem gives a sufficient condition for the existence of a set $X$ with the equicontinuity property:

\begin{theorem}
\label{theoremcondition}
Suppose that $\AA\implies\RR\butnot\RR_1$ is a finite set, and suppose that $F\implies\C$ is finite with $\AA(F)\implies F$. Suppose further that for all $\ell\in\N$, for all $T\in\AA^\ell$, and for all $p\in\FP_T$,
\begin{itemize}
\item[i)] $p\in\RP_T$ implies $p\in F$
\item[ii)] $p\in F$ implies $T_*(p)<1$, i.e. $p$ is an attracting fixed point of $T$.
\end{itemize}
Fix $\tau < 1$. Then there exist $B$ and $\BBB\Kin\RR$ neighborhoods of $F$ and $\AA$ respectively so that
\begin{itemize}
\item[A)] $X := \C\butnot B$ is closed, connected, contains at least three points, and
\[
\BBB(B)\Kin B
\]
\item[B)] Fix $C_1 < \infty$ and $\alpha > 0$. Then there exist $M < \infty$ and $\gamma$ a modulus of continuity such that if $(T_j)_{j\in\N}$ is a sequence of rational maps in $\BBB$ and $(\phi_j)_{j\in\N}$ is a sequence of potential functions and if for all $j\in\N$,
\begin{align} \label{C1isabound}
\|\phi_j\|_\AL &\leq C_1\\ \label{tauisabound}
e^{\sup(\phi_j)} &\leq \tau\inf(L_j[\one]),
\end{align}
then for all $n\in\N$ \textup{(\ref{Misaboundzero})} and \textup{(\ref{gammaisaboundzero})} hold.
\end{itemize}
\end{theorem}

\begin{remark}
\label{remarkdeterministic}
If $T\in\RR\butnot\RR_1$, then the hypotheses of Theorem \ref{theoremcondition} are satisfied with $\AA = \{T\}$ and $F$ the set of forward images of periodic ramification points of $T$.
\end{remark}

\begin{corollary}
\label{corollarysufficient}
Suppose that $\AA$, $F$, and $\tau$ are as in Theorem \ref{theoremcondition}, and let $B$ and $\BBB$ be given by Theorem \ref{theoremcondition}. If $(T,\Omega,\basemeasure,\theta)$ is any holomorphic random dynamical system on $\C$ such that $T_*[\basemeasure](\BBB) = 1$, and if $\phi:\Omega\times\C\rightarrow\R$ is a random potential function on $\Omega$ satisfying \textup{(\ref{C1isabound})} and \textup{(\ref{tauisabound})} almost surely for some fixed $C_1$ and $\alpha$, then $X := \C\butnot B$ has the equicontinuity property. Thus $(\alpha,\beta:=1,\Omega,\basemeasure,\theta,T,\phi,X)$ satisfies the hypotheses of Theorem \ref{maintheorem}.
\end{corollary}
\begin{proof}
The only claim which requires proof is the fact that $(T,\Omega,\basemeasure,\theta)$ is nonsingular. However, this follows by Proposition \ref{propositionnonsingular} as $T_*[\basemeasure]$ is supported on $\BBB$ which is relatively compact in $\RR$.
\end{proof}

\begin{comment}% Superceded by Proposition \ref{propositionnonsingular}.
First we give an example of a family of nonsingular RDSs which do not fall into any of the three categories of ...
\begin{example}
Suppose that $\sigma$ is a probability measure on $\complexplane$ with finite expectation. Define $Q:\complexplane\rightarrow\RR$ by $Q(c) = (z\mapsto z^2 + c)$. Let $\mu = Q_*[\sigma]$. Let $\Omega$, $\basemeasure$, $\theta$, and $T$ be as in Example \ref{examplepower}. Then $(T,\Omega,\basemeasure,\theta)$ is nonsingular.
\end{example}
\begin{proof}
Without loss of generality assume $\sigma\neq \delta_0$, so that $\SS_0 = \{\infty\}$ almost surely.

Let $C = \int |c|\d\sigma(c) < \infty$. By Tchebychev's inequality, for all $n\in\N$
\[
\sigma(c: |c|\geq C 2^{n + 2})\leq 2^{-(n + 2)}.
\]
Thus
\[
\pr(\exists n\in\N\text{ such that }|c_n|\geq C 2^{n + 2})\leq 1/2
\]
In particular, with positive probability $|c_n| < C 2^{n + 2}$ for all $n\in\N$. Fix $\omega\in\Omega$ for which this is the case.

For all $z\in\complexplane$, if $|z| \geq 2C + 2$, then by induction $|T_0^n(z)| \geq 2C + 2^{n+1}$. (This bound is not optimal.) Thus $T_0^n\tendston\infty$ uniformly on $\C\butnot B_e(0,2C + 2)$. It follows that $\C\butnot B_e(0,2C + 2)\implies \FF_0$; in particular $\SS_0\implies \FF_0$.

Since $\mu^\Z$ is ergodic, it follows that $\SS_0\implies \FF_0$ almost certainly.
\end{proof}
\end{comment}

Next, we show that the assumptions of Theorem \ref{theoremcondition} are reasonable, despite the fact that it is difficult to determine whether they are satisfied for any specific set of rational maps:
\begin{remark}
\label{remarksufficient}
For every $m\in\N$, the set of all sequences $(T_i)_{i=1}^m$ in $\RR\butnot\RR_1$ such that hypotheses (i) and (ii) of Theorem \ref{theoremcondition} are satisfied with $\AA = \{T_i:i=1,\ldots,m\}$ and $F = \emptyset$ is generic (i.e. comeager) in $(\RR\butnot\RR_1)^m$.
\end{remark}
\begin{proof}
It suffices to show that for all $n\in\N$ and for every finite sequence $(i_j)_{j=0}^{n-1}$ with $i_j=1,\ldots,m$ for all $j=0,\ldots,n-1$, the set
\[
\AA_{\vec{i}} := \{(T_i)_{i=1}^m: T = T_{i_{n-1}}\circ\cdots\circ T_{i_0}\text{ has no fixed ramification points}\}
\]
is open dense. In fact, we show that for every finite sequence $(d_i)_{i=1}^m$, $d_i\geq 2$, the set
\[
\AA_{\vec{i},\vec{d}} := \AA_{\vec{i}}\cap \prod_{i=1}^m\RR_{d_i}
\]
is a nonempty Zariski open subset of the analytic variety $\prod_{i=1}^m\RR_{d_i}$. It is a well-known fact that every nonempty Zariski open subset of an irreducible analytic variety is open and dense in the usual topology. (This follows from the multidimensional identity principle.)%(It is easiest to see this in the case of the affine plane, where a Zariski open set is the nonvanishing set of some multivariate polynomial. If a multivariate polynomial vanishes in a neighborhood of some point, then its derivatives of all orders must vanish there as well. But derivatives can be used to recover the coefficients of the polynomial. Thus in this case the polynomial must be identically zero, contradicting that its nonvanishing set is nonempty.)

To see that $\AA_{\vec{i},\vec{d}}$ is Zariski open, note that for each $d\geq 2$,
\begin{align*}
&\{T\in\RR_d: T\text{ has no fixed ramification points}\}\\
&= \{f/g:f,g\in \poly_d,\Res_{d,d}(f,g)\neq 0,\Res_{2d - 2,d + 1}(f'g - fg',f - \id \cdot g) \neq 0\}
\end{align*}
is Zariski open. ($\poly_d$ is the set of all polynomials of degree at most $d$. $\Res_{k,\ell}$ is the resultant whose domain is $\poly_k\times\poly_\ell$.) Since the map
\[
(T_i)_{i=1}^m\mapsto T_{i_{n-1}}\circ\ldots\circ T_{i_0}
\]
is an algebraic morphism, it is continuous in the Zariski topology. Thus $\AA_{\vec{i},\vec{d}}$ is Zariski open.

We wish to show that $\AA_{\vec{i},\vec{d}}$ is nonempty. Fix $a\in\complexplane$ transcendental. For each $d\geq 2$ we define
\[
S_d(z) := \frac{a z^d + 1}{a z^d - 1}
\]
We claim that $(S_{d_i})_{i=1}^m\in\AA_{\vec{i},\vec{d}}$. For ease of notation we write $T_j := S_{d_{i_j}}$, $j = 0,\ldots,n-1$, so that it suffices to show that $T := T_0^n$ has no fixed ramification points.

By contradiction, suppose that $p\in\C$ is a fixed ramification point of $T$. Then there exists $j=0,\ldots,n-1$ such that $T_0^j(p)$ is a ramification point of $T_j$. By cycling the indices, we may without loss of generality assume that $j = 0$. In this case $p$ is a ramification point of $T_0$ i.e. $p$ is zero or infinity.

We claim by induction that for $j=1,\ldots,n$, $T_0^j(p)$ can be expressed as $\pm 1$ times the quotient of two monic polynomials in $a$ whose degrees are equal. In particular, $p = T_0^n(p) \neq 0,\infty$, a contradiction.

Base case $j=1$: $T_0(0) = -1/1$, $T(\infty) = 1/1$.

Inductive step: If
\[
T_0^j(p) = \frac{\pm a^k + \ldots}{a^k + \ldots}
\]
then
\begin{align*}
T_0^{j + 1}(p) &= \frac{a(\pm a^k + \ldots)^{d_j} + (a^k + \ldots)^{d_j}}{a(\pm a^k + \ldots)^{d_j} - (a^k + \ldots)^{d_j}}\\
&= \frac{(\pm)^{d_j}a^{k d_j + 1} + \ldots}{(\pm)^{d_j}a^{k d_j + 1} + \ldots}.
\end{align*}
\end{proof}
\begin{remark}
It does not seem obvious how to prove Remark \ref{remarksufficient} without using the machinery of algebraic geometry. One would like to be able to wiggle the maps a bit (say by post-composing with a M\"obius transformation close to the identity), to ``shake off'' any given fixed ramification point. This works for a single map, or more generally if the sequence $(i_j)_{j = 0}^{n - 1}$ has an element which occurs exactly once. However, it is hard to account for the ``double effect'' of perturbation (in particular, to make sure it does not cancel itself out) in the case where each element occurs at least twice.
\end{remark}
The same is true if polynomials are considered instead of rational functions:
\begin{remark}
\label{remarksufficientpolynomial}
For every $m\in\N$, the set of all sequences $(T_i)_{i=1}^m$ in $\poly\butnot\poly_1$ such that hypotheses (i) and (ii) of Theorem \ref{theoremcondition} are satisfied with $\AA = \{T_i:i=1,\ldots,m\}$ and $F = \{\infty\}$ is generic in $(\poly\butnot\poly_1)^m$.
\end{remark}
\begin{proof}
Almost the exact same as the proof of Remark \ref{remarksufficient}, except that we define
\[
S_d(z) := a z^d + 1.
\]
Details are left to the reader.
\end{proof}
\end{section}

\begin{section}{Consequences of Theorem \ref{maintheorem}}\label{sectionconsequences}
In this section, we fix a nonsingular holomorphic random dynamical system $(T,\Omega,\basemeasure,\theta)$ on $\C$ with a potential function $\phi:\Omega\rightarrow\CC(\C)$ and a set $X\implies\C$ satisfying the hypotheses of Theorem \ref{maintheorem}. We assume that Theorem \ref{maintheorem} has already been proven. (The proof of Theorem \ref{maintheorem} is given in Sections \ref{sectionpreliminaries} - \ref{sectionconvergence}.)

All measures are assumed to be nonnegative.

\begin{remark}
\label{remarkmeasures}
Events \ref{mainevent} and \ref{eventexactUJ} imply the following event:
\begin{quote}
\begin{event}
\label{eventmeasures}
For every $g\in\CC(\C)$ with $g>0$, there exists a unique measure $\nu_g$ whose support is $\JJ_0'$ such that for every $f\in\CC(\C)$ and for every $\kappa > 0$,
\begin{equation}
\label{equationmu}
\frac{L_0^n[f]}{L_0^n[g]}\tendston \mathbf{1}\int f\d \nu_g
\end{equation}
uniformly on $(X\butnot B_s(\SS_n,\kappa)\cup\JJ_n)_n$. Furthermore, if $(\sigma_n)_n$ is any sequence of probability measures such that $\sigma_n$ is supported on $X\butnot B_s(\SS_n,\kappa)\cup\JJ_n$ for all $n\in\N$, then the convergence
\begin{equation}
\label{weakstar}
\nu_g = \lim_{n\rightarrow\infty}L_n^0\left[\frac{\sigma_n}{L_0^n[g]}\right]
\end{equation}
holds in the weak-* topology.
\end{event}
\end{quote}
\end{remark}
%\percent\begin{comment}
\begin{proof}
Fix $\omega\in\Omega$ such that Event \ref{mainevent} is satisfied.

Fix $g\in\CC(\C)$ with $g > 0$. Suppose that $f\in\CC(\C)$. Let $\gamma_1 = \rho_{f/g}$ and let $\gamma_2 = \rho_{\ln(g)}$. Since Event \ref{mainevent} holds, for all $\varepsilon,\kappa > 0$ there exists $N\in\N$ such that for all $n\geq N$, (\ref{mainequation}) holds. But this exactly means that
\begin{equation}
\label{mainequationrephrased}
\left\|\frac{L_0^n[f]}{L_0^n[g]}\right\|_{\osc,X\butnot B_s(\SS_n,\kappa)\cup \JJ_n} \tendston 0.
\end{equation}
for all $\kappa > 0$.

Now (\ref{suplessthan}), (\ref{inflessthan}), and (\ref{mainequationrephrased}) together with the backward invariance of the Julia set imply that
\[
\left(\left[\inf_{\JJ_n}\frac{L_0^n[f]}{L_0^n[g]},\sup_{\JJ_n}\frac{L_0^n[f]}{L_0^n[g]}\right]\right)_{n\in\N}
\]
is a decreasing sequence of intervals whose diameters tend to zero. Let $\nu_g[f]$ be their unique point of intersection. It is easily verified that $\nu_g$ is a positive linear functional on $\CC(\C)$, so by the Riesz representation theorem we may identify it with a measure. Clearly, (\ref{mainequationrephrased}) implies (\ref{equationmu}).

If $f\in\CC(\C)$ satisfies $f\on \JJ_0' = 0$, then the left hand side of (\ref{equationmu}) is identically zero on $\JJ_n'$. Since the convergence of (\ref{equationmu}) is uniform on the sequence $(\JJ_n')_n$, we have $\int f\d\nu_g = 0$. Since this is true for all $f$ such that $f\on \JJ_0' = 0$, we have $\nu_g(\C\butnot \JJ_0') = 0$ i.e. $\Supp(\nu_g)\implies \JJ_0'$. For the other direction, note that Event \ref{eventexactUJ} implies that for every $U\implies\C$ open with $U\cap\JJ_0'\neq\emptyset$, for sufficiently large $n\in\N$, $L_0^n[\one_U]$ is strictly positive on $\JJ_n'$; since $L_0^n[\one_U]$ is lower semicontinuous, it follows that there exists $\varepsilon > 0$ such that $L_0^n[\one_U] \geq \varepsilon L_0^n[g]$ on $\JJ_0'$. Then $\nu_g(U) \geq \varepsilon \int g\d\nu_g = \varepsilon > 0$.

Finally, since the convergence (\ref{equationmu}) holds uniformly on $(X\butnot B_s(\SS_n,\kappa)\cup\JJ_n)_n$, the same convergence holds when integrated against a sequence of measures $(\sigma_n)_n$ as in the hypothesis. Algebra, followed by the implicitization of the $f$ variable, yields (\ref{weakstar}).
\end{proof}
%\percent\end{comment}

For the remainder of this section, we assume that $X$ has the equicontinuity property.

\begin{corollary}
\label{corollaryeigenvalues}
The following event is almost certain to occur:
\begin{quote}
\begin{event}
\label{eventeigenvalues}
Fix $p_0\in X\butnot \SS_0$. Then the sequence
\begin{equation}
\label{equationG}
\left[\ln\left(\frac{L_{-n}^0[\one]}{L_{-n}^0[\one](p_0)}\right)\right]_{n\in\N}
\end{equation}
is uniformly Cauchy on $X$.
\end{event}
\end{quote}
\end{corollary}
\begin{proof}
Let $\gamma_1$ be the modulus of continuity corresponding to the fact that $X$ has the equicontinuity property, and let $\gamma_2 = 0$. For each $k\in\N$, let $\varepsilon_k = 2^{-k} > 0$, and let $0 < \kappa_k < \diam(X)/4$ be small enough so that $\gamma_1(4\kappa_k)\leq 2^{-k}$. For each $n\in\N$, consider the following event:
\begin{quote}
\begin{event}
\label{eventinternal}
For all $f,g\in\CC(\C)$ with $g > 0$, $\rho_{f/g}^{(X)} \leq \gamma_1$, and $\rho_{\ln(g)}^{(X)}\leq \gamma_2$, we have that (\ref{mainequation}) holds. (Note that (\ref{mainequation}) depends on $\varepsilon_k$ and $\kappa_k$, and thus indirectly on $k$.)
\end{event}
\end{quote}
Now Theorem \ref{maintheorem} implies that
\[
\pr(\text{Event \ref{eventinternal} is satisfied for all $n$ sufficiently large}) = 1.
\]
Thus by Lemma \ref{lemmaflip},
\[
\pr(\exists n_k\in\N\text{ such that Event \ref{eventinternal} is satisfied for $\theta^{-n_k}\omega$}) = 1.
\]
We should say a word about the measurability of Event \ref{eventinternal}, since this is necessary to apply Lemma \ref{lemmaflip}. Let
\[
\K = \{(f,g)\in\CC(X)\times\CC(X):g > 0,\rho_{f/g} \leq \gamma_1,\rho_{\ln(g)} \leq \gamma_2,f(p_0) = 0,g(p_0) = 1\}
\]
By the Arzel\`a-Ascoli theorem, $\K$ is a compact metric space under the supremum norm. It will readily be verified that Event \ref{eventinternal} is equivalent to:
\begin{quote}
\begin{event}
\label{eventinternalstandard}
(\ref{mainequation}) holds for all $(f,g)\in\K$.
\end{event}
\end{quote}
which is measurable by Corollary \ref{corollaryconventions}.% see appendix

Fix $\omega\in\Omega$ satisfying the following events:
\begin{enumerate}[A)]
\item $\rho_{\ln(L_{-n}^{-j}[\one])}^{(X)} \leq \gamma_1$ for all $j,n\in\N$, $j\leq n$
\item For all $k\in\N$, there exists $n_k\in\N$ such that  Event \ref{eventinternal} is satisfied for $\theta^{-n_k}\omega$
\end{enumerate}
By (\ref{gammaisabound}) and by the above calculation, such sequences form an almost certain event. Thus, if we show that $\omega$ satisfies Event \ref{eventeigenvalues}, then we are done.

Fix $k\in\N$, and fix $j\geq n_k$. We write $n := n_k$.

Let $f = L_{-j}^{-n}[\one]$, and let $g = \one\sup_X(f)$. Now
\begin{equation*}
\rho_{f/g}^{(X)} \leq 
\rho_{\ln(f/g)}^{(X)}\sup_X(f/g)
= \rho_{\ln(f)}^{(X)}
\leq \gamma_1.
\end{equation*}
Clearly $g > 0$, and $\rho_{\ln(g)}^{(X)} = 0 = \gamma_2$. Thus Event \ref{eventinternal} applies. We write $K := X\butnot B_s(\SS_0,\kappa_k)\cup\JJ_0$,  so that (\ref{mainequation}) becomes
\[
\left\|\frac{L_{-n}^0[f]}{L_{-n}^0[g]}\right\|_{\osc,K} \leq 2^{-k}.
\]
Now
\begin{align*}
\|\ln(L_{-j}^0[\one]) - \ln(L_{-n}^0[\one])\|_{\osc,K}
&= \|\ln(L_{-n}^0[f]) - \ln(L_{-n}^0[g])\|_{\osc,K}\\
&= \left\|\ln\left(\frac{L_{-n}^0[f]}{L_{-n}^0[g]}\right)\right\|_{\osc,K}\\
&\leq \sup_K\left(\frac{L_{-n}^0[g]}{L_{-n}^0[f]}\right) \left\|\frac{L_{-n}^0[f]}{L_{-n}^0[g]}\right\|_{\osc,K}\\
&\leq 2^{-k}\sup(g/f) \leq e^M 2^{-k}
\end{align*}
We claim that $K$ is close to $X$ in the Hausdorff metric:
\begin{claim}
\[
X\implies \cl{B_s}(K,4\kappa_k).
\]
\end{claim}
\begin{proof}
Fix $p\in X$. Since $X$ is connected and has diameter at least $4\kappa_k$, the sets $S_s(p,2\kappa_k)$ and $S_s(p,4\kappa_k)$ must intersect $X$, say at $q$ and $r$. Since $\diam_3\{p,q,r\}=2\kappa_k$ and since $\#(\SS_0)\leq 2$, the pigeonhole principle implies that either $p$, $q$, or $r$ is not in $B_s(\SS_0,\kappa_k)$. But then this point is in $K$, and $p\in\cl{B_s}(K,4\kappa_k)$.
\QEDmod\end{proof}
As a result of this claim, we have the bound
\begin{align*}
\|\ln(L_{-j}^0[\one]) - \ln(L_{-n}^0[\one])\|_{\osc,X}
&\leq \|\ln(L_{-j}^0[\one]) - \ln(L_{-n}^0[\one])\|_{\osc,K} + 4\gamma_1(4\kappa_k)\\
&\leq (4 + e^M) 2^{-k}.
\end{align*}
Since this is true for all $j\geq n_k$, we see that $\|\ln(L_{-j_1}^0[\one]) - \ln(L_{-j_2}^0[\one])\|_{\osc,X}$ tends to zero as $j_1$ and $j_2$ approach infinity jointly. Thus if $p_0\in X$, then
\[
\left\|\ln\left(\frac{L_{-j_1}^0[\one]}{L_{-j_1}^0[\one](p_0)}\right) - \ln\left(\frac{L_{-j_2}^0[\one]}{L_{-j_2}^0[\one](p_0)}\right)\right\|_{\infty,X} \xrightarrow[j_1,j_2]{} 0.
\]
(The function whose $\infty,X$ norm is being taken vanishes at $p_0\in X$.) Thus we are done.
\end{proof}

Without loss of generality suppose that for all $\omega\in\Omega$, the following events are satisfied:
\begin{enumerate}[A)]
\item Events \ref{eventmeasures} and \ref{eventeigenvalues}
\item $T(B)\Kin B$
\item $(T_n)_{n\in\N}$ is non-quasilinear and nonsingular
\end{enumerate}
Note that (B) implies that $\JJ_\omega\implies X$ for all $\omega\in\Omega$, since $\#(X)\geq 3$.

The limit of the sequence (\ref{equationG}) depends on $p_0\in X\butnot\SS_0$. Since $\JJ$ is strongly measurable and always nonempty, by the selection theorem [\cite{Mo} Theorem 2.13, p.32] we may choose a (measurable) random point $p_0\in\JJ_0\implies X\butnot\SS_0$.

Fix $\omega\in\Omega$. The backwards invariance of $X$ implies that (\ref{Ldef}) defines a family of maps $L_m^n:\CC(X)\rightarrow\CC(X)$. We do not distinguish notationally from this family and from the original family $L_m^n:\CC(\C)\rightarrow\CC(\C)$ defined in Section \ref{sectionPF}. We make the following definitions, whose validity is justified by Events \ref{eventmeasures} and \ref{eventeigenvalues}:
\begin{align} \label{gdef}
g_0 &:= \lim_{n\rightarrow\infty}\frac{L_{-n}^0[\one]}{L_{-n}^0[\one](p_0)}
\in\CC(X)\hspace{.5 in}g_0 > 0
\\ \label{lambdadef}
\lambda_0^n &:= L_0^n[g_0](p_n)
> 0
\\
\nu_0 &:= \lim_{n\rightarrow\infty}L_n^0\left[\frac{\delta_{p_n}}{L_0^n[g_0]}\right]
\in\M(\JJ_0')
\end{align}
Recall that in Section \ref{sectionintroduction} we made the convention that $O_n = O_{\theta^n\omega} = O_0\on_{\theta^n\omega}$ for any random object $O_0$.
We continue:
\begin{align} \label{nudef}
\mu_0 &:= g_0 \nu_0
\in\M(\JJ_0')
\\
\psi_0^n &:= \phi_0^n + \ln(g_0) - \ln(g_n\circ T_0^n) - \ln(\lambda_0^n)
\in\CC(T_n^0(X))
\\ \label{Lpsidef}
\L_0^n[f](p) &:= \frac{L_0^n[f g_0](p)}{\lambda_0^n g_n(p)} = \sum_{x\in T_n^0(p)}e^{\psi_0^n(x)}f(x)
:\CC(X)\rightarrow\CC(X)
\\ \label{Lpsistardef}
\L_n^0[\sigma] &:= g_0 L_n^0\left[\frac{\sigma}{\lambda_0^n g_n}\right] = \int \sum_{x\in T_n^0(p)}e^{\psi_0^n(x)}\delta_x\d\sigma(p)
:\M(X)\rightarrow\M(X)
\end{align}
We make the following observations, whose proofs are algebraic in nature and are left to the reader:
\begin{align} \label{eigenvalues}
L_0^n[g_0] &= \lambda_0^n g_n\\ \label{measures}
L_n^0[\nu_n] &= \lambda_0^n\nu_0\\ \label{normalized}
\nu_0[g_0] &= 1\\ \label{equationmunew}
\frac{L_0^n[f]}{\lambda_0^n g_n} &\tendston \mathbf{1}\int f\d \nu_0 \hspace{.5 in}[\text{on }X\butnot B_s(\SS_n,\kappa)\cup\JJ_n,f\in\CC(X)]\\ \label{invariant}
\L_0[\one] &= \one\\ \label{measuresnu}
\L_n^0[\mu_n] &= \mu_0\\ \label{equationnu}
\L_0^n[f] &\tendston \mathbf{1}\int f\d \mu_0 \hspace{.5 in}[\text{on }X\butnot B_s(\SS_n,\kappa)\cup\JJ_n,f\in\CC(X)]\\ \label{probabilitymeasure}
\mu_0[\one] &= 1\\ \label{rewriteend}
(T_0^n)_*[\mu_0] &= \mu_n
\end{align}
The last two formulas imply that $(\mu_n)_n$ is a $(T_n)_n$-invariant sequence of probability measures.

\begin{remark}
\label{remarkweakstar}
Fix $\kappa > 0$. If $(\sigma_n)_n$ is any sequence of probability measures such that
\begin{equation*}
\sigma_n(X\butnot B_s(\SS_n,\kappa)\cup\JJ_n)\tendston 1,
\end{equation*}
then the convergence
\begin{equation}
\label{weakstarnew}
\mu_0 = \lim_{n\rightarrow\infty}\L_n^0[\sigma_n]
\end{equation}
holds in the weak-* topology.
\end{remark}
\begin{proof}
This follows directly from (\ref{weakstar}), plust the fact that $\L$ is a probability-preserving operator.
\end{proof}

\begin{remark}
The measurability of (\ref{gdef}) - (\ref{Lpsistardef}) follows directly from Theorem \ref{theoremexpressions}. Of crucial importance here is the measurability of the random point $p_0$.% see appendix
\end{remark}

\begin{remark}
The expressions
\begin{equation*}
\ln\sup_X(g_0), \ln\inf_X(g_0), \ln(\lambda_0),\ln(\nu_0[\one])
\end{equation*}
are bounded deterministically (independent of $\omega$.) Thus the expression $\sup_X(\psi_0)$ has finite expectation.
\end{remark}
\begin{proof}
This follows directly from (\ref{gammaisabound}) and (\ref{tauisaboundintegral}), together with the fact that $\ln(g_0)$ vanishes at $p_0$.
\end{proof}

\begin{claim}
\label{claimMexact}
With the above assumptions and constructions, the following event is satisfied:
\begin{quote}
\begin{event}
\label{eventMexact}
The sequence $(\mu_n)_n$ is metrically exact i.e. for all $A\implies \C$ Borel measurable we have either $\mu_0(A) = 0$ or $\mu_n(T_0^n(A)) \tendston 1$.
\end{event}
\end{quote}
\end{claim}
%\percent\begin{comment}
\begin{proof}
Fix $A\implies \C$, and let
\[
\wtilde{A} := \bigcup_{n\in\N}T_n^0 T_0^n A
\]
so that
\begin{equation*}
\mu_0(A)
\leq \mu_0(\wtilde{A})
= \lim_{n\rightarrow\infty}\mu_0(T_n^0 T_0^n A)
= \lim_{n\rightarrow\infty}\mu_n(T_0^n(A))
\end{equation*}
Thus it suffices to show that $\mu_0(\wtilde{A})$ is either zero or one.

Suppose that $\mu_0(\wtilde{A}) > 0$. Then
\[
\left(\frac{\one_{T_0^n \wtilde{A}}\mu_n}{\mu_0(\wtilde{A})}\right)_{n\in\N}
\]
is a sequence of probability measures supported on $(\JJ_n)_n$. Thus by Remark \ref{remarkweakstar},
\begin{align*}
\mu_0 &= \lim_{n\rightarrow\infty}\L_n^0\left[\frac{\one_{T_0^n \wtilde{A}}\mu_n}{\mu_0(\wtilde{A})}\right]\\
&= \frac{1}{\mu_0(\wtilde{A})}\lim_{n\rightarrow\infty}\left(\one_{T_0^n\wtilde{A}}\circ T_0^n\right)\left(\L_n^0[\mu_n]\right)\\
&= \frac{1}{\mu_0(\wtilde{A})}\lim_{n\rightarrow\infty}\one_{\wtilde{A}}\mu_0
= \frac{\one_{\wtilde{A}}\mu_0}{\mu_0(\wtilde{A})}.
\end{align*}
Evaluating at $\wtilde{A}$, we see that $\mu_0(\wtilde{A}) = 1$.
\end{proof}
%\percent\end{comment}

\begin{proposition}
\label{propositionatomless}
\[
\pr(\mu_0\text{ is atomless}) = 1.
\]
\end{proposition}
%\percent\begin{comment}
\begin{proof}
Let
\[
a_0 = \max_{p\in\C}\mu_0(p),
\]
so that we are trying to show $\pr(a_0 = 0) = 1$. Fix $\omega\in\Omega$. For all $p\in\C$ we have
\[
\mu_0(p)\leq \mu_1(T_0(p))\leq a_1;
\]
taking the supremum over $p\in\C$ yields $a_0\leq a_1$. Since $(\Omega,\basemeasure,\theta)$ is ergodic, there exists $C\in [0,1]$ such that $\pr(a_0 = C) = 1$. Thus, we are done if $C = 0$.

Fix $\omega\in\Omega$ so that $a_n = C$ for all $n\in\N$.

Fix $p\in\C$ with $\mu_0(p) = C$. By contradiction suppose $C > 0$. By Claim \ref{claimMexact}, we have
\[
C = a_n\geq \mu_n(T_0^n(p)) \tendston 1.
\]
Thus $C = 1$, and $\mu_0 = \delta_p$. It follows that
\[
p\in\Supp(\mu_0)\implies J_0\implies\C\butnot\SS_0.
\]
Let $m = 2$, and let $\ell\in\N$ be given by Lemma \ref{lemmasix}. Then
\[
\#(T_\ell^0 T_0^\ell(p))\geq 2;
\]
in particular, this set is not equal to $\{p\}$. But (\ref{measuresnu}) gives that
\[
\{p\} = \Supp(\mu_0) = T_\ell^0(\Supp(\mu_\ell)) = T_\ell^0 T_0^\ell(p),
\]
a contradiction. Thus we are done.
\end{proof}
%\percent\end{comment}

\begin{proposition}
\label{propositionequivalent}
The random objects
$\lambda_0 > 0$, $g_0 \in \CC(X)$, and $\nu_0\in\M(X)$ are well-defined up to equivalence, where $(\lambda,g,\nu)\sim(\wtilde{\lambda},\wtilde{g},\wtilde{\nu})$ if and only if $(\lambda,g,\nu)$ and $(\wtilde{\lambda},\wtilde{g},\wtilde{\nu})$ are related almost surely by the change of variables
\begin{align} \label{lambdabetatwoequals}
\wtilde{\lambda_0} &= \frac{k_0}{k_1} \lambda_0 \\ \label{gbetatwoequals}
\wtilde{g_0} &= k_0 g_0\\ \label{mubetatwoequals}
\wtilde{\nu_0} &= \nu_0/k_0
\end{align}
where $k_0 > 0$ is (measurable) random. $\mu_0$, $\psi_0$, and $\L_0$ are well-defined up to a set of measure zero. In particular \textup{(\ref{lambdabetatwoequals})} says that $\ln(\lambda)$ is well-defined up to cohomological equivalence.

More precisely, any triple $(\wtilde{\lambda},\wtilde{g},\wtilde{\nu})$ satisfies \textup{(\ref{eigenvalues}) - (\ref{normalized})} if and only if there exists $k_0 > 0$ random satisfying \textup{(\ref{lambdabetatwoequals})} - \textup{(\ref{mubetatwoequals})}. In this case, if $\wtilde{\mu}$, $\wtilde{\psi}$, and $\wtilde{\L}$ are defined by \textup{(\ref{nudef}) - (\ref{Lpsidef})}, then $\wtilde{\mu} = \mu$ almost everywhere, and so on.
\end{proposition}
%\begin{comment}
\begin{proof}
For the forward direction, define the random variable
\[
k_0 := \int \wtilde{g_0}\d\nu_0.% see appendix
\]
Clearly $k_0 > 0$; we wish to show (\ref{lambdabetatwoequals}) - (\ref{mubetatwoequals}). Fix $\omega\in\Omega$. Integrating (\ref{measures}) [no tildes] against $\wtilde{g_0}$, using (\ref{eigenvalues}) [with tildes] to simplify, setting $n = 1$, and dividing by $k_1$ yields (\ref{lambdabetatwoequals}).

Fix $\varepsilon,\kappa > 0$. Consider the random set $K_0 := X\butnot B_s(\SS_0,\kappa)$. For each $\omega\in\Omega$, (\ref{equationmu}) implies that the equation
\begin{equation}
\label{flipped}
\left\|\frac{L_0^n[\wtilde{g_0}]}{L_0^n[g_0]} - k_0\right\|_{\infty,K_n} \leq \varepsilon k_0
\end{equation}
holds for sufficiently large $n$.
Fix $\omega\in\Omega$ and assume that there exists $n\in\N$ such that \ref{flipped} holds for $\theta^{-n}\omega$. By Lemma \ref{lemmaflip}, % see appendix
this assumption is almost certainly valid. Rearranging the reindexed (\ref{flipped}) yields
\begin{align*}
\left\|\frac{k_{-n} \wtilde{g_0}}{k_0 g_0} - k_{-n}\right\|_{\infty,K_0} &\leq \varepsilon k_{-n}\\
\left\|\frac{\wtilde{g_0}}{k_0 g_0} - 1\right\|_{\infty,K_0} &\leq \varepsilon.
\end{align*}
Since $\varepsilon,\kappa > 0$ were arbitrary, the above equation almost certainly holds for all $\varepsilon,\kappa > 0$. But this implies (\ref{gbetatwoequals}).

Now, (\ref{normalized}) [with tildes] implies that $(\wtilde{g_n}\wtilde{\nu_n})_n$ is a sequence of probability measures supported on the Julia set. (\ref{weakstar}), (\ref{gbetatwoequals}), (\ref{eigenvalues}) [no tildes], (\ref{lambdabetatwoequals}), and (\ref{measures}) [with tildes] yield
\[
\nu_0 = \lim_{n\rightarrow\infty}L_n^0\left[\frac{\wtilde{g_n}\wtilde{\nu_n}}{L_0^n[g_0]}\right]
= \lim_{n\rightarrow\infty}L_n^0\left[\frac{k_n g_n\wtilde{\nu_n}}{\lambda_0^n g_n}\right]
= k_0\lim_{n\rightarrow\infty}L_n^0[\wtilde{\nu_n}]
= k_0 \wtilde{\nu_n}.
\]
The backwards direction and the claims made about $\wtilde{\mu}$, $\wtilde{\psi}$, and $\wtilde{\L}$ are purely algebraic and are left to the reader.
\end{proof}
%\end{comment}

We now step back and take a more global view by considering the relative dynamical system $(\Omega,\basemeasure,\theta,\X,\T,\pi_1)$ associated to $(T,\Omega,\basemeasure,\theta)$ (see Definition \ref{definitionrandomaction}). Let $\CC_f(\X)$ be the set of all measurable fiberwise continuous functions on $\X$. We define global objects
\begin{align*}
%L[f](\omega,p) &:= L_{-1}(\omega)[f(\theta^{-1}\omega,\cdot)](p)\\
\lambda&:\Omega\rightarrow\R
&\lambda(\omega) := \lambda_0\on_\omega\\
\phi&:\X\rightarrow\R
&\phi(\omega,p) := \phi_0(p)\on_\omega\\
%g(\omega,p) &:= g_\omega(p)\\
%\nu &:= \int \delta_\omega\times\nu_\omega\d\basemeasure(\omega)\\
\mu&\in\M(\JJ)
&\mu := \int \delta_\omega\times\mu_0\on_\omega\d\basemeasure(\omega)\\
\psi&:\T^{-1}(\Omega\times X)\rightarrow\R
&\psi(\omega,p) := \psi_0(p)\on_\omega\\
\L&:\CC_f(\Omega\times X)\rightarrow\CC_f(\Omega\times X)
&\L[f](\omega,p) := \L_{-1}\on_\omega[f(\theta^{-1}\omega,\cdot)](p)\\
\L^*&:\M(\Omega\times X)\rightarrow\M(\Omega\times X)
&\L^*[\sigma] := \int \sum_{x\in T_{\theta^{-1}\omega}^{-1}(p)}e^{\psi(\theta^{-1}\omega,x)}\delta_{(\theta^{-1}\omega,x)}\d\sigma(\omega,p)
\end{align*}
and so (\ref{invariant}) - (\ref{rewriteend}) yield global formulas
\begin{align*}
%g(\omega,p) &= \lambda_{\theta^{-1}\omega} L[g](\omega,p)\\
%\int L[f]\d\nu &= \int \lambda f\d\nu\\
\L[\one] &= \one\\
\L^*[\mu] &= \mu\\
%\L^n[f] &\tendston \mathbf{1}\int f\d \mu\\
\mu[\one] &= 1\\
\T_*[\mu] &= \mu
\end{align*}
In particular, $\mu$ is a $\T$-invariant probability measure on $\X$ with $\pi_*[\mu] = \basemeasure$.

\end{section}
\begin{section}{Ergodic theory of holomorphic random dynamical systems} \label{sectionergodic}
For this section, see \cite{Bo}, although the notation differs considerably.

Suppose that $(\Omega,\basemeasure,\theta,\X,\T,\pi)$ is a relative dynamical system. Let $\M(\X,\T,\basemeasure)$ be the set of all $\T$-invariant probability measures $\sigma$ on $\X$ such that $\pi_*[\sigma] = \basemeasure$, and let $\M_e(\X,\T,\basemeasure)$ be the set of all ergodic elements of $\M(\X,\T,\basemeasure)$. We have the following definition:

\begin{definition} If $\sigma\in\M(\X,\T,\basemeasure)$, the \emph{relative entropy} of $\T$ over $\theta$ with respect to $\sigma$ is defined by the equations
\begin{align*}
h_\sigma(\T\on\theta) &:= \sup_\Par h_\sigma(\T\on\theta;\Par)\\
h_\sigma(\T\on\theta;\Par) &:= \lim_{n\rightarrow\infty}\frac{1}{n}H_\sigma\left(\bigvee_{j = 0}^{n - 1}\T^{-j}\Par\on\pi^{-1}\epsilon_\Omega\right)
\end{align*}
(The supremum is taken over all partitions $\Par$ of $\X$ such that $H_\sigma(\Par\on\pi^{-1}\epsilon_\Omega) < \infty$. $\epsilon_\Omega$ and $\epsilon_\X$ are the partitions into points of $\Omega$ and $\X$, respectively.) For proof of the existence of the limit see [\cite{Bo} Theorem 2.2, p.102].
\end{definition}

\begin{proposition}
\label{propositionkolmogorovsinai}
Fix $\sigma\in\M(\X,\T,\basemeasure)$. Suppose that there exists a partition $\Par$ of $\X$ such that
\begin{enumerate}[A)]
\item $\Par$ has finite $\sigma$-entropy over $\Omega$ i.e.
\begin{equation}
\label{finiteentropy}
H_\sigma(\Par\on\pi^{-1}\epsilon_\Omega) < \infty.
\end{equation}
\item $\Par$ $\sigma$-almost generates $\X$ over $\Omega$ i.e.
\begin{equation}
\label{partitionsequal}
\bigvee_{j\in\N} \T^{-j}\Par \vee \pi^{-1}\epsilon_\Omega \equiv_\sigma \epsilon_\X,
\end{equation}
where ``$\Par_1\equiv_\sigma\Par_2$'' means that there exists a set $A\implies\X$ with $\sigma(\X\butnot A) = 0$ such that $\Par_1\on A = \Par_2\on A$.
\end{enumerate}
Then the following equations hold:
\begin{equation}
\label{kolmogorovsinai}
h_\sigma(\T\on\theta)
= h_\sigma(\T\on\theta;\Par)
= H_\sigma(\epsilon_\X\on\T^{-1}\epsilon_\X)
\end{equation}
Furthermore, if $(\sigma_{\omega,p})_{(\omega,p)\in\X}$ is the Rohlin decomposition of $\sigma$ relative to $\T^{-1}(\points_\X)$ i.e.
\begin{align*}
\sigma &= \int \sigma_{\omega,p}\d\sigma(\omega,p)\\
\T_*[\sigma_{\omega,p}] &= \delta_{\omega,p},
\end{align*}
then
\begin{equation}
\label{hequalsH}
h_\sigma(\T\on\theta) = \int H_{\sigma_{\omega,p}}(\points_\X)\d\sigma(\omega,p).
\end{equation}
\end{proposition}
\begin{proof}
(\ref{kolmogorovsinai}) is a straightforward but tedious generalization of the deterministic case [\cite{PU} Theorem 1.9.7 p.60]. A proof of the first equality is furthermore given in [\cite{Ki} Lemma 1.5 p.45]. The equality
\[
H_\sigma(\epsilon_\X\on\T^{-1}\epsilon_\X) = \int H_{\sigma_{\omega,p}}(\points_\X)\d\sigma(\omega,p)
\]
just follows from the definition of conditional entropy; c.f. [\cite{PU} Definition 1.8.3 p.54].
\end{proof}
\begin{remark}
Note that the right hand side of (\ref{kolmogorovsinai}) does not depend on $\theta$. The key here is the hypothesis that there exists a partition which is generating and of finite entropy \emph{relative to $\pi^{-1}\epsilon_\Omega$.} This condition is in a sense an indicator that $\Omega$ is the ``correct quotient space to look at''. For example, if $h_\basemeasure(\theta) > 0$, then the conditions ``there exists a partition generating and of finite entropy relative to $\pi^{-1}\epsilon_\Omega$'' and ``there exists a partition generating and of finite entropy in the absolute sense'' are incompatible.
\end{remark}

Now, we assume that $(\Omega,\basemeasure,\theta,\X,\T,\pi)$ is the relative dynamical system associated to a random dynamical system $(T,\Omega,\basemeasure,\theta)$ on a compact metric space $X$.

\begin{definition}
\label{definitionpressure}
Suppose that $\phi:\Omega\rightarrow\CC(X)$ is a random potential function on $(T,\Omega,\basemeasure,\theta)$. Assume further that $\phi$ satisfies
\begin{equation}
\label{L1}
\EE[\|\phi\|_\infty] < \infty.
\end{equation}
The \emph{relativistic pressure} of $\T$ over $\theta$ with respect to $\phi$ is defined by the equation
\begin{equation*}
P_{\phi,\basemeasure}(\T\on\theta) := \lim_{\varepsilon\rightarrow 0}\limsup_{n\rightarrow\infty}\frac{1}{n}\int\ln\sup_{E\implies X}\left(\sum_{x\in E}e^{\phi_0^n(x)}\right)\d\basemeasure(\omega),
\end{equation*}
where the supremum is taken over all $(\omega,n,\varepsilon)$-separated subsets $E$ of $X$, i.e. all sets $E\implies X$ such that
\[
x,y\in E, \dist(T_0^j(x),T_0^j(y))\leq \varepsilon\all j = 0,\ldots,n - 1 \Rightarrow x = y.
\]
(c.f. [\cite{Bo} Definition 5.4, p.109])
If $\phi = 0$, then $P_{\phi,\basemeasure}(\T\on\theta)$ is called the \emph{relative topological entropy} of $\T$ over $\theta$, and is denoted $h_{\text{top},\basemeasure}(\T\on\theta)$.
\end{definition}

We have the following variational principle:
\begin{theorem}
\label{theoremvarprinciple}
Suppose that $(T,\Omega,\basemeasure,\theta)$ is a random dynamical system on a compact metric space $X$, and suppose that $\phi:\Omega\rightarrow\CC(X)$ is a random potential function satisfying (\ref{L1}). Then
\begin{equation}
\label{varprinciple}
\sup_{\sigma\in \M_e(\X,\T,\basemeasure)}\left(h_\sigma(\T\on\theta) + \int \phi\d\sigma\right) =
\sup_{\sigma\in \M(\X,\T,\basemeasure)}\left(h_\sigma(\T\on\theta) + \int \phi\d\sigma\right) = P_{\phi,\basemeasure}(\T\on\theta).
\end{equation}
\end{theorem}
\begin{proof}~
The first equation is a straightforward but tedious generalization from the deterministic case [\cite{PU} Corollary 2.4.3, p.93].

The proof of the second equation is essentially found in [\cite{Bo} Theorem 6.1 p.110]; however, his setup is slightly different from ours. We include a brief justification that the theorem still holds in our new setting:

First of all, although Bogensch\"utz specifically states his theorem for invertible relative dynamical systems (he defines an RDS as an invertible RDS), he does not use the hypothesis of invertibility anywhere in the proof of the variational principle. There are a couple of points in the proof where Bogensch\"utz's notation suggests that he is using the invertibility, however closer examination reveals that this is not the case.

Secondly, Bogensch\"utz defines relative dynamical entropy as a supremum over only finite partitions, and in fact only over partitions coming from a partition of $X$, i.e. partitions of the form $\pi_2^{-1}(\Par)$ for some finite partition $\Par$ of $X$. In fact, this supremum is the same as our definition of the relative dynamical entropy as the supremum over all partitions of $\X$ which have finite entropy relative to $\pi_1^{-1}\epsilon_\Omega$. The proof of the reduction to finite partitions of $\X$ is the same as in the deterministic case; we do not repeat it here. To prove that it suffices to consider only partitions coming from partitions of $X$, consider any finite partition $\Par$ of $\X = \Omega\times X$. It is a standard exercise to show that for every $\varepsilon > 0$, there exist two partitions $\Par_1$ of $\Omega$ and $\Par_2$ of $X$ so that $H_\sigma(\Par\on\Par_1\times\Par_2) \leq \varepsilon$. From this, it clearly follows that $H_\sigma(\Par\on\pi_2^{-1}\Par_2\vee\pi_1^{-1}\epsilon_\Omega) \leq \varepsilon$. Thus
\[
h_\sigma(\T\on\theta;\Par) \leq h_\sigma(\T\on\theta;\pi_2^{-1}\Par_2) + \varepsilon.
\]
The remainder of the proof should be clear.
\end{proof}

\begin{definition}
If $(\Omega,\basemeasure,\theta,T,X,\phi)$ are as in Theorem \ref{theoremvarprinciple}, then an \emph{equilibrium state} of $(\X,\T,\phi)$ over $(\Omega,\basemeasure,\theta)$ is an element $\sigma\in \M(\X,\T,\basemeasure)$ on which the supremum in (\ref{varprinciple}) is achieved. If $\phi = 0$, an equilibrium state is called a \emph{measure of maximal relative entropy} of $(\X,\T)$ over $(\Omega,\basemeasure,\theta)$.
\end{definition}

As in the deterministic case, if there exists an equilibrium state then there exists an ergodic equilibrium state. Thus any unique equilibrium state is automatically ergodic.

We return to our setting of rational maps, and state our main results concerning ergodic theory of holomorphic random dynamical systems:

\begin{theorem}
\label{theoremmanepartition}
Fix a holomorphic random dynamical system $(T,\Omega,\basemeasure,\theta)$ on $\C$. Assume that
\begin{enumerate}[A)]
\item There exists $D < \infty$ such that $\pr(\deg(T)\leq D) = 1$
\item
\begin{equation}
\label{lnH}
\EE[\ln\sup(T_*)] < \infty.
\end{equation}
\end{enumerate}
Then for every $\sigma\in\M_e(\X,\T,\basemeasure)$ with $h_\sigma(\T\on\theta) > 0$, there exists a (measurable) partition $\Par$ of $\X$ such that the hypotheses of Proposition \ref{propositionkolmogorovsinai} are satisfied.
\end{theorem}

\begin{theorem}
\label{theoremequilibrium}
Fix a nonsingular holomorphic random dynamical system $(T,\Omega,\basemeasure,\theta)$ on $\C$ with a potential function $\phi:\Omega\rightarrow\CC(\C)$ and a set $X\implies\C$ satisfying the hypotheses of Theorem \ref{maintheorem}. Assume further that (\ref{L1}) holds, that $X$ has the equicontinuity property, and that the conclusion of Theorem \ref{theoremmanepartition} is satisfied. Let $\mu\in\M(\X,\T,\basemeasure)$ be as defined in Section \ref{sectionconsequences}. Then $\mu$ is the unique equilibrium state of $(\X,\T,\phi)$ over $(\Omega,\basemeasure,\theta)$. Furthermore, the relativistic pressure of $\phi$ is given by
\begin{equation}
\label{pressureequals}
P_{\phi,\basemeasure}(\T\on\theta) = \int \ln(\lambda)\d\basemeasure.
\end{equation}
\end{theorem}

\begin{corollary}
\label{corollaryhequalslndeg}
Fix an antilinear holomorphic random dynamical system $(T,\Omega,\basemeasure,\theta)$ on $\C$ satisfying the hypotheses of Theorem \ref{theoremmanepartition}. Then the random version of the equation $h_{\text{top}}(T) = \ln(\deg(T))$ holds; more specifically, we have
\[
h_{\text{top},\basemeasure}(\T\on\theta) := P_{0,\basemeasure}(\T\on\theta) = \EE[\ln(\deg(T))].
\]
\end{corollary}
\begin{proof}
Let $\phi = 0$. Clearly, (\ref{Disaboundintegral}) - (\ref{tauisaboundintegral}) are satisfied, and $X = \C$ has the equicontinuity property. Proposition \ref{propositionnonsingular} shows that $(T,\Omega,\basemeasure,\theta)$ is nonsingular (here we are using that the hypotheses of Theorem \ref{theoremmanepartition} are satisfied, not just the conclusion). Thus the hypotheses of Theorem \ref{maintheorem} are satisfied. Since (\ref{L1}) is also clearly satisfied, we have verified the hypotheses of Theorem \ref{theoremequilibrium}.

A simple calculation shows that if $\phi = 0$, then $\lambda_0 = \deg(T_0)$. Now (\ref{pressureequals}) gives the result.
\end{proof}
\end{section}
\begin{section}{Preliminaries from Complex Analysis}\label{sectionpreliminaries}
We begin with the following lemma, which describes the behavior of injective maps from the unit disk to the Riemann sphere. Note that the Koebe distortion theorem does not apply to such maps, as shown by Example \ref{examplekoebe} below.
\begin{lemma}
\label{lemmakoebe}
Suppose that $U$ is a simply connected hyperbolic open set, and suppose that $\zeta:U\rightarrow\C$ is holomorphic and injective. Then $\zeta$ is Lipschitz continuous with a corresponding constant of $\sqrt{\frac{\lambda_s(\zeta(U))}{1 - \lambda_s(\zeta(U))}}$, if the continuity is measured with respect to the hyperbolic metric $h_U$ on $U$, and with respect to the spherical metric $s$ on $\C$ (note normalizations in Section \ref{sectionnotation}).
\end{lemma}
\begin{proof}
By the Riemann mapping theorem, we may without loss of generality suppose that $U = \B$.

It is enough to show that for each $x\in\B$,
\begin{equation}
\label{koebebounds}
\|\zeta_*(x)\|_h^s\leq\sqrt{\frac{\lambda_s(\zeta(\B))}{1 - \lambda_s(\zeta(\B))}}.
\end{equation}
(Recall that according to our conventions, $\|\cdot\|_h^s$ indicates the operator norm from the hyperbolic metric to the spherical metric. $\zeta_*(x)$ indicates the induced map on tangent spaces.)

Without loss of generality we suppose that $x=0$, and that $\zeta(x)=\infty$. Let $\zeta(z)=\sum_{k=-1}^\infty c_k z^k$ be the Laurent series for $\zeta$ in the annulus $0<|z|<1$. (The injectivity of $\zeta$ implies that the pole is simple and unique.) The area theorem [\cite{CG} Theorem 1.1 p.1] gives that
\begin{align*}
\lambda_e(\C\butnot \zeta(\B))
&= -\pi\sum_{k=-1}^\infty k|c_k|^2\\
&\leq \pi|c_{-1}|^2\\
&=\lambda_e(B_e(0,|c_{-1}|))
\end{align*}
Notice that the set $B_e(0,|c_{-1}|)$ is in fact the solution to the optimization problem of maximizing a set's spherical area while holding its Euclidean area fixed. (The formula for spherical area in terms of a Euclidean integral implies that it is optimal for the mass to be as close to the origin as possible.) Since $B_e(0,|c_{-1}|)$ has maximal spherical area among sets with fixed Euclidean area, it also has maximal spherical area among sets whose Euclidean area is less than or equal to its own. Thus we have
\begin{align*}
\lambda_s(\C\butnot \zeta(\B))
&\leq \lambda_s(B_e(0,|c_{-1}|))\\
&= \int_{r=0}^{|c_{-1}|}\int_{\theta=0}^{2\pi}\frac{r\d r\d \theta}{\pi(1 + r^2)^2}\\
&= 1 - \frac{1}{1 + |c_{-1}|^2}
\end{align*}
or
\[
\lambda_s(\zeta(\B)) \geq \frac{1}{1 + |c_{-1}|^2}
\]
Solving for $|c_{-1}|$ yields
\[
|c_{-1}| \geq \sqrt{\frac{1 - \lambda_s(\zeta(\B))}{\lambda_s(\zeta(\B))}}
\]
Finally, we note that $\|\zeta_*(0)\|_h^s=\frac{1}{|c_{-1}|}$, so taking the reciprocal of both sides yields (\ref{koebebounds}).
\QEDmod\end{proof}

\begin{example}
\label{examplekoebe}
The family of maps $(z\mapsto cz)_{c\in\R^+}$ shows that the bound (\ref{koebebounds}) is sharp. Furthermore, as $c$ tends to infinity, this family become a counterexample to any distortion claim similar to that of the Koebe theorem.
\end{example}
\begin{proof}
In fact, direct calculation shows that if $\zeta(z)=cz$, then $\|\zeta_*(0)\|_h^s = c$ and $\lambda_s(\zeta(\B)) = \frac{1}{1 + c^2}$, yielding that (\ref{koebebounds}) is sharp.

As $c$ goes to infinity, the derivative at zero goes to infinity, but since the map is injective, the change of variables formula implies that the derivative remains uniformly square-integrable, and thus cannot tend uniformly to infinity on any set of positive measure. Thus there exist sequences $(\zeta_n)_n$ and $(x_n)_n$ such that $x_n\tendston 0$ and $\displaystyle{\frac{\|(\zeta_n)_*(0)\|_h^s}{\|(\zeta_n)_*(x_n)\|_h^s}}\tendston\infty$; in other words the distortion is unbounded.
\QEDmod\end{proof}

\begin{questions}
Is Lemma \ref{lemmakoebe} true if $\zeta$ is a map from a hyperbolic Riemann surface which is not simply connected? Alternatively, does the lemma hold if we drop the requirement of injectivity, but in the conclusion $\lambda_s(\zeta(\B))$ is counted with multiplicity? See also Lemma \ref{lemmakoebetwo} below.
\end{questions}

The next lemma extends Lemma \ref{lemmakoebe} to the case where $\zeta$ is only locally injective. We first need a definition:

\begin{definition}
\label{definitionlocallyinjective}
Fix $\sigma>0$. A map $\zeta:U\rightarrow\C$ is \emph{$\sigma$-locally injective} if for all $x,y\in U$ with $\dist_U(x,y)\leq\sigma$ and $\zeta(x) = \zeta(y)$, we have $x = y$. $\zeta$ is \emph{uniformly locally injective} if there exists $\sigma>0$ such that $\zeta$ is $\sigma$-locally injective.
\end{definition}

\begin{lemma}
\label{lemmakoebetwo}
Fix $\sigma > 0$. Suppose that $U$ is a simply connected hyperbolic open set, and suppose that $\zeta:U\rightarrow\C$ is holomorphic and $\sigma$-locally injective. Then $\zeta$ is Lipschitz continuous with a corresponding constant of $\coth(\sigma/2)\sqrt{\frac{\lambda_s(\zeta(U))}{1 - \lambda_s(\zeta(U))}}$, if the continuity is measured with respect to the hyperbolic metric $h_U$ on $U$, and with respect to the spherical metric $s$ on $\C$.
\end{lemma}
Note that Lemma \ref{lemmakoebe} is a special case, achieved by letting $\sigma = \infty$.
\begin{proof}
As in the proof of Lemma \ref{lemmakoebe}, we may without loss of generality suppose that $U = \B$. Thus we need to show that for all $x\in\B$,
\begin{equation}
\label{koebeboundstwo}
\|\zeta_*(x)\|_h^s\leq\coth(\sigma/2)\sqrt{\frac{\lambda_s(\zeta(\B))}{1 - \lambda_s(\zeta(\B))}},
\end{equation}
Again, without loss of generality we suppose that $x=0$.

The hyperbolic diameter of $B := B_e(0,\tanh(\sigma/2))$ is $\sigma$. Thus $\zeta \on B$ is injective, since for all $x,y\in B\implies\B$ such that $\zeta(x) = \zeta(y)$, we have $\dist_h(x,y)\leq\sigma$ and thus $x=y$ since $\zeta$ is $\sigma$-locally injective. Thus Lemma \ref{lemmakoebe} applies, and
\begin{equation*}
\|\zeta_*(0)\|_{h_B}^s \leq \sqrt{\frac{\lambda_s(\zeta(B))}{1 - \lambda_s(\zeta(B))}} \leq \sqrt{\frac{\lambda_s(\zeta(\B))}{1 - \lambda_s(\zeta(\B))}}.
\end{equation*}
(Recall that $h_B$ denotes the hyperbolic metric of $B$.) A simple calculation shows that $h(0) = \tanh(\sigma/2)h_B(0)$. The result then follows by composition.
\end{proof}

We shall furthermore need the following facts, which we will state without proof:
\begin{lemma}
\label{lemmacomplex}
Fix $H < \infty$ and $\delta_2 > 0$. Then there exists $\delta_1 > 0$ so that if $T$ is a rational map with $\sup(T_*) \leq H$ (and thus $\deg(T) = \int T_*^2\d\lambda_s \leq H^2$), then $T$ has the following property: For all $p,q\in\C$ with $\dist_s(p,q)\leq \delta_1$, there exists a bijection $\Phi:T^{-1}(p)\rightarrow T^{-1}(q)$ (recall that these are multisets according to our conventions) such that
\begin{equation}
\label{closetoid}
\dist_s(x,\Phi(x)) \leq \delta_2
\end{equation}
for all $x\in T^{-1}(p)$.
\end{lemma}

\begin{lemma}
\label{lemmamorecomplex}
Suppose that $T$ is a rational map. Fix $\delta_2 > 0$. Then there exists a neighborhoood $\BBB$ of $T$ such that for all $S\in\BBB$ and for all $p\in\C$, there exist bijections
\begin{align*}
\Phi_p &: T^{-1}(x)\rightarrow S^{-1}(x)\\
\Phi_\RP &: \RP_T\rightarrow \RP_S\\
\Phi_\BP &: \BP_T\rightarrow \BP_S\\
\Phi_\FP &: \FP_T\rightarrow \FP_S
\end{align*}
each with the property that \textup{(\ref{closetoid})} holds for all $x$ in the appropriate domain.
\end{lemma}

\begin{lemma}
\label{lemmamiranda}
Suppose $x\in\C$ and $T\in\RR$. Then there exist arbitrarily small neighborhoods $B_x$ of $x$ such that
\[
T\on B_x:B_x\rightarrow T(B_x)
\]
is proper of degree $k := \mult_T(x)$. In particular, $T(\del B_x)\cap T(B_x) = \emptyset$.
\end{lemma}
\begin{proof}
Without loss of generality we may assume $x = 0$ and $T(z) = z^k$ on a small ball [\cite{Mir}, Proposition 4.1, p.44].
\end{proof}

\end{section}

\begin{section}{Inverse Branch Formalism}\label{sectionformalism}
In this section we systematize a technique found for example in [\cite{DU} Lemma 4, p.108]. We will use this technique repeatedly. The idea is to count the number of inverse branches of the map $T_j^n$ on a ``nice'' set $U\implies \C$ which are ``good'' (meaning the area of the image is small, to apply Koebe Distortion). By choosing the right numbers, we can make it so there is a large number of ``good'' inverse branches, implying that for any $p\in U$, the summation $L_j^n[f](p)=\sum_{x\in T_n^j(p)}\exp(\phi_j^n(x))f(x)$ has a large contribution from those terms where $x$ is given by a ``good'' inverse branch. These terms are not affected much as $p$ varies, thus the oscillation of $L_j^n[f]$ is slightly less than that of $f$.

We do this construction in greater generality than is done in \cite{DU}. The first generalization, from the autonomous case to the non-autonomous case, is not very significant. The same arguments still work. The second generalization is more important, and in fact corrects an error found in \cite{DU}. See Lemma \ref{lemmaerrorfix}.

Suppose that $U$ is a simply connected hyperbolic open set and that $T$ is a rational map. If $V$ is a connected component of $T^{-1}(U)$ containing no critical points, then there exists a unique map $T_V^{-1}:U\rightarrow V$ such that $T\circ T_V^{-1}=\id$. However, we can still think of $T_V^{-1}$ as being a lift of the identity map under $T$. It turns out that this construction is not sufficiently general. Instead of considering lifts of the identity map, we will instead consider lifts of arbitrary locally injective maps whose domains are simply connected hyperbolic open sets. Each such map can be thought of as a ``set which overlaps with itself''. For example, let $U = B_e(1,.8)$ and let $\zeta:U\rightarrow\C$ be the 4th power map $\zeta(z) := P_4(z) := z^4$. $\zeta$ is not injective since $\zeta(\frac{\sqrt{2}}{2}(1 + i)) = \zeta(\frac{\sqrt{2}}{2}(1 - i)) = -1$. We can think of $\zeta$ as encoding a multiset $V := \zeta(U)$ for which $\mult_V(-1) = 2$. Since $V$ is not simply connected, there is no way of proceeding with the inverse branch formalism based on the $V$ alone; the map $\zeta$ gives necessary information. On the other hand, the domain $U$ plays no particular role. Using the Riemann mapping theorem, we could in fact require that $U = \B$, but there seems to be no reason to do this.

Suppose $\ZZ$ is a finite collection of maps such that each $\zeta\in\ZZ$ is a holomorphic map from some set $U_\zeta\implies\C$ to $\C$. We define the \emph{multiplicity} of $\ZZ$ by
\[
\mult(\ZZ) := \mult\left(\bigcup_{\zeta\in\ZZ}\zeta(U_\zeta)\right) := \max_{p\in\C}\sum_{\zeta\in\ZZ}\#(\zeta^{-1}(p)).
\]

Our first step is to define the concept of an inverse branch in this context, and to define which branches are the ``good'' ones:

\begin{definition}
Suppose that $U$ is a simply connected hyperbolic open set, and suppose $\zeta:U\rightarrow\C$ is holomorphic and locally injective (we shall often take $\zeta = \id$), and suppose $T$ is a rational map. A holomorphic map $\eta:U\rightarrow\C$ is an \emph{inverse branch} or \emph{lift} of $\zeta$ under $T$ if $T\circ\eta = \zeta$. If $c \leq 1$, then $\eta$ is \emph{$c$-good} if $\lambda_s(\eta(U))\leq c$. We denote the collection of inverse branches of $\zeta$ under $T$ by $\I(\zeta,T)$, and the collection of $c$-good inverse branches of $\zeta$ under $T$ by $\I(\zeta,T,c)$. If $\ZZ$ is a collection of holomorphic maps from $U$ to $\C$, then
\begin{align*}
\I(\ZZ,T) &:= \bigcup_{\zeta\in\ZZ}\I(\zeta,T)\\
\I(\ZZ,T,c) &:= \bigcup_{\zeta\in\ZZ}\I(\zeta,T,c)
\end{align*}
\end{definition}
Next, we give an upper bound for the number of preimages of a point which are not given by good inverse branches:
\begin{proposition}
\label{propositionmrecursion}
For each $i=1,\ldots,m$, suppose that $\zeta_i:U_i\rightarrow\C$ is holomorphic and locally injective, suppose that the collection $\{\zeta_1,\ldots,\zeta_m\}$ has multiplicity $r$, and suppose that $T$ is a rational map. Then
\begin{equation}
\label{multrbounds}
\mult\left(\bigcup_{i=1}^m\I(\zeta_i,T)\right)\leq r.
\end{equation}
Fix $c > 0$. Then
\begin{equation}
\label{mrecursion}
\sum_{i=1}^m\max_{x\in U_i}\#\{z\in T^{-1}(\zeta_i(x)):\nexists \eta\in\I(\zeta_i,T,c)\text{ such that }z=\eta(x)\}\leq 2r\deg^2(T) + \frac{r}{c}.
\end{equation}
\end{proposition}
\begin{proof}
We begin by showing (\ref{multrbounds}). Fix $z\in\C$. Since $\{\zeta_1,\ldots,\zeta_m\}$ has multiplicity $r$, we have that $\sum_{i=1}^m\#(\zeta_i^{-1}(T(z)))\leq r$. Thus it is enough to show that for each $i=1,\ldots,m$,
\begin{equation}
\label{uniqueness}
\sum_{\eta\in\I(\zeta_i,T)}\#(\eta^{-1}(z))\leq \#(\zeta_i^{-1}(T(z))).
\end{equation}
For each $\eta\in\I(\zeta_i,T)$ and for each $x\in\eta^{-1}(z)$, we have $x\in\zeta_i^{-1}(T(z))$. Thus
\[
\pi_2:\bigcup_{\eta\in\I(\zeta_i,T)}\{\eta\}\times\eta^{-1}(z)\rightarrow \zeta_i^{-1}(T(z)).
\]
(Here $\pi_2$ is projection onto the second coordinate.) We will be done if $\pi_2$ is injective.

Suppose that $\pi_2(\eta_1,x_1) = \pi_2(\eta_2,x_2)$, i.e. $x := x_1 = x_2$. Since $\zeta_i$ is locally injective and $\zeta_i = T \circ \eta_1$, $T$ is injective in a neighborhood of $\eta_1(x) = z$. Thus $T$ is invertible in a neighborhood of $z$, so we have $\eta_1 = T^{-1}\circ\zeta_i = \eta_2$ in a neighborhood of $x$. By the identity principle, $\eta_1 = \eta_2$. Thus $\pi_2$ is injective, and we have shown (\ref{multrbounds}).

Next, we show (\ref{mrecursion}). For each $i=1,\ldots,m$ fix $x_i\in U_i$ at which the maximum in (\ref{mrecursion}) is attained. Fix $i=1,\ldots,m$, and suppose that $z\in T^{-1}(\zeta_i(x_i))$. If $\zeta_i(U_i)$ does not contain a branch point of $T$, then by the homotopy lifting principle $\zeta_i$ has a unique inverse branch $\eta_{i,z}:U_i\rightarrow\C$ such that $\eta_{i,z}(x_i) = z$. If furthermore $\eta_{i,z}$ is $c$-good, then $z$ is not counted in (\ref{mrecursion}). Thus for each $i=1,\ldots,m$ and for each $z\in T^{-1}(\zeta_i(x_i))$, exactly one of the following the three possibilities holds:

\begin{itemize}
\item[A)] $\zeta_i(U_i)$ contains a branch point of $T$.
\item[B)] $\zeta_i(U_i)$ does not contain a branch point of $T$, but the inverse branch $\eta_{i,z}$ is not $c$-good i.e. $\lambda_s(\eta_{i,z}(U_i))>c$.
\item[C)] $\zeta_i(U_i)$ does not contain a branch point of $T$, and the inverse branch $\eta_{i,z}$ is $c$-good i.e. $\lambda_s(\eta_{i,z}(U_i))\leq c$.
\end{itemize}

We have already established that category (C) is not counted in (\ref{mrecursion}). Thus to complete the proof, it suffices to show that category (A) represents at most $2r\deg^2(T)$ pairs $(i,z)$ (counting multiplicity), and that category (B) represents at most $\frac{r}{c}$ pairs $(i,z)$ (multiplicity is not needed since every ramification point is in category (A))

\begin{itemize}
\item[A)] It suffices to show that for at most $2r\deg(T)$ values of $i=1,\ldots,m$, $\zeta_i(U_i)$ contains a branch point. By the Riemann-Hurwitz formula there are at most $2\deg(T) - 2$ branch points (exactly that many counting multiplicity), and since $\mult(\zeta_i)_{i=1}^m = r$, each branch point is contained in at most $r$ sets of the form $\zeta_i(U_i)$.
\item[B)] Let
\[
\mathscr{C} := \bigcup_{i=1}^m \{i\}\times(\I(\zeta_i,T)\butnot \I(\zeta_i,T,c)),
\]
so that (\ref{multrbounds}) implies
\[
\sum_{\eta\in\mathscr{C}}\one_{\eta(U_i)}\leq r\one.
\]
Integrating with respect to $\d\lambda_s$ and using the fact that $\lambda_s(\eta(U_i))>c$ for all $\eta\in\mathscr{C}$ to simplify yields $\#(\mathscr{C})\leq\frac{r}{c}$. (Recall that we have normalized $\lambda_s(\C) = 1$.) Now the map $(i,\eta)\mapsto (i,\eta(x_i))$ is a surjection from $\mathscr{C}$ onto category (B). Thus we are done.
\end{itemize}
\end{proof}

\begin{remark}
It is possible to get a bound which is linear in $\deg(T)$ instead of quadratic by using the monodromy theorem rather than the homotopy lifting principle; however this requires more work. A quadratic bound is sufficient for our purposes.
\end{remark}

Next, we formalize the idea, hinted at above, that $L_j^n$ can be split up into a summation over terms which are a result of good inverse branches and those which are not:

\begin{definition}
\label{definitionZAB}
Fix $n\in\Z$ and $m\in\N$. For each $i=1,\ldots,m$, suppose that $\zeta_i:U_i\rightarrow\C$ is holomorphic and locally injective, and suppose that $(T_j)_{j=0}^{n-1}$ is a finite sequence of rational maps. Fix $0<c<1$. Using backwards recursion, we define
\begin{align*}
\ZZ_n^{(i)}&:=\{\zeta_i\}\\
\ZZ_j^{(i)}&:=\I\left(\ZZ_{j+1}^{(i)},T_j,\frac{c^{2(n - j)}}{1 + c^{2(n - j)}}\right)
\end{align*}
Suppose also that $(\phi_j)_{j=0}^{n-1}$ is a finite sequence of potential functions, so that $L_j^n$ is defined for all $j=0,\ldots,n$. For each $j=0,\ldots,n$, we define auxiliary operators $A_j^n,B_j^n:\CC(\C)\rightarrow\bigoplus_{i=1}^m \CC(U_i)$:
\begin{align}
\label{Adef}
(A_j^n[f])_i(x) &:= \sum_{\eta\in\ZZ_j^{(i)}}e^{\phi_j^n(\eta(x))}f(\eta(x))\\
\label{Bdef}
B_j^n[f] &:= A_{j+1}^n[L_j[f]] - A_j^n[f]
\end{align}
These definitions are called the \emph{inverse branch formalism}.
\end{definition}

\begin{remark}
An element $(f_i)_{i=1}^m$ of $\bigoplus_{i=1}^m \CC(U_i)$ can be thought of as a function from $\C$ to $\R$ which is undefined at some points (i.e. $\C\butnot\cup_{i=1}^m\zeta_i(U_i)$) and takes on multiple values at others (if $p\in\C$, then for each $i=1,\ldots,m$ and for each $x\in\zeta_i^{-1}(p)$, the function takes on the value $f_i(x)$ at $p$). If $f\in\CC(\C)$, a natural way to get an element of $\bigoplus_{i=1}^m \CC(U_i)$ is to consider $(f\circ\zeta_i)_{i=1}^m$.
\end{remark}

The idea is that $A_j^n$ is the approximation of $L_j^n$ obtained by summing over the good branches, and that $B_j^n$ is a summation over the branches thrown out in the $(n-j)$th step. This is made explicit in Proposition \ref{propositionZAB} below.

Note: the value $\wtilde{c}$ is chosen merely to simplify calculations with Lemma \ref{lemmakoebe}; the point is just to choose some exponentially decaying quantity depending only on $n-j$.

For the remainder of this section, we will suppose that $n$, $m$, $(U_i)_{i=1}^m$, $(\zeta_i)_{i=1}^m$, $(T_j)_{j=0}^{n-1}$, $(\phi_j)_{j=0}^{n-1}$, and $0<c<1$ are as in Definition \ref{definitionZAB}, and that for each $j=0,\ldots,n$, $(\ZZ_j^{(i)})_{i=1}^m$, $A_j^n$, and $B_j^n$ are given by Definition \ref{definitionZAB}. We will write $r := \mult(\zeta_i)_{i=1}^m$.

\begin{proposition}
\label{propositionZAB}
Fix $f\in\CC(\C)$. For all $k=1,\ldots,n$,
\begin{equation}
\label{Aequals}
A_k^n[f] = \left(L_k^n[f]\circ\zeta_i\right)_{i=1}^m - \sum_{j=k}^{n-1} B_j^n[L_k^j[f]].
\end{equation}
Furthermore, if $K\implies\C$, then for all $j=0,\ldots,n-1$
\begin{equation}
\label{Bislessthan}
\sum_{i=1}^m \sup_{\zeta_i^{-1}(K)}(B_j^n[f])_i\leq r(2\deg^2(T_j) + 1 + c^{-2(n-j)})e^{\sup_{K_j}(\phi_j^n)}\sup_{K_j}(f).
\end{equation}
where $K_j := T_n^j(K)$. If $f\geq 0$, then the left hand side is positive.
\end{proposition}
\begin{proof}
The proof of (\ref{Aequals}) is by backwards induction on $k$. If $k=n$, the identity is trivial. If we suppose that the formula is true for $\wtilde{k} = k+1$ for all $f\in\CC(\C)$, then we can substitute $L_k[f]$ for $f$, yielding
\begin{equation*}
A_{k+1}^n[L_k[f]] = \left(L_k^n[f]\circ\zeta_i\right)_{i=1}^m - \sum_{j=k+1}^{n-1} B_j^n[L_k^j[f]]
\end{equation*}
Subtracting the reverse of (\ref{Bdef}) with $j=k$ yields (\ref{Aequals}). Thus we are done.

We now wish to show (\ref{Bislessthan}). By backwards induction on $j$, we see that Proposition \ref{propositionmrecursion} implies that $\mult(\cup_{i=1}^m\ZZ_{j+1}^{(i)})\leq r$. Fix $j=0,\ldots,n-1$ and $f\in\CC(\C)$.

We are now in a position to apply the second half of Proposition \ref{propositionmrecursion} to the collection $\cup_{i=1}^m \ZZ_{j+1}^{(i)}$ (which is a collection locally injective of holomorphic maps with multiplicity at most $r$), to the map $T_j$, and to the value
\[
\wtilde{c} := \displaystyle{\frac{c^{2(n-j)}}{1 + c^{2(n-j)}}}>0.
\]
With these inputs, (\ref{mrecursion}) becomes
\begin{equation}
\label{mrecursionnew}
\sum_{i=1}^m\sum_{\eta\in\ZZ_{j+1}^{(i)}}\max_{x\in U_i}\#(S_x) \leq 2r\deg^2(T_j) + r(1 + c^{-2(n-j)}),
\end{equation}
where
\[
S_x := \left\{\left.z\in(T_j)^{-1}(\eta(x)) \right| \nexists \xi\in\I\left(\eta,T_j,\wtilde{c}\right)\text{ such that }z = \xi(x)\right\}.
\]
Fix $i = 1,\ldots,m$ and $x\in\zeta_i^{-1}(K)$. Consider (\ref{Bdef}). On both sides we take the $i$th coordinate and evaluate at $x$. We use (\ref{Adef}) and (\ref{Ldef}) to evaluate further. The result is
\begin{equation*}
(B_j^n[f])_i(x)
= \sum_{\eta\in\ZZ_{j+1}^{(i)}}\sum_{z\in(T_j)^{-1}(\eta(x))}e^{\phi_j^n(z)}f(z)
- \sum_{\xi\in\ZZ_j^{(i)}}e^{\phi_j^n(\xi(x))}f(\xi(x))
\end{equation*}
Since $\ZZ_j^{(i)}=\I\left(\ZZ_{j+1}^{(i)},T_j,\wtilde{c}\right)$, we can rewrite the right-hand summation as a double summation; $\eta$ runs over $\ZZ_{j+1}^{(i)}$, and $\xi$ runs over $\I\left(\eta,T_j,\wtilde{c}\right)$. Factoring the common summation $\eta\in\ZZ_{j+1}^{(i)}$ yields
\begin{equation*}
(B_j^n[f])_i(x)
= \sum_{\eta\in\ZZ_{j+1}^{(i)}}\left[\sum_{z\in(T_j)^{-1}(\eta(x))}e^{\phi_j^n(z)}f(z)
- \sum_{\xi\in\I\left(\eta,T_j,\wtilde{c}\right)}e^{\phi_j^n(\xi(x))}f(\xi(x))\right]
\end{equation*}
at which point we notice that every term in the right hand summation appears in the left hand summation as well. Cancelling the common terms yields
\begin{equation*}
(B_j^n[f])_i(x)
= \sum_{\eta\in\ZZ_{j+1}^{(i)}}\sum_{z\in S_x}e^{\phi_j^n(z)}f(z)
\end{equation*}
Bounding the right-hand summand by $\exp(\sup_{K_j}(\phi_j^n))\sup_{K_j}(f)$, we see that
\begin{equation*}
(B_j^n[f])_i(x)
\leq \sum_{\eta\in\ZZ_{j+1}^{(i)}}\#(S_x)e^{\sup_{K_j}(\phi_j^n)}\sup_{K_j}(f)
\end{equation*}
Taking the supremum over $x\in\zeta_i^{-1}(K)$, summing over $i=1,\ldots,m$, and combining with (\ref{mrecursionnew}) yields (\ref{Bislessthan}).
\end{proof}

We shall be interested in (\ref{Aequals}) only in the case where $k=0$.

\begin{corollary}
\label{corollaryZAB}
Suppose that $D,C_2 < \infty$, $\tau < c^2$, and $\beta > 0$ are such that for all $j=0,\ldots,n-1$, \textup{(\ref{Disaboundmod})} and \textup{(\ref{tauisaboundmod})} hold. Fix $K\implies \C$, and suppose that for all $j=0,\ldots,n-1$, \textup{(\ref{Misaboundzero})} holds for $X := T_n^j(K)$. Then there exists $C_3 < \infty$ depending only on $r$, $c$, $D$, $C_2$, $\tau$, and $\beta$ such that for all $f\in\CC(\C)$,
\begin{equation}
\label{ZAB}
\sum_{i=1}^m \sup_{\zeta_i^{-1}(K)}[L_0^n[f]\circ\zeta_i - (A_0^n[f])_i] \leq C_3 e^M \inf_K L_0^n[f]
\end{equation}
\end{corollary}
\begin{proof}
\begin{align*}
&\sum_{i=1}^m \sup_{\zeta^{-1}(K)}[L_0^n[f]\circ\zeta_i - (A_0^n[f])_i]\\
&\leq \sum_{j=0}^{n-1} \sum_{i=1}^m \sup_{\zeta^{-1}(K)}(B_j^n[L_0^j[f]])_i\\
&\leq \sum_{j=0}^{n-1} r(2\deg^2(T_j) + 1 + c^{-2(n-j)})e^{\sup(\phi_j^n)}\sup_{T_n^j(K)}(L_0^j[f])\\
&\leq \sum_{j=0}^{n-1} r(2D^2 (n - j)^{2/\beta} + 1 + c^{-2(n-j)})e^{C_2}\tau^{n-j} \inf(L_j^n[\one]) e^M \inf_{T_n^j(K)}(L_0^j[f])\\
&\leq \sum_{j=0}^{n-1} r(2D^2 (n - j)^{2/\beta} + 1 + c^{-2(n-j)})e^{C_2}\tau^{n-j} e^M \inf_K (L_0^n[f])\\
&\leq C_3 e^M \inf_K (L_0^n[f]),
\end{align*}
where
\begin{equation}
\label{C5def}
C_3 := e^{C_2}\sum_{k=1}^\infty r(2D^2 k^{2/\beta} + 1 + c^{-2k})\tau^k;
\end{equation}
$C_3 < \infty$ since $\tau < c^2 < 1$. As promised, $C_3$ depends only on $r$, $c$, $D$, $C_2$, $\tau$, and $\beta$. Thus we are done.
\end{proof}
\begin{lemma}
\label{lemmahyperbolic}
Fix $C_1 < \infty$ and $\alpha,\sigma,\beta > 0$. Suppose that \textup{(\ref{C1isaboundmod})} is satisfied for all $j = 0,\ldots,n - 1$. Fix $i=1,\ldots,m$, and suppose that $\zeta_i$ is $\sigma$-locally injective. Then there exists $C_4 < \infty$ depending only on $C_1$, $\alpha$, $\sigma$, $\beta$, and $c$ such that
\begin{enumerate}[A)]
\item For all $\eta\in \ZZ_0^{(i)}$ and for all $x,y\in U_i$,
\begin{align}\label{distancesmall}
\dist_s(\eta(x),\eta(y)) &\leq \coth(\sigma/2) c^n \dist_{U_i}(x,y)\\ \label{equationlemma}
\phi_0^n(\eta(x)) - \phi_0^n(\eta(y)) &\leq C_4 \dist_{U_i}(x,y)^\alpha
\end{align}
\item For all $g\in\CC(\C)$ with $g > 0$, and for all $\varepsilon > 0$,
\begin{equation}
\label{Aislessthan}
\rho_{\ln(A_0^n[g])_i}(\varepsilon) \leq C_4 \varepsilon^\alpha + \rho_{\ln(g)}(\coth(\sigma/2) c^n \varepsilon)
\end{equation}
\end{enumerate}
\end{lemma}
\begin{proof} Fix $i=1,\ldots,m$ and $x,y\in U_i$. Fix $\eta\in\ZZ_{-n}^{(i)}$.
\begin{itemize}
\item[(\ref{distancesmall}):]  For each $j=0,\ldots,n-1$, $\eta$ is $\frac{c^{2(n-j)}}{1 + c^{2(n-j)}}$-good. By Lemma \ref{lemmakoebetwo}, $\eta$ is Lipschitz continuous with a corresponding constant of $\coth(\sigma/2)c^{n-j}$. Plugging in $j=0$ gives (\ref{distancesmall}).
\item[(\ref{equationlemma}):]
\begin{align*}
\max_{\eta\in \ZZ_0^{(i)}}[\phi_0^n(\eta(x)) - \phi_0^n(\eta(y))]
&\leq \sum_{j=0}^{n-1}\max_{\eta\in \ZZ_j^{(i)}}[\phi_j(\eta(x)) - \phi_j(\eta(y))]\\
&\leq \sum_{j=0}^{n-1} \|\phi_j\|_\AL \left(\max_{\eta\in \ZZ_j^{(i)}}(\dist_s(\eta(x),\eta(y)))\right)^\alpha\\
&\leq \sum_{j=0}^{n-1} C_1 (n-j)^{1/\beta} (\coth(\sigma/2) c^{n-j} \dist_s(x,y))^\alpha\\
&\leq C_4 \dist_s(x,y)^\alpha,
\end{align*}
where
\begin{equation}
\label{C6def}
C_4 := \sum_{k=1}^\infty C_1 k^{1/\beta} (\coth(\sigma/2) c^k)^\alpha < \infty.
\end{equation}
As promised, $C_4$ depends only on $C_1$, $\alpha$, $\sigma$, $\beta$, and $c$.

\item[(\ref{Aislessthan}):] For all $x,y\in U_i$,
\begin{equation*}
\frac{(A_0^n[g])_i(x)}{(A_0^n[g])_i(y)}
\leq \max_{\eta\in\ZZ_0^{(i)}}\frac{e^{\phi_0^n(\eta(x))}g(\eta(x))}{e^{\phi_0^n(\eta(y))}g(\eta(y))}
\leq \exp\left(C_4 \dist_s(x,y)^\alpha + \rho_{\ln(g)}(\coth(\sigma/2) c^n \dist_s(x,y))\right).
\end{equation*}
\end{itemize}
\end{proof}

Unlike the previous propositions, the following proposition will be used only in the proof of Theorem \ref{maintheorem}, and not in the proof of Theorem \ref{theoremcondition}. The idea is to show that when the Perron-Frobenius operator is applied, then the oscillation norm of $f$ computed relative to $g$, as in Lemma \ref{lemmanonincreasing}, actually goes down some of the time, instead of staying constant. All we need is a small negative term beyond what Lemma \ref{lemmanonincreasing} would give us. The idea behind Proposition \ref{propositionhyperbolic} is the same: (\ref{formulahyperbolic}) would be trivial without the second term of the right hand side (note that this term is actually negative).

We include a diagram (Figure \ref{figurehyperbolic}) which we feel is useful for understanding the structure of the remaining proof of Theorem \ref{maintheorem}.

\begin{figure}
\centerline{\mbox{\includegraphics[scale=.5]{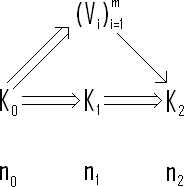}}}
\caption{The double arrows represent that one set covers another with high probability (close to $1$), in the sense that all preimages of the latter set lie in the former. The single arrow represents that one set covers another with low but positive probability, in the weaker sense that every point in the latter set contains at least one preimage in the former. The numbers represent in which universe $X_n$ the objects live. The left hand side of the diagram is analyzed in Proposition \ref{propositionhyperbolic}, and the right hand side is analyzed in Lemma \ref{lemmamultiexact}. Both sides together are analyzed in Proposition \ref{propositiontinysteps}. The diagram is then iterated to prove Theorem \ref{maintheorem}.}
\label{figurehyperbolic}
\end{figure}

\begin{proposition}
\label{propositionhyperbolic}
Fix $C_1,C_2,C_5,C_6,M,D<\infty$, $r\in\N$, $\tau < c^2 < 1$, and $\alpha,\sigma,\beta>0$; we call these ``the parameters''. Fix $m\in\N$ and $K \implies \C$. For each $j = 0,\ldots,n - 1$, assume that \textup{(\ref{Disaboundmod})} - \textup{(\ref{tauisaboundmod})} are satisfied, and that \textup{(\ref{Misaboundzero})} is satisfied for $X := T_n^j(K)$. Next, for each $i=1,\ldots,m$, fix $x_i,y_i\in U$ satisfying
\begin{align}\label{C3isabound}
\dist_h(x_i,y_i) &\leq C_5\\ \label{Kisabound}
\zeta_i(x_i),\zeta_i(y_i) &\in K.
\end{align}
Assume that $\zeta_i$ is $\sigma$-locally injective. Then for each $f,g\in\CC(\C)$ with $g>0$, and such that $\|\ln(g)\|_{\osc,T_n^0(K)}\leq C_6$, we have
\begin{equation}
\label{formulahyperbolic}
\sum_{i=1}^m\left[\frac{L_0^n[f](\zeta_i(x_i))}{L_0^n[g](\zeta_i(x_i))} - \frac{L_0^n[f](\zeta_i(y_i))}{L_0^n[g](\zeta_i(y_i))}\right]
\leq m\|f/g\|_{\osc,T_n^0(K)} + \varepsilon_m\left(\rho_{f/g}^{(T_n^0(K))}(C_5 c^n) - \|f/g\|_{\osc,T_n^0(K)}\right)
\end{equation}
where $\varepsilon_m\in\R$ depends only on $m$ and the parameters (in particular it does not depend on $n$ or on the sequences of rational maps and potential functions) in such a way that if the parameters are fixed, then there exists $m\in\N$ such that $\varepsilon_m > 0$.
\end{proposition}
\begin{proof}
For each $j=0,\ldots,n$, let $K_j = T_n^j(K)$.

For convenience, we now make the following notational conventions:
\begin{itemize}
\item By a subscript of $\IETA$, we mean that the sum, maximum, or minimum is to be taken over all $i=1,\ldots,m$ and over all $\eta\in\ZZ_0^{(i)}$. We illustrate this notation with two facts which we will use in the sequel:
\begin{itemize}
\item $\eta(x_i),\eta(y_i)\in K_0$ for all $\IETA$
\item (\ref{C3isabound}), (\ref{distancesmall}) and (\ref{equationlemma}) imply that
\begin{align} \label{distancesmalltwo}
\displaystyle{\max_\IETA}(\dist_s(\eta(x_i),\eta(y_i)))\leq C_5 c^n\\ \label{equationlemmatwo}
\displaystyle{\max_\IETA}(\phi_0^n(\eta(x_i)) - \phi_0^n(\eta(y_i))) \leq C_4 C_5^\alpha
\end{align}
\end{itemize}
\item We define
\[
\Psi_0^n(z):=\displaystyle{\frac{e^{\phi_0^n(z)}g(z)}{L_0^n[g](T_0^n(z))}}.
\]
Simple calculation demonstrates the following identities:
\begin{align*}
\frac{L_0^n[f](x)}{L_0^n[g](x)}
&=\sum_{z\in T_n^0(x)}\Psi_0^n(z)\frac{f(z)}{g(z)}\\
\sum_{z\in T_n^0(x)}\Psi_0^n(z)&=1
\end{align*}
\end{itemize}

We now begin calculation. Note that
\begin{align*}
\frac{L_0^n[f](x)}{L_0^n[g](x)}
&=\sum_{z\in T_n^0(x)}\Psi_0^n(z)\frac{f(z)}{g(z)}\\
&=\sup_{K_0}(f/g) + \sum_{z\in T_n^0(x)}\Psi_0^n(z)\left[\frac{f(z)}{g(z)} - \sup_{K_0}(f/g)\right]\\
&=\inf_{K_0}(f/g) + \sum_{z\in T_n^0(x)}\Psi_0^n(z)\left[\frac{f(z)}{g(z)} - \inf_{K_0}(f/g)\right]
\end{align*}
Plugging in $x = \zeta_i(x_i)$ and $x = \zeta_i(y_i)$, subtracting, and summing over all $i=1,\ldots,m$ yields
\begin{align*}
&\sum_{i=1}^m\left[\frac{L_0^n[f](\zeta_i(x_i))}{L_0^n[g](\zeta_i(x_i))} - \frac{L_0^n[f](\zeta_i(y_i))}{L_0^n[g](\zeta_i(y_i))}\right]\\
&=m\left\|\frac{f}{g}\right\|_{\osc,K_0} + \sum_{i=1}^m\left[\sum_{z\in T_n^0(\zeta_i(x_i))}\Psi_0^n(z)\left[\frac{f(z)}{g(z)} - \sup_{K_0}\left(\frac{f}{g}\right)\right] - \sum_{w\in T_n^0(\zeta_i(y_i))}\Psi_0^n(w)\left[\frac{f(w)}{g(w)} - \inf_{K_0}\left(\frac{f}{g}\right)\right]\right]\\
&\leq m\left\|\frac{f}{g}\right\|_{\osc,K_0} + \sum_\IETA\left[\min\left(\Psi_0^n(\eta(x_i)),\Psi_0^n(\eta(y_i))\right)\left(\frac{f(\eta(x_i))}{g(\eta(x_i))} - \sup_{K_0}\left(\frac{f}{g}\right) - \frac{f(\eta(y_i))}{g(\eta(y_i))} + \inf_{K_0}\left(\frac{f}{g}\right)\right)\right]\\
&\leq m\left\|\frac{f}{g}\right\|_{\osc,K_0} + \left[\sum_\IETA\min\left(\Psi_0^n(\eta(x_i)),\Psi_0^n(\eta(y_i))\right)\right]\left(\rho_{f/g}^{(K_0)}\left(\max_\IETA\dist_s(\eta(x_i),\eta(y_i))\right) - \left\|\frac{f}{g}\right\|_{\osc,K_0}\right)\\
\end{align*}
which together with (\ref{distancesmalltwo}) implies (\ref{formulahyperbolic}) as long as $\varepsilon_m$ is defined in such a way as to be a lower bound for
\[
\sum_\IETA\min\left(\Psi_0^n(\eta(x_i)),\Psi_0^n(\eta(y_i))\right).
\]
Thus, we aim at finding such a lower bound; for all $\IETA$,

\begin{align*}
\Psi_0^n(\eta(x_i))
&\geq \frac{\inf_{K_0}(g)}{\sup_{K_0}(g)}\frac{1}{\sup_{K_n}(L_0^n[\one])}e^{\phi_0^n(\eta(x_i))}\\
\Psi_0^n(\eta(y_i))
&\geq \frac{\inf_{K_0}(g)}{\sup_{K_0}(g)}\frac{1}{\sup_{K_n}(L_0^n[\one])}e^{\phi_0^n(\eta(y_i))}
\end{align*}
We take the minimum of the two equations, and sum over all $\IETA$:

\begin{equation}
\label{whereweleftoff}
\sum_\IETA\min\left(\Psi_0^n(\eta(x_i)),\Psi_0^n(\eta(y_i))\right)
\geq \frac{e^{-C_6}}{\sup_{K_n}(L_0^n[\one])}\sum_\IETA \min\left(e^{\phi_0^n(\eta(x_i))},e^{\phi_0^n(\eta(y_i))}\right)
\end{equation}
Raising $e$ to both sides of (\ref{equationlemmatwo}), solving for $e^{\phi_0^n(\eta(y_i))}$, and taking the minimum with the inequality $e^{\phi_0^n(\eta(x_i))} \geq e^{-C_4 C_5^\alpha} e^{\phi_0^n(\eta(x_i))}$ yields that for all $\IETA$,
\[
\min\left(e^{\phi_0^n(\eta(x_i))},e^{\phi_0^n(\eta(y_i))}\right)
\geq e^{-C_4 C_5^\alpha}e^{\phi_0^n(\eta(x_i))}.
\]
Combining with (\ref{whereweleftoff}) and using (\ref{Adef}) to simplify,
\begin{equation}
\label{leavingoffagain}
\sum_\IETA\min\left(\Psi_0^n(\eta(x_i)),\Psi_0^n(\eta(y_i))\right)
\geq \frac{e^{-(C_6 + C_4 C_5^\alpha)}}{\sup_{K_n}(L_0^n[\one])}\sum_{i=1}^m(A_0^n[\one])_i(x_i).
\end{equation}
By Corollary \ref{corollaryZAB},
\begin{align*}
\sum_{i=1}^m(A_0^n[\one])_i(x_i)
&\geq \sum_{i=1}^m L_0^n[\one](\zeta_i(x_i)) - e^M C_3 \inf_{K_n}(L_0^n[\one])\\
&\geq (m - e^M C_3)\inf_{K_n}(L_0^n[\one])
\end{align*}
Combining with (\ref{leavingoffagain}) and (\ref{Misaboundzero}),
\[
\sum_\IETA\min\left(\Psi_0^n(\eta(x_i)),\Psi_0^n(\eta(y_i))\right)
\geq e^{-(C_6 + C_4 C_5^\alpha + M)}\left(m -  e^M C_3\right)
\]
The right hand side we call $\varepsilon_m$; as promised, it depends only on $m$ and the parameters. To finish the proof, suppose that the parameters are fixed; we wish to find $m\in\N$ so that $\varepsilon_m>0$. Let $m=\lceil e^M C_3\rceil + 1$; we have $\varepsilon_m \geq e^{-(C_6 + C_4 C_5^\alpha + M)}> 0$.
\end{proof}
\end{section}

\begin{section}{Proof of Theorem \ref{maintheorem}}\label{sectionconvergence}
In this section, we fix $\alpha,\beta > 0$, a holomorphic random dynamical system $(T,\Omega,\basemeasure,\theta)$ on $\C$ with a potential function $\phi:\Omega\rightarrow\CC(\C)$, and $X\implies\C$ which satisfy the hypotheses of Theorem \ref{maintheorem}, i.e. $(T,\Omega,\basemeasure,\theta)$ is nonsingualar, $X$ has the bounded distortion property, and (\ref{Disaboundintegral}) - (\ref{tauisaboundintegral}) are satisfied.

The idea is to use Theorem \ref{theoremexact} in combination with Corollary \ref{corollarysixquantitative} to show that with positive probability, there exist disjoint sets $U_i$ with relatively compact subsets whose iterates cover all of $X\butnot B_s(\SS_0,\kappa)$ within a bounded number of steps. Since there are infinitely many chances, it must happen sometime. The contribution in oscillation from these particular inverse images is slightly lower than expected, so the oscillation as a whole must go down.

\begin{lemma}
\label{lemmamultiexact}
Fix $C_5,\varepsilon_1>0$. Then there exists $\kappa > 0$ such that
\begin{equation}
\label{kappaadmissible}
\pr(B_s(\SS_0,\kappa)\Kin \FF_0) \geq 3/4,
\end{equation}
and such that for each $m\in\N$ there exists $N\in\N$ such that the following event occurs with probability at least $1 - \varepsilon_1$:
\begin{quote}
\begin{event}
\label{eventmultiexact}
There exist a disjoint collection of open disks $(U_i)_{i=1}^m$ and relatively compact subsets $V_i\Kin U_i$ such that
\begin{align}\label{multiexactstart}
\diam_{U_i}(V_i)&\leq C_5\\ \label{multiexactmiddle}
T_0^N(V_i)&\supseteq \C\butnot B_s(\SS_N,\kappa)\\ \label{multiexactend}
T_0^N(B_s(\SS_0,\kappa))&\implies B_s(\SS_N,\kappa)
\end{align}
\end{event}
\end{quote}
\end{lemma}
\begin{proof}
Since $(T,\Omega,\basemeasure,\theta)$ is nonsingular, there exists $\kappa > 0$ such that
\begin{equation}
\label{K4containsjulia}
\pr\left(\dist_s(\JJ_0,\SS_0) > \kappa\right) \geq \max(1 - \varepsilon_1/5,3/4).% see appendix
\end{equation}
Of course, this implies (\ref{kappaadmissible}).

Suppose $m\in\N$. Corollary \ref{corollarysixquantitative} guarantees the existence of $\ell\in\N$ and $\delta_3 > 0$ such that

\begin{equation}
\label{probabilitydeltaone}
\pr\left(\diam_m(T_\ell^0(p))\geq\delta_3\all p\in \C\butnot B_s(\SS_\ell,\kappa)\right) \geq 1 - \varepsilon_1/5.
\end{equation}
Let $\delta_2>0$ be such that $\diam_{B_s(0,\delta_3/2)}(B_s(0,\delta_2))=C_5$. An exact formula can be given for $\delta_2$ in terms of $\delta_3$ and $C_5$, but it is irrelevant here.

Let $H < \infty$ be large enough so that 
\begin{equation}
\label{czero}
\pr\left(\sup(T_0^\ell)_*\leq H\right) \geq 1 - \varepsilon_1/5.% see appendix
\end{equation}
Let $\delta_1 > 0$ be given by Lemma \ref{lemmacomplex}. By Theorem \ref{theoremexact}, for all sufficiently large $n\in\N$
\begin{equation}
\label{equationexample}
\pr\left(\exists p\in \JJ_0\text{ such that }T_0^n(\cl{B_s}(p,\delta_1))\supseteq \C\butnot B_s(\SS_n,\kappa)\right) \geq 1 - \varepsilon_1/5.% see appendix
\end{equation}
Without loss of generality, suppose that $n$ is also large enough so that
\begin{equation}
\label{probabilitylast}
\pr\left(T_0^{\ell + n}(\cl{B_s}(\SS_0,\kappa))\implies \cl{B_s}(\SS_{\ell + n},\kappa)\right) \geq 1 - \varepsilon_1/5% see appendix
\end{equation}
this is valid by Lemma \ref{lemmafatoucompact}.

Let $N = \ell + n$.

Now, fix $\omega\in\Omega$ such that:
\begin{enumerate}[A)]
\item $\dist_s(\JJ_\ell,\SS_\ell) > \kappa$ \label{B}
\item $T_0^N(B_s(\SS_0,\kappa))\implies B_s(\SS_N,\kappa)$ \label{C}
\item $\diam_m(T_\ell^0(p))\geq\delta_3\all p\in \C\butnot B_s(\SS_\ell,\kappa)$ \label{D}
\item $\sup(T_0^\ell)_* \leq H$ \label{E}
\item There exists $p\in \JJ_\ell$ such that $T_\ell^N(\cl{B_s}(p,\delta_1))\supseteq \C\butnot B_s(\SS_N,\kappa)$ \label{F}
\end{enumerate}
By (\ref{K4containsjulia}) - (\ref{probabilitylast}), the set of all such $\omega$ is of measure at least $1 - \varepsilon_1$. We claim that $\omega$ satisfies Event \ref{eventmultiexact}.

Let $p\in \JJ_\ell$ be as in event (\ref{F}).

Fix $x\in T_\ell^0(p)$. For all $q\in \cl{B_s}(p,\delta_1)$, event (\ref{E}) and Lemma \ref{lemmacomplex} imply that $\dist_s(x,\Phi_{p,q}(x))\leq\delta_2$, and so $q\in T_0^\ell(\cl{B_s}(x,\delta_2))$. Thus we have $T_0^\ell(\cl{B_s}(x,\delta_2))\supseteq \cl{B_s}(p,\delta_1)$. Applying $T_0^n$ to both sides yields $T_0^N(\cl{B_s}(x,\delta_2))\supseteq \C\butnot B_s(\SS_0,\kappa)$.

Event \ref{B} implies that $p\in \C\butnot B_s(\SS_0,\kappa)$; thus event (\ref{D}) implies that $\diam_m(T_\ell^0(p))\geq\delta_3$. Let $(x_i)_{i=1}^m$ be a $\delta_3$-separated subset of $T_\ell^0(p)$.

Next we define
\begin{align*}
V_i &:= \cl{B_s}(x_i,\delta_2)\\
U_i &:= B_s(x_i,\delta_3/2)
\end{align*}
The fact that $(x_i)_i$ is $\delta_3$-separated implies that $(U_i)_i$ is disjoint. The definition of $\delta_2$ implies that $\diam_{U_i}(V_i)=C_5$. Finally, the fact that $x_i\in T_\ell^0(p)$ implies that $T_0^N(V_i)\supseteq \C\butnot B_s(\SS_N,\kappa)$. Thus we are done.
\end{proof}

We are almost ready to prove Theorem \ref{maintheorem}. We start by tiny steps: more often than not, the oscillation goes down a small amount. We combine the results of Proposition \ref{propositionhyperbolic} and Lemma \ref{lemmamultiexact} into the following proposition:

\begin{proposition}
\label{propositiontinysteps}
Fix $C_6<\infty$. Then there exist $\kappa,\varepsilon_2 > 0$ such that \textup{(\ref{kappaadmissible})} is satisfied, and such that the following event is almost certain to occur:
\begin{quote}
\begin{event}
\label{eventtinysteps}
Suppose that $\sigma$ is a modulus of continuity, and suppose that $B_s(\SS_0,\kappa)\Kin \FF_0$. Then there exists $n_2\in\N$ such that (\ref{exceptionalfatou}) is satisfied, and such that for all $f,g\in\CC(\C)$ with $\rho_{f/g}^{(X\butnot B_s(\SS_0,\kappa))}\leq\gamma$, $g>0$, and $\|\ln(g)\|_{\osc,X\butnot B_s(\SS_0,\kappa)}\leq C_6$,
\begin{equation}
\label{tinysteps}
\left\|\frac{L_0^{n_2}[f]}{L_0^{n_2}[g]}\right\|_{\osc,X\butnot B_s(\SS_{n_2},\kappa)}\leq(1 - \varepsilon_2)\gamma(\pi/2).
\end{equation}
\end{event}
\end{quote}
\end{proposition}
\begin{proof}
The first half of the proof is an exercise in quantifier logic:

Let $C_5 = 1$ and let $\varepsilon_1 = 1/4$. Lemma \ref{lemmamultiexact} guarantees the existence of $\kappa > 0$. Lemma \ref{lemmamotivation} guarantees the existence of $D,C_1,C_2 < \infty$, and $\tau < 1$.

Let $M < \infty$ be the constant guaranteed by the fact that $X$ has the bounded distortion property, let $r=1$, and let $\sigma = \infty$. Let $c = \tau^{1/3}$, so that $\tau < c^2 < 1$. Since $\alpha$, $\beta$, and $C_6$ were given by hypothesis, this completes the specification of the parameters. Thus, the final clause of Proposition \ref{propositionhyperbolic} guarantees the existence of $m\in\N$ such that $\varepsilon_m > 0$.

Lemma \ref{lemmamultiexact} guarantees the existence of $N_2\in\N$.

Let
\begin{equation*}
\varepsilon_2 = \frac{1}{2}\varepsilon_m\frac{e^{-(N_2 C_1 (\pi/2)^\alpha + M + C_6)}}{D^{N_2}} > 0.
\end{equation*}

We now consider the following events:

\begin{enumerate}[A)]
\item (\ref{Misaboundzero}) for all $j\in\N$, with $K := X$ \label{tinystepsenumeratestart}
\item $T_j^n(B)\implies B$ for all $j,n\in\N$ with $j\leq n$ \label{tinystepsenumeratemiddle}
\item For infinitely many $n\in\N$, \label{tinystepsenumerateend}
\begin{enumerate}[i)]
\item Event \ref{eventmultiexact} occurs for $\theta^n\omega$ \label{Ci}
\item (\ref{Disaboundmod}) - (\ref{tauisaboundmod}) hold for $j = 0,\ldots,n - 1$ \label{Cii}
\item \label{Ciii}
\begin{equation}
\label{exceptionalfatou}
B_s(\SS_{n + N_2},\kappa)\Kin \FF_{n + N_2}.
\end{equation}
\end{enumerate}
\end{enumerate}

From the fact that $X$ is backward invariant and has the bounded distortion property, it follows that events (\ref{tinystepsenumeratestart}) - (\ref{tinystepsenumeratemiddle}) have full measure. By Lemmas \ref{lemmamultiexact} and \ref{lemmamotivation}, (\ref{Ci}) - (\ref{Ciii}) have probability at least $1/4 > 0$; by Lemma \ref{lemmapoincare}, event (\ref{tinystepsenumerateend}) has full measure.

Next, we fix a sequence $(T_j)_{j\in\N}$ for which events (\ref{tinystepsenumeratestart}) - (\ref{tinystepsenumerateend}) occur. Our goal is to show that Event \ref{eventtinysteps} also occurs. To this end, we fix $\gamma$ a modulus of continuity, and suppose that $B_s(\SS_0,\kappa)\Kin \FF_0$. Let $N_1\in\N$ be large enough so that
\begin{align}
\gamma(c^{N_1 + N_2}) &\leq \frac{1}{2}\gamma(\pi/2)\\ \label{K1K0}
T_0^n(B_s(\SS_0,\kappa)) &\implies B_s(\SS_n,\kappa) \all n\geq N_1
\end{align}
Such $N_1$ exists by Lemma \ref{lemmafatoucompact}.

By event (\ref{tinystepsenumerateend}), there exists $n\geq N_1$ such that (\ref{Ci}) - (\ref{Ciii}) hold.

Fix $f,g\in\CC(\C)$ with $\rho_{f/g}^{(X\butnot B_s(\SS_0,\kappa))}\leq\gamma$, $g>0$, and $\|\ln(g)\|_{\osc,X\butnot B_s(\SS_0,\kappa)}\leq C_6$. We are done if (\ref{tinysteps}) holds.

Let $n_0 = 0$, $n_1 = n$, and $n_2 = n + N_2$. Let $K_i = X\butnot B_s(\SS_{n_i},\kappa)$ for $i=0,1,2$. Now (\ref{K1K0}) and (\ref{multiexactend}) can be rewritten:
\begin{align}\label{Ksubset01}
T_{n_1}^0(K_1)\implies K_0\\ \label{Ksubset12}
T_{n_2}^{n_1}(K_2)\implies K_1
\end{align}
respectively. Combining with Lemma \ref{lemmanonincreasing} gives
\begin{align}
\left\|\frac{L_0^{n_1}[f]}{L_0^{n_1}[g]}\right\|_{\osc,K_1}
&\leq \left\|\frac{f}{g}\right\|_{\osc,K_0}\\
\left\|\frac{L_0^{n_2}[f]}{L_0^{n_2}[g]}\right\|_{\osc,K_2}
&\leq \left\|\frac{L_0^{n_1}[f]}{L_0^{n_1}[g]}\right\|_{\osc,K_1}
\end{align}
We also rewrite (\ref{tinysteps}); to complete the proof it suffices to show
\begin{equation}
\label{enoughtoshow}
\left\|\frac{L_0^{n_2}[f]}{L_0^{n_2}[g]}\right\|_{\osc,K_2}\leq(1 - \varepsilon_2)\gamma(\pi/2).
\end{equation}
Fix $p,q\in K_2$. Since Event \ref{eventmultiexact} is satisfied for $\theta^n\omega$, we have that there exist a disjoint collection $(U_i)_{i=1}^m$ of open disks, and relatively compact subsets $V_i\Kin U_i$ satisfying (\ref{multiexactstart}) and (\ref{multiexactmiddle}); rewriting (\ref{multiexactmiddle}) in terms of $\omega$ yields
\begin{equation}
\label{multiexactmiddlenew}
T_{n_1}^{n_2}(V_i)\supseteq K_2.
\end{equation}
For each $i=1,\ldots,m$ let $\zeta_i:U_i\rightarrow\C$ be the identity map. By (\ref{multiexactmiddlenew}), for each $i=1,\ldots,m$ there exist $x_i,y_i\in V_i$ with $T_{n_1}^{n_2}(\zeta_i(x_i)) = p$ and $T_{n_1}^{n_2}(\zeta_i(y_i)) = q$.

We verify the hypotheses of Proposition \ref{propositionhyperbolic}:
\begin{itemize}
\item $n_1,m\in\N$
\item $\{\zeta_1,\ldots,\zeta_m\}$ has multiplicity one (since the $U_i$s are disjoint)
\item $(T_j)_{j=0}^{n-1}$ is a finite sequence of rational maps
\item $(\phi_j)_{j=0}^{n-1}$ is a finite sequence of potential functions
\item $K_1\implies\C$
\item (\ref{Disaboundmod}) - (\ref{tauisaboundmod}) hold for all $j = 0,\ldots,n_1 - 1$
\item For each $i=1,\ldots,m$, $x_i,y_i\in U_i$
\item For each $i=1,\ldots,m$, $\zeta_i$ is injective; in particular, it is $\sigma$-locally injective since $\sigma = \infty$
\item (\ref{multiexactstart}) implies (\ref{C3isabound})
\item (\ref{Ksubset12}) implies (\ref{Kisabound})
\item $f,g\in\CC(\C)$ with $g>0$
\item (\ref{Ksubset01}) together with the fact that $\|\ln(g)\|_{\osc,K_0}\leq C_6$ implies $\|\ln(g)\|_{\osc,T_{n_1}^0(K)}\leq C_6$
\end{itemize}

Thus by Proposition \ref{propositionhyperbolic}, we have (\ref{formulahyperbolic}) for $n = n_1$ and $K = K_1$.

Having discharged the quantifiers, we move on to the calculation:

We write
\[
\Psi_{n_1}^{n_2}(x) = \frac{e^{\phi_{n_1}^{n_2}(x)}L_0^{n_1}[g](x)}{L_0^{n_2}[g](T_{n_1}^{n_2}(x))}
\]

The next several lines are modified from the proof of Proposition \ref{propositionhyperbolic}:

\begin{align*}
\sum_{x\in T_{n_2}^{n_1}(p)}\Psi_{n_1}^{n_2}(x)&=1\\
\frac{L_0^{n_2}[f](p)}{L_0^{n_2}[g](p)}
&=\sum_{x\in T_{n_2}^{n_1}(p)}\Psi_{n_1}^{n_2}(x)\frac{L_0^{n_1}[f](x)}{L_0^{n_1}[g](x)}\\
&=\sup_{K_0}(f/g) + \sum_{x\in T_{n_2}^{n_1}(p)}\Psi_{n_1}^{n_2}(x)\left[\frac{L_0^{n_1}[f](x)}{L_0^{n_1}[g](x)} - \sup_{K_0}(f/g)\right]\\
&=\inf_{K_0}(f/g) + \sum_{x\in T_{n_2}^{n_1}(p)}\Psi_{n_1}^{n_2}(x)\left[\frac{L_0^{n_1}[f](x)}{L_0^{n_1}[g](x)} - \inf_{K_0}(f/g)\right]
\end{align*}
\begin{align*}
&\frac{L_0^{n_2}[f](p)}{L_0^{n_2}[g](p)} - \frac{L_0^{n_2}[f](q)}{L_0^{n_2}[g](q)}\\
&=\left\|\frac{f}{g}\right\|_{\osc,K_0} + \left[\sum_{x\in T_{n_2}^{n_1}(p)}\Psi_{n_1}^{n_2}(x)\left[\frac{L_0^{n_1}[f](x)}{L_0^{n_1}[g](x)} - \sup_{K_0}\left(\frac{f}{g}\right)\right] - \sum_{y\in T_{n_2}^{n_1}(q)}\Psi_{n_1}^{n_2}(y)\left[\frac{L_0^{n_1}[f](y)}{L_0^{n_1}[g](y)} - \inf_{K_0}\left(\frac{f}{g}\right)\right]\right]\\
&\leq \left\|\frac{f}{g}\right\|_{\osc,K_0} + \inf_{T_{n_2}^{n_1}(K_2)}(\Psi_{n_1}^{n_2})\sum_{i=1}^m\left(\frac{L_0^{n_1}[f](\zeta_i(x_i))}{L_0^{n_1}[g](\zeta_i(x_i))} - \frac{L_0^{n_1}[f](\zeta_i(y_i))}{L_0^{n_1}[g](\zeta_i(y_i))} - \left\|\frac{f}{g}\right\|_{\osc,K_0}\right)
\end{align*}

We now begin new calculations. By (\ref{formulahyperbolic}),
\begin{align*}
&\sum_{i=1}^m\left(\frac{L_0^{n_1}[f](\zeta_i(x_i))}{L_0^{n_1}[g](\zeta_i(x_i))} - \frac{L_0^{n_1}[f](\zeta_i(y_i))}{L_0^{n_1}[g](\zeta_i(y_i))} - \left\|\frac{f}{g}\right\|_{\osc,K_0}\right)\\
&\leq \left[m\|f/g\|_{\osc,K_0} + \varepsilon_m\left(\rho_{f/g}^{(K_0)}(C_5 c^{n_2}) - \|f/g\|_{\osc,K_0}\right)\right] - m\|f/g\|_{\osc,K_0}\\
&\leq \varepsilon_m\left(\frac{1}{2}\gamma(\pi/2) - \|f/g\|_{\osc,K_0}\right)
\end{align*}
Now,
\begin{align*}
\inf_{T_{n_2}^{n_1}(K_2)}(\Psi_{n_1}^{n_2})
&\geq \frac{e^{\inf(\phi_{n_1}^{n_2})}\displaystyle\inf_{X\cap K_0}(L_0^{n_1}[g])}{\displaystyle\sup_{K_2}(L_0^{n_2}[g])}\\
&\geq e^{\|\phi_{n_1}^{n_2}\|_\osc}\frac{e^{\sup(\phi_{n_1}^{n_2})}}{\sup(L_{n_1}^{n_2}[\one])}\frac{\inf(L_0^{n_1}[\one])}{\sup(L_0^{n_1}[\one])}\frac{\displaystyle\inf_{K_0}(g)}{\displaystyle\sup_{K_0}(g)}\\
&\geq \frac{e^{-(N_2 C_1 (\pi/2)^\alpha + M + C_6)}}{D^{N_2}}
\end{align*}
Combining, we continue our calculation:
\begin{align*}
&\frac{L_0^{n_2}[f](p)}{L_0^{n_2}[g](p)} - \frac{L_0^{n_2}[f](q)}{L_0^{n_2}[g](q)}\\
&\leq \left\|\frac{f}{g}\right\|_{\osc,K_0} + \frac{e^{-(N_2 C_1 (\pi/2)^\alpha + M + C_6)}}{D^{N_2}}\varepsilon_m\left(\frac{1}{2}\gamma(\pi/2) - \|f/g\|_{\osc,K_0}\right)\\
&\leq \gamma(\pi/2)\left[1 + \frac{e^{-(N_2 C_1 (\pi/2)^\alpha + M + C_6)}}{D^{N_2}}\varepsilon_m\left(\frac{1}{2} - 1\right)\right]\\
&= (1 - \varepsilon_2)\gamma(\pi/2)
\end{align*}
Taking the supremum over all $p,q\in K_2$ yields (\ref{enoughtoshow}). Thus we are done.
\end{proof}

Next, we need a technical lemma, which in some sense says that the Perron-Frobenius operator is uniformly continuous:

\begin{lemma}
\label{lemmatechnical}
Suppose that $\gamma_1$ and $\gamma_2$ are moduli of continuity, suppose $T$ is a rational map, and suppose $\phi:\C\rightarrow\R$ is continuous. Then there exists a modulus of continuity $\gamma_3$ such that if $K\implies\C$ and if $f,g:\C\rightarrow\R$ are continuous functions such that $g>0$ on $T^{-1}(K)$, $\rho_{f/g}^{(T^{-1}(K))}\leq\gamma_1$, and $\rho_{\ln(g)}^{(T^{-1}(K))}\leq\gamma_2$, then $\rho_{\frac{L[f]}{L[g]}}^{(K)}\leq\gamma_3$.
\end{lemma}
\begin{proof}
Fix $\delta_2 > 0$; let $\delta_1 > 0$ be the number guaranteed by Lemma \ref{lemmacomplex}.

Fix $x,y\in K$ with $\dist_s(x,y) \leq \delta_1$, and let $\Phi$ be the bijection guaranteed by Lemma \ref{lemmacomplex}. We make the notational convention that summation over $\ZW$ indicates that the summation is taken over $z\in T^{-1}(x)$, and that $w$ is shorthand for $\Phi(z)$.

We begin calculation:
\[
\frac{L[f](x)}{L[g](x)} - \frac{L[f](y)}{L[g](y)}
= \sum_\ZW \left[\frac{e^\phi(z)g(z)}{L[g](x)}\frac{f(z)}{g(z)} - \frac{e^\phi(w)g(w)}{L[g](y)}\frac{f(w)}{g(w)}\right]\\
\]
We recall the following lemma:
\begin{lemma}
\label{lemmaepsilon}
Suppose $S$ is a finite or countably infinite set, and suppose $(a_s)_{s\in S}$ and $(b_s)_{s\in S}$ are probability vectors i.e. positive sequences that sum to one. Suppose $(c_s)_{s\in S}$ and $(d_s)_{s\in S}$ are bounded sequences. Let $K=\sup_{s\in S}|\ln(a_s) - \ln(b_s)|$, and let $\varepsilon=\frac{2}{1 + e^K}$. Then
\begin{equation}
\label{equationepsilon}
\sum_{s\in S}[a_s c_s - b_s d_s]\leq (1-\varepsilon) \left(\sup_{s\in S\in S}(c_s)-\inf_{s\in S}(d_s)\right) + \varepsilon\left(\sup_{s\in S}(c_s - d_s)\right)
\end{equation}
\end{lemma}
\begin{proof}
\cite{SG}, or as an exercise left to the reader.
\QEDmod\end{proof}

Applying this lemma to our particular circumstances:

\begin{align*}
\max_\ZW\frac{f(z)}{g(z)} - \min_\ZW\frac{f(w)}{g(w)}
&\leq \gamma_1(\pi/2)\\
\max_\ZW\left[\frac{f(z)}{g(z)} - \frac{f(w)}{g(w)}\right]
&\leq \gamma_1(\delta_2)
\end{align*}
and finally
\begin{align*}
&\max_\ZW|\phi(z) + \ln(g(z)) - \ln(L[g](x)) - \phi(w) - \ln(g(w)) + \ln(L[g](y))|\\
&\leq 2[\rho_\phi(\delta_2) + \gamma_2(\delta_2)]\\
&\frac{L[f](x)}{L[g](x)} - \frac{L[f](y)}{L[g](y)}\\
&\leq \left(1 - \frac{2}{1 + e^{2[\rho_\phi(\delta_2) + \gamma_2(\delta_2)]}}\right)\gamma_1(\pi/2) + \frac{2}{1 + e^{2[\rho_\phi(\delta_2) + \gamma_2(\delta_2)]}}\gamma_1(\delta_2)
\end{align*}
which tends to zero as $\delta_2$ tends to zero. Thus for all $\sigma > 0$, there exists $\delta_1 > 0$ not depending on $K$, $f$, or $g$, such that for all $x,y\in K$ with $\dist_s(x,y)\leq\delta_2$, we have
\[
\frac{L[f](x)}{L[g](x)} - \frac{L[f](y)}{L[g](y)} \leq \sigma
\]
Thus we are done.
\end{proof}
\begin{proof}[Proof of Theorem \ref{maintheorem}]
Fix $C_6 < \infty$. Let $\kappa,\varepsilon_2 > 0$ be as in Proposition \ref{propositiontinysteps}.

Fix $\omega\in\Omega$ and assume that Event \ref{eventtinysteps} is satisfied for all forward translates of $\omega$, and that $B_s(\SS_n,\kappa)\Kin \FF_n$ for at least one $n\in\N$. By Proposition \ref{propositiontinysteps}, this assumption is almost certainly valid.

Fix $\gamma_1$ and $\gamma_2$ moduli of continuity with $\gamma_2(\pi/2)\leq C_6$.

\begin{claim}
For every $k\in\N$, there exists $n\in\N$ so that
\begin{itemize}
\item $B_s(\SS_n,\kappa)\Kin \FF_n$
\item For all $f,g\in\CC(\C)$ with $g>0$, $\rho_{f/g}^{(X)}\leq\gamma_1$, and $\rho_{\ln(g)}^{(X)}\leq \gamma_2$,
\begin{equation}
\label{Lzeron}
\left\|\frac{L_0^n[f]}{L_0^n[g]}\right\|_{\osc,X\butnot B_s(\SS_n,\kappa)} \leq (1 - \varepsilon_2)^k \gamma_1(\pi/2).
\end{equation}
\end{itemize}
\end{claim}
\begin{proof}
By induction:

Base case $k = 0$: By assumption, there exists $n\in\N$ satisfying $B_s(\SS_n,\kappa)\Kin \FF_n$. (\ref{Lzeron}) follows from (\ref{osclessthan}).

Inductive step: Assume the claim is true for $k\in\N$; let $n\in\N$ be the value that works. By Lemma \ref{lemmatechnical}, there exists a modulus of continuity $\gamma_3$ such that if $f,g:\C\rightarrow\R$ are continuous functions such that $g>0$ on $X$, $\rho_{f/g}^{(X)}\leq\gamma_1$, and $\rho_{\ln(g)}^{(X)}\leq\gamma_2$, then
\[
\rho_{\frac{L_0^n[f]}{L_0^n[g]}}^{(X)}\leq\gamma_3.
\]

Let $\gamma_4(\varepsilon) = \min(\gamma_3(\varepsilon),(1 - \varepsilon_2)^k\gamma_1(\pi/2))$. Clearly, $\gamma_4$ is a modulus of continuity, and $\gamma_4(\pi/2) \leq (1 - \varepsilon_2)^k\gamma_1(\pi/2)$. By the inductive hypothesis,
\begin{equation}
\label{gamma4}
\rho_{\frac{L_0^n[f]}{L_0^n[g]}}^{(X\butnot B_s(\SS_n,\kappa))}\leq\gamma_4
\end{equation}
for relevant $f,g$.

Let $n_2\in\N$ be given by Event \ref{eventtinysteps} for $\theta^n\omega$.

Let $\wtilde{k} = k + 1$ and let $\wtilde{n} = n + n_2$.
\begin{itemize}
\item By (\ref{exceptionalfatou}), we have $B_s(\SS_{\wtilde{n}},\kappa)\Kin \FF_{\wtilde{n}}$.
\item Fix $f,g\in\CC(\C)$ with $g>0$, $\rho_{f/g}^{(X)}\leq\gamma_1$, and $\rho_{\ln(g)}^{(X)}\leq \gamma_2$. Then (\ref{gamma4}) holds, and $\|\ln(L_0^n[g])\|_\osc\leq C_6$. Thus by (\ref{tinysteps}), we have
\begin{equation*}
\left\|\frac{L_0^{\wtilde{n}}[f]}{L_0^{\wtilde{n}}[g]}\right\|_{\osc,X\butnot B_s(\SS_{\wtilde{n}},\kappa)} \leq (1 - \varepsilon_2) \gamma_4(\pi/2) \leq (1 - \varepsilon_2)^{\wtilde{k}}\gamma_1(\pi/2)
\end{equation*}
completing the inductive step.
\end{itemize}
\QEDmod\end{proof}

We claim that Event \ref{mainevent} is satisfied, with the restriction that $\gamma_2(\pi/2)\leq C_6$. If we show this, we are done since $C_6 < \infty$ was arbitrary, and can be quantified countably.

Thus, we need to show that for all $\varepsilon > 0$ and for all $\wtilde{\kappa} > 0$, there exists $N\in\N$ not depending on $f,g$ such that for all $n\geq N$, (\ref{mainequation}) holds. To this end, fix $\varepsilon > 0$ and $\wtilde{\kappa} > 0$. Let $k\in\N$ be large enough so that $(1 - \varepsilon_2)^k \gamma_1(\pi/2) \leq \varepsilon$. Let $n\in\N$ be given by the claim. By Lemma \ref{lemmafatoucompact}, there exists $N\in\N$ such that for all $\wtilde{n}\geq N$,
\[
T_n^{\wtilde{n}}(B_s(\SS_n,\kappa))\implies B_s(\SS_{\wtilde{n}},\wtilde{\kappa})
\]
and thus the backwards invariance of $X$ and of the Julia set imply
\[
T_{\wtilde{n}}^n(X\butnot B_s(\SS_{\wtilde{n}},\wtilde{\kappa}) \cup \JJ_{\wtilde{n}})\implies X\butnot B_s(\SS_n,\kappa) \cup \JJ_n = X\butnot B_s(\SS_n,\kappa).
\]
(Here we have used that $\#(X)\geq 3$ plus backwards invariance to yield that $\JJ_n\implies X$.)

Now, (\ref{Lzeron}) and (\ref{osclessthan}) yield (\ref{mainequation}).
\end{proof}
\end{section}
\begin{section}{Lemmas Leading Up to the Proof of Theorem \ref{theoremcondition}}\label{sectioncondition}
First, we construct the set $B$ in the conclusion of Theorem \ref{theoremcondition}:

\begin{lemma}
\label{lemmaB}
Suppose that $\AA\implies\RR$ is a finite set, and suppose $F\implies\C$ is finite with $\AA(F)\implies F$. Also suppose that (ii) of Theorem \ref{theoremcondition} holds. Then there exists a neighborhood $B$ of $F$ and a neighborhood $\BBB_1$ of $\AA$ such that $X := \C\butnot B$ is closed, connected, and contains at least three points, and such that
\[
\BBB_1(B)\Kin B.
\]
\end{lemma}
\begin{proof}
We leave it to the reader to verify that condition (ii) implies that there exists a function $c:F\rightarrow (0,\infty)$ such that for all $p\in F$ and for all $T\in\AA$,
\[
T_*(p)\frac{c(p)}{c(T(p))} < 1.
\]
Assume that such a $c$ is given. For each $p\in F$ let $\phi_p$ be any M\"obius transformation such that $\phi_p(0) = p$ and $\|(\phi_p)_*(0)\|_e^s = c(p)$. For each $p\in F$ and $T\in\AA$ let $\psi_{p,T} = \phi_{T(p)}^{-1}\circ T \circ \phi_p$, so that
\begin{align*}
\psi_{p,T}(0) = 0\\
\|(\psi_{p,T})_*(0)\|_e < 1
\end{align*}
i.e. $0$ is an attracting fixed point for the finite collection $(\psi_{p,T})_{p\in F,T\in\AA}$. By the definition of derivative, it follows that for each $p\in F$ and $T\in\AA$ there exists $\delta_{p,T}>0$ such that $|\psi_{p,T}(w)| < |w|$ for all $w\in \cl{B_e}(0,\delta_{p,T})$. Let $\delta = \min_{p,T}\delta_{p,T} > 0$, so that for all $p\in F$, for all $T\in\AA$, and for all $w\in \cl{B_e}(0,\delta)$, we have $|\psi_{p,T}(w)| < |w| \leq \delta$, i.e. $\psi_{p,T}(w)\in B_e(0,\delta)$.

Let $\BBB_0 = \{\psi\in\RR: \psi(\cl{B_e}(0,\delta))\implies B_e(0,\delta)\}$; $\BBB_0$ is subbasic open in the compact-open topology, and contains $\psi_{p,T}$ for all $p\in F$ and $T\in\AA$.  Let $B_0 = B_e(0,\delta)$, so that $\BBB_0(B_0)\Kin B_0$.

Let $B = \cup_{p\in F}\phi_p(B_0)$ and $\BBB_1 = \cap_{p\in F}\cup_{w\in F}\phi_w\BBB_0\phi_p^{-1}$. It is easily verified that $B$ is an open neighborhood of $F$, that $\BBB_1$ is an open neighborhood of $\AA$, that $X := \C\butnot B$ is closed, connected, and contains at least three points (assuming $\delta$ is sufficiently small) and that $\BBB_1(B)\Kin B$.
\end{proof}

Next, we give an elementary bound on the multiplicity of a point $p\notin F$:

\begin{lemma}
\label{lemmaR}
Suppose that $\AA\implies\RR$ is a finite set, and suppose $F\implies\C$ is finite with $\AA(F)\implies F$. Also suppose that (i) of Theorem \ref{theoremcondition} holds. Let
\[
r = \prod_{S\in\AA}\prod_{w\in\RP_S\butnot F}\mult_S(w).
\]
For all $\ell\in\N$, for all $T\in\AA^\ell$, and for all $p\in\C\butnot T^{-1}(F)$,
\begin{equation}
\label{equationR}
\mult_T(p)\leq r.
\end{equation}
\end{lemma}
\begin{proof}
Suppose $T = T_0^\ell$, where $T_j\in\AA$ for $j=0,\ldots,\ell-1$. Then
\[
\mult_T(p) = \prod_{j=0}^{\ell-1}\mult_{T_j}(T_0^j(p))
= \prod_{S\in\AA}\prod_{q\in\RP_S}(\mult_S(q))^{\#(j: q = T_0^j(p)\text{ and }S = T_j)}
\]
so it suffices to show that for all $S\in\AA$ and for all $q\in\RP_S$, $\#(j: q = T_0^j(p)\text{ and }S = T_j)$ is at most one, and is zero when $q\in F$.

To prove the latter, note that if $T_0^j(p)\in F$, then $T_0^\ell(p)\in F$ since $\AA(F)\implies F$, contradicting the hypothesis that $p\notin T^{-1}(F)$.

To prove the former, by contradiction we suppose $0 \leq j_1 < j_2 < \ell$ with $T_0^{j_1}(p) = T_0^{j_2}(p)\in \RP_{T_{j_1}} = \RP_{T_{j_2}}$. But then
\[
T_0^{j_1}(p) \in \RP_{T_{j_1}^{j_2}}\cap\FP_{T_{j_1}^{j_2}} \implies F,
\]
and contradiction follows as above.
\end{proof}

\begin{lemma}
\label{lemmaL}
Suppose that $\AA\implies\RR\butnot\RR_1$ is a finite set, and suppose $F\implies\C$ is finite with $\AA(F)\implies F$. Also suppose that (i) of Theorem \ref{theoremcondition} holds. If $B\implies\C$ is a neighborhood of $F$, then there exist $\inc\in\N$, $\delta_3 > 0$, and $\BBB_2$ a neighborhood of $\AA$ and  so that for all $T\in\BBB_2^\inc$ and for all $p\in\C\butnot B$,
\begin{equation}
\label{equationL}
\diam(T^{-1}(p))\geq \delta_3.
\end{equation}
\end{lemma}
\begin{proof}
Let $r\in\N$ be as in Lemma \ref{lemmaR}. Let $\inc$ be large enough so that $2^\inc > r$. Then (\ref{equationR}) yields that for all $T\in\AA^\inc$ and for all $x\in\C\butnot T^{-1}(F)$, we have
\[
\mult_T(x) < 2^\inc \leq \deg(T).
\]
Thus no point of $\C\butnot T^{-1}(F)$ is totally ramified, and so no point of $\C\butnot F$ is totally branched.

Now, the map $(T,p)\mapsto \diam(T^{-1}(p))$ is continuous; since $\C\butnot B$ is compact the map
\[
T\mapsto\inf_{p\in\C\butnot B}\diam(T^{-1}(p))
\]
is continuous, and strictly positive on $\AA^\inc$. Thus there exists $\delta_3 > 0$ and a neighborhood $\BBB_\inc$ of $\AA^\inc$ on which (\ref{equationL}) holds for all $p\in\C\butnot B$; letting $\BBB_2$ be a neighborhood of $\AA$ such that $\BBB_2^\inc\implies\BBB_\inc$ yields the lemma.
\end{proof}

In the following lemma, the idea is to construct a set $U$ which can be used as a domain on which to bound the Perron-Frobenius operator by considering inverse branches. The set $U$ should be simply connected and should not contain any branch points of $T$, so that it has $\deg(T)$ conformally isomorphic preimages. Since $U$ must not contain branch points, it must vary depending on $T$, but in a predictable way. Specifically, if $S$ is a perturbation of $T$ then the the isomorphic preimages of $U_S$ under $S$ will be close in some sense to the isomorphic preimages of $U_T$ under $T$. We now state our lemma:

\begin{lemma}
\label{lemmaspiderperturbation}
Suppose $F\implies\C$ is finite, and suppose $B\implies\C$ is a neighborhood of $F$. Let $X = \C\butnot B$. Fix $r\in\N$ and $d\geq 1$. Suppose $T\in\RR_d$ is such that $\mult_T(p)\leq r$ for all $p\in\C\butnot T^{-1}(F)$. Then there exist $C<\infty$, $\sigma>0$, $\BBB\implies\RR_d$ a neighborhood of $T$, and $(\zeta_i)_{i=1}^d$ a collection of $\sigma$-locally injective holomorphic maps from $\B$ to $\C$, such that
\begin{equation}
\label{multzetabounds}
\mult(\zeta_i)_{i=1}^m\leq 9r
\end{equation}
Furthermore, for all $S\in\BBB$, there is a simply connected hyperbolic open set $U_S$ such that
\begin{itemize}
\item[i)] $\cl{U_S\cap X} = X$
\item[ii)] $U_S\cap \BP_S = \emptyset$
\end{itemize}
Furthermore, for each $i = 1,\ldots,d$ there is a holomorphic map $\xi_{i,S}:U_S\rightarrow\B$ such that
\begin{itemize}
\item[iii)] $S\circ\zeta_i\circ\xi_{i,S} = \id$
\item[iv)] $\diam_h(\xi_{i,S}(U_S))\leq C$
\item[v)] For all $p\in U_S$, we have the equality of multisets
\[
S^{-1}(p) = \{\zeta_i\circ\xi_{i,S}(p):i=1,\ldots,d\}.
\]
\end{itemize}
\end{lemma}

The proof of this lemma is very technical, and requires other lemmas which will not be used directly in the proof of Theorem \ref{theoremcondition}. It may be advisable to temporarily assume Lemma \ref{lemmaspiderperturbation}, and skip directly to the proof of Theorem \ref{theoremcondition}.

\begin{remark}
We shall use the following fact from topology repeatedly and without explicit mention: If $W$ is an open subset of a Riemann surface $X = \C\text{ or }\B$, then $P\implies X$ is clopen relative to $W$ if and only if $\del P\implies X\butnot W$. In particular,
\begin{itemize}
\item If $P\in\connected(W)$, then $\del P\implies \del W$
\item If $P\implies W$ and $\del P\implies \del W$ and if $W$ has finitely many connected components, then $P$ is the union of some subcollection of the connected components of $W$
\end{itemize}
\end{remark}

We begin with a technical definition:

\begin{definition}
A set $U\implies\C$ is \emph{nice} if:
\begin{itemize}
\item[i)] $U$ is open, simply connected, and hyperbolic
\item[ii)] $\cl{U}$ is finitely connected (Recall that a set is defined to be finitely connected if its complement has finitely many connected components)
\item[iii)] The set $\BB_U^{(3)}:=\{A\implies\C\text{ open}:\#(\connected(A\cap U))\leq 3\}$ is a basis of topology for $\C$
\end{itemize}
\end{definition}

Note: The occurence of the number $3$ here is best explained by the caption of Figure \ref{figureU}.

Next, we construct the set $U$ for a given value of $T$, not worrying about perturbations of $T$. The idea is that $F_1 := \BP_T$ and $F_2 := F$. We call the following a construction rather than a lemma since specific details of the construction will be used in the sequel.

\begin{construction}
\label{constructionspider}
Suppose $F_1,F_2\implies\C$ are finite, and suppose that $B\implies\C$ is a neighborhood of $F_2$. Let $X = \C\butnot B$. We construct a nice set $U\implies\C$ such that
\begin{itemize}
\item[i)] $\cl{U\cap X}\supseteq X$
\item[ii)] $U\cap F_1=\emptyset$, $\cl{U}\cap F_2=\emptyset$
\end{itemize}
\end{construction}

\begin{proof}
Let $F = F_1\cup F_2$; without loss of generality suppose that $F$ is contained in the upper half plane and that $\#(F)\geq 3$. For each $p\in F$, let $K_p=[0,\RE[p]]\cup[\RE[p],p]$. Fix $\delta > 0$ such that $B_e(F_2,\delta)\implies B$. Let
\[
K=\bigcup_{p\in F}K_p \cup \bigcup_{p\in F_2}\cl{B_e}(p,\delta),
\]
and let $U=\C\butnot K$. We leave it to the reader to verify that the set $U$ has the properties mentioned above, assuming that $\delta$ is sufficiently small.
\end{proof}

\begin{figure}
\centerline{\mbox{\includegraphics[scale=1.2]{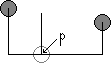}}}
\caption{The set $K$ of Construction \ref{constructionspider}. The point $p$ has no neighborhood $A\implies B$ such that $\#(\connected(A\cap U)) < 3$.}
\label{figureU}
\end{figure}

The following lemma is key to correcting an error made in [\cite{DU}, p.109, statement directly below (2.14)]. The paper implicitly assumes that if $U$ is a topological disc then there exist arbitrarily small neighborhoods of $\cl{U}$ which are also topological discs. However this is only true if we assume that $\cl{U}$ is itself simply connected, which is false in general. For example, consider the set $U$ constructed in \ref{constructionspider}. $U$ is simply connected, but the number of connected components of $\C\butnot\cl{U}$ is $\#(F_2)$ which could be greater than one. Furthermore, if $B$ is small enough this is true no matter how $U$ is chosen, because of requirements (i) and (ii).

The way in which we correct this error is as follows: Rather than considering a neighborhood of $\cl{U}$, we consider a map $\zeta:\B\rightarrow\C$ whose image contains $\cl{U}$. The image may intersect itself, but we give a bound on the multiplicity of self-intersection i.e. $\mult(\zeta)$. We also prove that the map $\zeta$ is uniformly locally injective, so as to be able to apply Lemma \ref{lemmakoebetwo}.

\begin{lemma}
\label{lemmaerrorfix}
Suppose that $U$ is a nice set. Then for every $\delta > 0$, there exist $\zeta:\B\rightarrow B_s(U,\delta)$ and $\xi:U\rightarrow\B$ such that
\begin{itemize}
\item[i)] $\zeta\circ\xi = \id$
\item[ii)] $\xi(U)\Kin \B$.
\item[iii)] $\mult(\zeta)\leq 3$
\item[iv)] $\zeta$ is uniformly locally injective.
\end{itemize}
\end{lemma}

\begin{proof}
Let $K$ be a transversal of $\connected(\C\butnot\cl{U})$. Since $\cl{U}$ is finitely connected, $K$ is finite; in particular $K$ is closed. (Note by contrast that if $\cl{U}$ were  not finitely connected, there would be no closed transversal of $\connected(\C\butnot\cl{U})$.) Let $B$ be the connected component of $B_s(U,\delta)\butnot K$ containing $U$. Note that $P\nsubseteq B$ for each $P\in\connected(\C\butnot\cl{U})$.

The uniformization theorem guarantees that the universal cover of $B$ is conformally isomorphic to $\B$. Thus there exists a covering map $\basemeasure:\B\rightarrow B$. By the homotopy lifting principle, there exists a map $\psi:U\rightarrow\B$ such that $\basemeasure\circ\psi = \id$.

Clearly, we cannot set $\zeta := \basemeasure$ and $\xi := \psi$, for in this case we would have $\mult(\zeta) = \infty$ (unless $B$ was simply connected), contradicting (iii). Instead, we will prove the existence of a neighborhood $W$ of $\cl{\psi(U)}$ such that if $\zeta_W:\B\rightarrow W$ is a conformal isomorphism, then $\zeta := \basemeasure \circ \zeta_W$ and $\xi := \zeta_W^{-1} \circ \psi$ satisfy (i) - (iv).

In the following claim, (E) is most difficult step.
\begin{claim}
\label{claimpsiU}~
\begin{enumerate}[A)]
\item $\psi(U) \Kin \B$.
\item $\basemeasure\on B_h(\cl{\psi(U)},1)$ is uniformly locally injective
\item For all $p\in\cl{U}$, $\#(\basemeasure^{-1}(p)\cap\cl{\psi(U)})\leq 3$
\item There exists $0 < \delta \leq 1$ such that for all $p\in\C$, $\#(\basemeasure^{-1}(p)\cap B_h(\cl{\psi(U)},\delta))\leq 3$
\item $\cl{\psi(U)}$ is simply connected.
\item There exists $W$ a simply connected neighborhood of $\cl{\psi(U)}$ such that $W\implies B_h(\cl{\psi(U)},\delta)$.
\end{enumerate}
\end{claim}
\begin{proof}~
\begin{enumerate}[A)]
\item It suffices to show that for any sequence $(p_n)_{n\in\N}$ in $U$, the sequence $(\psi(p_n))_{n\in\N}$ has a cluster point. By the compactness of $\C$, we may without loss of generality suppose that $(p_n)_n$ converges, say to $p\in\cl{U}$.

We claim that $(\psi(p_n))_{n\in\N}$ has at least one and at most three cluster points. (The upper bound will be used in the proof of (C).) To see this, note that since $U$ is nice and since $\basemeasure$ is a covering map, there exists a neighborhood $A_p$ of $p$ for which $\#\connected(A_p\cap U)\leq 3$ and for which $\basemeasure^{-1}(A_p)$ is the disjoint union of open sets on which $\basemeasure$ is a homeomorphism. For each $Q\in\connected(A_p\cap U)$, $\psi(Q)$ is a connected subset of $\basemeasure^{-1}(A_p)$, and is thus entirely contained in some open set $V_Q\implies\basemeasure^{-1}(A_p)$ on which $\basemeasure$ is a homeomorphism. Now the cluster points of $(\psi(p_n))_{n\in\N}$ are exactly the points of the form $\basemeasure_{V_Q}^{-1}(p)$, where $Q$ contains a subsequence of $(p_n)_{n\in\N}$. There is at least one, and at most three.
\item
If $\sigma > 0$ is the Lebesgue number of the cover
\[
\{U\implies\C\text{ open}:\basemeasure\on U\text{ is injective}\}
\]
of $\cl{B_h}(\cl{\psi(U)},1)$, then $\basemeasure$ is $\sigma$-locally injective on $B_h(\cl{\psi(U)},1)$.
\item
By contradiction, suppose that $(x^{(i)})_{i=1}^4$ are distinct elements of $\basemeasure^{-1}(p)\cap\cl{\psi(U)}$. For $i=1,\ldots,4$ there exists a sequence $(p_n^{(i)})_{n\in\N}$ in $U$ such that $\psi(p_n^{(i)})\tendston x^{(i)}$. By combining these sequences, we find a sequence $(p_n)_n$ in $U$ such that $(\psi(p_n))_n$ has four cluster points, contradicting the proof of (A).
\item
By contradiction, suppose that for each $n\in\N$ there exists $p_n\in\C$ with four distinct preimages $x_n^{(i)}\in\basemeasure^{-1}(p_n)\cap B_h(\cl{\psi(U)},2^{-n})$, $i=1,\ldots,4$. Since $B_h(\cl{\psi(U)},1)$ is relatively compact in $\B$, we may assume without loss of generality that for $i=1,\ldots,4$, the sequence $(x_n^{(i)})_{n\in\N}$ converges to some point $x^{(i)}\in\B$. In fact $x^{(i)} \in \cap_{n\in\N}\cl{B_h}(\cl{\psi(U)},2^{-n}) = \cl{\psi(U)}$. By continuity $\basemeasure(x^{(i)}) = p := \lim_{n\rightarrow\infty}p_n$. Thus $x^{(i)} \in \basemeasure^{-1}(p) \cap \cl{\psi(U)}$ for $i=1,\ldots,4$. By the pigeonhole principle, (C) implies that there exist $i\neq j$ with $x^{(i)} = x^{(j)}$. Since $\basemeasure$ is locally injective, for $n$ sufficiently large we have $x_n^{(i)} = x_n^{(j)}$. This is a contradiction, since for each $n\in\N$ $(x_n^{(i)})_{i=1}^4$ were assumed to be distinct.
\item
Since $\cl{\psi(U)}$ is compact, it suffices to show that every connected component of $\B\setminus\cl{\psi(U)}$ is (hyperbolically) unbounded. By contradiction, suppose that $Q\in\connected\left(\B\setminus\cl{\psi(U)}\right)$ is bounded. Then $\cl{Q}$ is compact, and thus $\basemeasure(\cl{Q})$ is closed (as a subset of $\C$). By the open mapping theorem, $\basemeasure(Q)$ is open.

Thus
\begin{equation*}
\del\basemeasure(Q)
\implies \basemeasure(\cl{Q})\setminus\basemeasure(Q)
\implies \basemeasure(\del Q)
\implies \basemeasure(\cl{\psi(U)})
\implies \cl{U}
\end{equation*}
Thus $\basemeasure(Q)$ is relatively clopen in $\C\setminus\cl{U}$. Thus for each $P\in\connected(\C\setminus\cl{U})$, $P\implies \basemeasure(Q)$ or $P\cap \basemeasure(Q) = \emptyset$. But $\basemeasure(Q)\implies B\nsupseteq P$,\footnote{This is the only point in the argument where we use the hypothesis that $\C\butnot\cl{U}$ is finitely connected.} so $P\cap \basemeasure(Q) = \emptyset$ for each $P\in\connected(\C\setminus\cl{U})$. Thus $\basemeasure(Q)\implies\cl{U}$. Since $\basemeasure(Q)$ is open, $\basemeasure(Q)\cap U\neq\emptyset$; fix $x\in Q$ with $\basemeasure(x)\in U$. Let $p = \basemeasure(x)$.

Now $\basemeasure(\psi(p)) = p = \basemeasure(x)$, so there exists a deck transformation $\phi:\B\rightarrow\B$, $\basemeasure\circ\phi = \basemeasure$, such that $\phi(\psi(p)) = x$. We have $\phi(\psi(U)) \implies \B\butnot\cl{\psi(U)}$ connected with $\phi(\psi(U))\cap Q\neq\emptyset$, so $\phi(\psi(U))\implies Q$.

Let $Q_\infty$ be the unbounded component of $\B\setminus\cl{\psi(U)}$. Then $Q_\infty\cap \cl{Q} = \emptyset$, so
\begin{equation*}
Q_\infty
\implies \B\setminus \cl{Q}
\implies \B\setminus\cl{\phi(\psi(U))}
= \phi\left(\B\setminus\cl{\psi(U)}\right)
\end{equation*}
Since $Q_\infty$ is unbounded and connected, it is contained in the unbounded component $\phi(Q_\infty)$ of $\phi\left(\B\setminus\cl{\psi(U)}\right)$. Thus
\[
\phi\left(\B\setminus\cl{Q_\infty}\right) \implies \B\setminus\cl{Q_\infty}
\]
which is a bounded set. However, it is well-known that any conformal isomorphism of $\B$ which preserves a bounded set has a fixed point (in $\B$, rather than in the closure). But since $\phi$ is a deck transformation, the uniqueness of homotopy lifting implies that $\phi = \id$. But then $\psi(p)\in Q$, contradicting that $Q\in\connected\left(\B\setminus\cl{\psi(U)}\right)$.
\item Let
\[
V = B_k\left(\B\butnot B_h(\cl{\psi(U)},\delta),\delta/2\right) \implies \B\butnot\cl{\psi(U)}.
\]
Clearly, the hyperbolic area of each connected component of $V$ is bounded from below, and the total hyperbolic area of the bounded components is finite. Thus $V$ has only finitely many connected components. For each $Q\in\connected(V)$, by (E) there exists a closed connected set (the image of a path) $K_Q\implies \B\butnot\cl{\psi(U)}$ which contains points both in $Q$ and in the unbounded component of $V$. Let
\[
K = \cl{V} \cup \bigcup_{Q\in\connected(V)}K_Q;
\]
so that $K$ is closed and connected, and
\[
\B\setminus B_h(\cl{\psi(U)},\delta)\implies V\implies K\implies \B\butnot\cl{\psi(U)}.
\]
Taking complements,
\[
\cl{\psi(U)}\implies \B\butnot K\implies B_h(\cl{\psi(U)},\delta).
\]
Let $W$ be the connected component of $\B\butnot K$ containing $\psi(U)$. Since $K$ is connected, $W$ is simply connected. Thus we are done.
\end{enumerate}
\QEDmod\end{proof}

By the Riemann mapping theorem, there exists a conformal isomorphism $\zeta_W:\B\rightarrow W$. Let $\zeta = \basemeasure\circ\zeta_W$, and let $\xi = \zeta_W^{-1}\circ\psi$, so that $\zeta:\B \rightarrow B_s(U,\delta)$ and $\xi:U\rightarrow\B$. It is easily verified that $(\zeta,\xi)$ satisfy (i) - (iv). (Note that for (iv), we use the Schwarz-Pick inequality on the map $\zeta_W$.)
\end{proof}

\begin{proof}[Proof of Lemma \ref{lemmaspiderperturbation}]
Let $F_1 = \BP_T$, and let $F_2 = F$. Construction \ref{constructionspider} yields a nice set $U$ such that $\cl{U\cap X} = X$, $U\cap \BP_T=\emptyset$, and $\cl{U}\cap F=\emptyset$. Since $U$ contains no branch points of $T$ and is simply connected, the homotopy lifting principle (or alternatively, the Riemann-Hurwitz formula) implies that $\connected{T^{-1}}(U)$ consists of $d$ sets $V_1,\ldots,V_d$, each of multiplicity one.

\begin{claim}
\label{claimnice}
For $i=1,\ldots,d$, $V_i$ is a nice set.
\end{claim}
\begin{proof}~
\begin{itemize}
\item[i)]
Clearly $V_i$ is open. Since $V_i$ has multiplicity one, it is conformally isomorphic to $U$. In particular, $V_i$ is simply connected and hyperbolic, since these are conformal invariants.
\item[ii)]
Since $U$ is nice, $\cl{U}$ is finitely connected. If $P$ is a connected component of $\C\setminus\cl{U}$, then $T^{-1}(P)$ has at most $d$ connected components (exactly $d$ counting multiplicity). Thus $T^{-1}(\C\setminus\cl{U}) = \C\setminus\cl{T^{-1}(U)}$ has finitely many connected components.

Now $T^{-1}(U) = \cup_{i=1}^d V_i$. In particular, $T^{-1}(U)\setminus\cl{V_i} = \cup_{j\neq i}V_j$ has exactly $d-1$ connected components.

Each connected component of $\C\setminus\cl{V_i}$ is open, and therefore intersects either $\C\setminus\cl{T^{-1}(U)}$ or $T^{-1}(U)\setminus\cl{V_i}$ nontrivially, and in fact contains some connected component of the set that it intersects. Thus the number of connected components of $\C\setminus\cl{V_i}$ is at most the number of connected components of $\C\setminus\cl{T^{-1}(U)}$ plus the number of connected components of $T^{-1}(U)\setminus\cl{V_i}$; in particular, this number is finite. Thus $\cl{V_i}$ is finitely connected.
\item[iii)]
Fix $x\in\C$ and $B_x\implies\C$ a neighborhood of $x$. We want to show that there exists $A_x\in\BB_{V_i}^{(3)}$ such that $x\in A_x\implies B_x$. By Lemma \ref{lemmamiranda}, we may without loss of generality assume that $T(\del B_x)\cap T(B_x) = \emptyset$.

Since $U$ is nice, there exists $A_{T(x)}\in\BB_U^{(3)}$ such that $T(x)\in A_{T(x)}\implies T(B_x)$. Let
\[
A_x = T^{-1}(A_{T(x)})\cap B_x.
\]
We have $x\in A_x\implies B_x$. We claim that $A_x\in\BB_{V_i}^{(3)}$ i.e. $\#(\connected(A_x\cap V_i))\leq 3$. Now
\[
A_x\cap V_i = \bigcup_{Q\in\connected(A_{T(x)}\cap U)}T^{-1}(Q)\cap B_x\cap V_i.
\]
Since $A_{T(x)}\in\BB_U^{(3)}$, $\#(\connected(A_{T(x)}\cap U))\leq 3$. Thus it suffices to show that for each $Q\in\connected(A_{T(x)}\cap U)$, $T^{-1}(Q)\cap B_x\cap V_i$ is connected.

Consider $T_{V_i}^{-1}:U\rightarrow V_i$ the inverse branch of $T$ corresponding to $V_i$. Now, $T^{-1}(Q)\cap V_i = T_{V_i}^{-1}(Q)$ is connected, being the continuous image of a connected set. Thus it suffices to show that $T_{V_i}^{-1}(Q)\implies B_x$ or $T_{V_i}^{-1}(Q)\cap B_x = \emptyset$. To see this, note that
\begin{align*}
Q &\implies U\cap A_{T(x)}\\
&\implies U\cap T(B_x)\\
&\implies U\setminus T(\del B_x)\\
T_{V_i}^{-1}(Q)
&\implies T_{V_i}^{-1}(U\butnot T(\del B_x))\\
&\implies V_i\butnot\del B_x\\
&= V_i\butnot\cl{B_x} \cup (V_i\cap B_x)
\end{align*}
Since $T_{V_i}(Q)$ is connected, this completes the proof.
\end{itemize}
\QEDmod\end{proof}

The next logical step would be to apply Lemma \ref{lemmaerrorfix} to get $(\zeta_i)_{i=1}^d$. However, we will instead delay this step and instead perform a somewhat more complicated logical maneuver: We will first prove the existence of a number $\delta > 0$ such that when Lemma \ref{lemmaerrorfix} is applied, the resulting sequence $(\zeta_i)_{i=1}^d$ satisfies (\ref{multzetabounds}). To this end we will prove the following:

\begin{claim}~
\[
\mult(\cl{V_i})_{i=1}^d\leq 3r
\]
\end{claim}
\begin{proof}
Fix $x\in\C$. By Lemma \ref{lemmamiranda}, there exists a neighborhood $B_x$ of $x$ such that
\[
T_{B_x} := T\on B_x:B_x\rightarrow T(B_x)
\]
is proper of degree $k := \mult_T(x)$. Since $U$ is nice, there exists $A_{T(x)}\in\BB_U^{(3)}$ with $T(x)\in A_{T(x)}\implies T(B_x)$.

As above, let
\[
A_x = T^{-1}(A_{T(x)})\cap B_x = T_{B_x}^{-1}(A_{T(x)});
\]
$A_x$ is a neighborhood of $x$. Now for each $Q\in\connected(A_{T(x)}\cap U)$, the set $T_{B_x}^{-1}(Q) = T^{-1}(Q)\cap B_x$ has at most $k$ connected components (exactly $k$ counting multiplicity). Thus
\begin{equation*}
A_x\cap T^{-1}(U)
= T_{B_x}^{-1}(A_{T(x)} \cap U)
= \bigcup_{Q\in\connected(A_{T(x)}\cap U)}T_{B_x}^{-1}(Q)
\end{equation*}
has at most $3k\leq 3r$ connected components.

For each $i=1,\ldots,d$, if $x\in\cl{V_i}$ then $V_i$ contains a connected component of $A_x\cap T^{-1}(U)$. Thus $\mult_x(\cl{V_i})_{i=1}^d \leq 3r$. Since $x$ was arbitrary, we are done.
\QEDmod\end{proof}

\begin{claim}
\label{claimmultViepsilon}
There exists $\delta > 0$ such that $\mult(B_s(V_i,\delta))_{i=1}^d\leq 3r$.
\end{claim}
\begin{proof}
By contradiction, suppose that for all $n\in\N$ there exists $x_n\in\C$ such that
\[
\mult_{x_n}(B_s(V_i,2^{-n}))_{i=1}^d > 3r.
\]
Without loss of generality, suppose that $(x_n)_n$ converges, say to $x\in\C$. Again without loss of generality, we may suppose that $I := \{i=1,\ldots,d:x_n\in B_s(V_i,2^{-n})\}$ is independent of $n$. Taking limits, we find that $x\in\cl{V_i}$ for all $i\in I$. But then $\#(I)\leq\mult(\cl{V_i})_{i=1}^d\leq 3r$, contradicting that $\mult_{x_n}(B_s(V_i,2^{-n}))_{i=1}^d>3r$ for all $n$.
\QEDmod\end{proof}

For each $i=1,\ldots,d$ we now apply Lemma \ref{lemmaerrorfix} to $V_i$ and $\delta$; the result is a pair of maps $\zeta_i:\B\rightarrow B_s(V_i,\delta)$ and $\wtilde{\xi_i}:V_i\rightarrow\B$. Let $\xi_i = \wtilde{\xi_i}\circ T_{V_i}^{-1}$, so that (i) - (iv) of Lemma \ref{lemmaerrorfix} become
\begin{itemize}
\item[i)] $\zeta_i\circ\xi_i = T_{V_i}^{-1}$
\item[ii)] $\xi_i(U) \Kin \B$
\item[iii)] $\mult(\zeta_i)\leq 3$
\item[iv)] $\zeta_i$ is uniformly locally injective
\end{itemize}
Combining Claim \ref{claimmultViepsilon} and (iii) above yields (\ref{multzetabounds}).

For each $\delta_0 > 0$ let $\BBB_{\delta_0}$ be the neighborhood of $T$ guaranteed by Lemma \ref{lemmamorecomplex}. We claim that if $\delta_0$ is sufficiently small, then the conclusion of Lemma \ref{lemmaspiderperturbation} is satisfied for all $S\in\BBB_{\delta_0}$.

To this end, fix $S\in\BBB_{\delta_0}$. Let $\Phi_\BP:\BP_T\rightarrow\BP_S$ be the bijection guaranteed by Lemma \ref{lemmamorecomplex}. From the construction of $U$ given in Construction \ref{constructionspider}, we have $\infty\notin\BP_T$. Thus if $\delta_0$ is sufficiently small, then $\delta_0 < \dist_s(\infty,\BP_T)$, so that $\infty \notin \BP_S$. Let
\[
U_S := U \butnot \bigcup_{p\in\BP_T}[p,\Phi_\BP(p)]
\]
From the construction of $U$, it follows that if $\delta_0$ is sufficiently small, then $U_S$ is open, simply connected, and hyperbolic, and that (i) and (ii) of Lemma \ref{lemmaspiderperturbation} are satisfied.

Let $\Phi_\infty: T^{-1}(\infty)\rightarrow S^{-1}(\infty)$ be the bijection guaranteed by Lemma \ref{lemmamorecomplex}. For each $i=1,\ldots,d$, let $z_i = \Phi_\infty(\zeta_i\circ\xi_i(\infty))$, so that $S^{-1}(\infty) = \{z_1,\ldots,z_n\}$ (as multisets). By the homotopy lifting principle, there exists a unique map $\eta_{i,S}:U_S\rightarrow\C$ so that $S\circ \eta_{i,S} = \id$ and $\eta_{i,S}(\infty) = z_i$. Thus $\{\eta_{i,S}:i=1,\ldots,d\} = \I(\id\on U_S<T)$.

\begin{figure}
\centerline{\mbox{\includegraphics[scale=1.2]{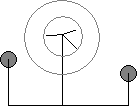}}}
\caption{The set $U_S\butnot B_e(\BP_T,\delta_1) = U\butnot B_e(\BP_T,\delta_1)$ is connected. Every connected component of $U_S\cap B_e(\BP_T,2\delta_1)$ intersects $U_S\butnot B_e(\BP_T,\delta_1)$.}
\label{figureUtwo}
\end{figure}

Fix $\delta_2 > 0$. By Lemma \ref{lemmacomplex}, there exists $\delta_1 > 0$ small enough so that for all $p,q\in B_e(\BP_T,2\delta_1)$, there exists a bijection $\Phi_{p,q}:T^{-1}(p)\rightarrow T^{-1}(q)$ satisfying (\ref{closetoid}) for all $x\in T^{-1}(p)$. From the construction of $U$, it is clear that if $\delta_1$ is sufficiently small, then $U\butnot B_e(\BP_T,\delta_1)$ is connected. Now assume that $\delta_0$ is small enough so that
\begin{equation*}
3\delta_0 \leq \inf_{p\in\C\butnot B_e(\BP_T,\delta_1)}\diam_d(T^{-1}(p))
\end{equation*}
(Note that $\delta_0$ now depends on $\delta_1$, and indirectly on $\delta_2$.) Fix $p\in U\butnot B_e(\BP_T,\delta_1)$, and let $\Phi_p:T^{-1}(p)\rightarrow S^{-1}(p)$ be the bijection guaranteed by Lemma \ref{lemmamorecomplex}. For all $x\in T^{-1}(p)$ and for all $y\in S^{-1}(p)$,
\begin{itemize}
\item If $y = \Phi_p(x)$, then $\dist_s(x,y)\leq \delta_0$
\item If $y \neq \Phi_p(x)$, then $\dist_s(x,y)\geq \dist_s(x,\Phi_p^{-1}(y)) - \delta_0 \geq 2\delta_0$
\end{itemize}
In particular, for $i=1,\ldots,d$, we have $\dist_s(\zeta_i\circ\xi_i(p),\eta_{i,S}(p))\notin (\delta_0,2\delta_0)$.

Now, if $\delta_0$ is sufficiently small, we have $B_s(\BP_T,\delta_0)\implies B_e(\BP_T,\delta_1)$, so that $U_S\butnot B_e(\BP_T,\delta_1) = U\butnot B_e(\BP_T,\delta_1)$; in particular, this set is connected. Furthermore
\[
\dist_s(\zeta_i\circ\xi_i(\infty),\eta_{i,S}(\infty)) = \dist_s(\zeta_i\circ\xi_i(\infty),\Phi_\infty(\zeta_i\circ\xi_i(\infty))) \leq \delta_0;
\]
thus we have
\begin{equation}
\label{zetaxieta}
\dist_s(\zeta_i\circ\xi_i(p),\eta_{i,S}(p)) \leq \delta_0
\end{equation}
for all $p\in U_S\butnot B_e(\BP_T,\delta_1)$.

From the construction of the set $U_S$, it is clear that each connected component $Q$ of $U_S\cap B_e(\BP_T,2\delta_1)$ intersects $U_S\butnot B_e(\BP_T,\delta_1)$, say at $p\in Q$. Let $q\in\BP_T$ be such that $\dist_e(p,q)\leq 2\delta_1$. Let $x = \Phi_{p,q}(\zeta_i\circ\xi_i(p))\in T^{-1}(q)$, so that by (\ref{closetoid}), we have $\dist_s(x,\zeta_i\circ\xi_i(p))\leq\delta_2$. Thus by (\ref{zetaxieta}), we have $\dist_s(x,\eta_{i,S}(p)) \leq \delta_0 + \delta_2$.

By a connectedness argument similar to the previous three paragraphs, it can be shown that if $\delta_2$, $\delta_0$ are sufficiently small, then $\dist_s(x,\zeta_i\circ\xi_i(p)) , \dist_s(x,\eta_{i,S}(p)) \leq \delta_0 + \delta_2$ for all $p\in Q$. Thus for all $p\in U_S$, we have
\begin{equation}
\label{zetaxietatwo}
\dist_s(\zeta_i\circ\xi_i(p),\eta_{i,S}(p)) \leq 2(\delta_0 + \delta_2)
\end{equation}
Let $\sigma > 0$ be the constant of uniform local injectivity for $\zeta_i$. Since $\zeta_i$ is an open mapping, we have that if $\delta_2$ and $\delta_0$ are sufficiently small, then for all $x\in\xi_i(U)\Kin\B$,
\[
B_s(\zeta_i(x),2(\delta_0 + \delta_2)) \implies \zeta_i(B_h(x,\sigma_4)).
\]
Combining with (\ref{zetaxietatwo}), we have that for all $p\in U_S$,
\begin{equation}
\label{etazetaxigamma}
\eta_{i,S}(p) \in \zeta_i(B_h(\xi_i(p),\sigma/4)).
\end{equation}
Since $\zeta_i$ is $\sigma$-locally injective, it is injective when restricted to the ball $B_h(\xi_i(p),\sigma/2)$. Let $\xi_{i,S}(p)$ be the unique inverse of $\eta_{i,S}(p)$ under $\zeta_i\on B_h(\xi_i(p),\sigma/2)$, so that $\eta_{i,S} = \zeta_i\circ\xi_{i,S}$.

We claim that the map $\xi_{i,S}$ is holomorphic, and in particular continuous. To this end, fix $p,q\in U_S$ so that
\begin{equation*}
\dist_h(\xi_i(p),\xi_i(q)) \leq \sigma/4\\
\end{equation*}
Then
\begin{equation*}
\zeta_i(\xi_{i,S}(q))
=\eta_{i,S}(q)  
\in \zeta_i(B_h(\xi_i(q),\sigma/4))
\implies \zeta_i(B_h(\xi_i(p),\sigma/2)).
\end{equation*}
By the injectivity of $\zeta_i\on B_h(\xi_i(p),\sigma/2)$, we have $\xi_{i,S}(q)\in B_h(\xi_i(p),\sigma/2)$. This is true for all $q$ sufficiently close to $p$. Thus the formula
\[
\xi_{i,S} = (\zeta_i)_{B_h(\xi_i(p),\sigma/2)}^{-1}\circ \eta_{i,S}
\]
holds in a neighborhood of $p$. Thus $\xi_{i,S}$ is holomorphic near $p$. Since $p\in U_S$ was arbitrary, $\xi_{i,S}$ is holomorphic.

We now show (iii) - (v) of Lemma \ref{lemmaspiderperturbation}:
\begin{itemize}
\item[iii)] $S\circ\zeta_i\circ\xi_{i,S} = S\circ\eta_{i,S} = \id$.
\item[iv)] $\diam_h(\xi_{i,S}(U_S))\leq C := \diam_h(\xi_i(U)) + \sigma/2$ which is finite and independent of $S$
\item[v)] For all $p\in U_S$,
\begin{equation*}
S^{-1}(p) = \{\eta(p): \eta\in\I(\id\on U_S,T)\} = \{\zeta_i\circ\xi_{i,S}(p): i=1,\ldots,d\}
\end{equation*}
\end{itemize}
Thus if $\delta_0$, $\delta_1$, and $\delta_2$ are sufficiently small, we are done.
\end{proof}

\end{section}

\begin{section}{Proof of Theorem \ref{theoremcondition}}
In this section, we fix $\AA\implies\RR\butnot\RR_1$ and $F\implies\C$ finite with $\AA(F)\implies F$ satisfying the hypotheses of Theorem \ref{theoremcondition}, i.e. for all $\ell\in\N$, for all $T\in\AA^\ell$, and for all $p\in\FP_T$,
\begin{itemize}
\item[i)] $p\in\RP_T$ implies $p\in F$
\item[ii)] $p\in F$ implies $T_*(p)<1$, i.e. $p$ is an attracting fixed point of $T$.
\end{itemize}
We furthermore fix $\tau < 1$.

Let $B\implies\C$ a neighborhood of $F$ and $\BBB_1\implies\RR$ a neighborhood of $\AA$ be given by Lemma \ref{lemmaB}. Let $K = X = \C\butnot B$. Let $r\in\N$ be as in Lemma \ref{lemmaR}. Let $\inc\in\N$, $\delta_3 > 0$, and $\BBB_2\implies\RR$ a neighborhood of $\AA$ be given by Lemma \ref{lemmaL}.

Fix $\ell\in\N$; $\ell$ will be defined later, and will depend only on $r$, $D := \max_\AA(\deg)$, and $\tau$.

Fix $T\in\AA^\ell$. By Lemma \ref{lemmaR}, $\mult_T(p)\leq r$ for all $p\in\C\butnot T^{-1}(F)$. Let $d_T = \deg(T)$; Lemma \ref{lemmaspiderperturbation} applies. The results are $C_T < \infty$, $\sigma_T>0$, $\BBB_T\implies\RR_{d_T}$ a neighborhood of $T$, and $(\zeta_i^{(T)})_{i=1}^{d_T}$ a collection of $\sigma$-locally injective holomorphic maps from $\B$ to $\C$ satisfying (\ref{multzetabounds}). Let
\begin{align*}
C_5 &= \max_{T\in\AA^\ell}C_T < \infty\\
\sigma &= \min_{T\in\AA^\ell}\sigma_T > 0\\
\BBB_\ell &= \bigcup_{T\in\AA^\ell}\BBB_T \implies \RR,
\end{align*}
so that $\BBB_\ell$ is a neighborhood of $\AA^\ell$. Let $\BBB \Kin \BBB_1\cap\BBB_2$ be a neighborhood of $\AA$ such that $\BBB^\ell \Kin \BBB_\ell$.

Fix $C_1 < \infty$ and $\alpha > 0$. Let
\[
H := \sup_\BBB(T_*) < \infty.
\]
Fix $M < \infty$ and $\gamma$ a modulus of continuity; $M$ and $\gamma$ will be defined later, and will depend only on
\[
\ell, \inc, C_1, C_5, r, D, H, \tau, \alpha, \sigma, \delta_3.
\]
(We call these ``the parameters''.) We claim that the conclusion of Theorem \ref{theoremcondition} is satisfied for $(M,\gamma,B,\BBB)$, i.e. we claim that if $(T_j)_{j\in\N}$ is a sequence of rational maps in $\BBB$, and if $(\phi_j)_{j\in\N}$ is a sequence of potential functions satisfying (\ref{C1isabound}) and (\ref{tauisabound}) for all $j\in\N$, then for all $n\in\N$ (\ref{Misaboundzero}) and (\ref{gammaisaboundzero}) hold.

To this end, suppose that $(T_j)_{j\in\N}$ is a sequence of rational maps in $\BBB$, and suppose that $(\phi_j)_{j\in\N}$ is a sequence of potential functions satisfying (\ref{C1isabound}) and (\ref{tauisabound}) for all $j\in\N$. Since $\BBB\implies\BBB_1$, we have $T_j(B)\implies B$ for all $j\in\N$; thus $(T_j)^{-1}(X)\implies X$.

Let $c = \tau^{1/3}$, so that $\tau < c^2 < 1$. Let $\beta = \infty$ and let $C_2 = 0$, so that (\ref{C1isabound}) and (\ref{tauisabound}) imply (\ref{C1isaboundmod}) and (\ref{tauisaboundmod}) for $j = 0,\ldots,n - 1$. By the definition of $D$ given above, (\ref{Disaboundmod}) holds. Thus we are well on our way towards being able to apply the inverse branch formalism (Definition \ref{definitionZAB}). However, the maps $(\zeta_i)_{i=1}^m$ will be different in the proof of (\ref{Misaboundzero}) and in the proof of (\ref{gammaisaboundzero}).

\begin{proof}[Proof of (\ref{Misaboundzero}):]

We will show (\ref{Misaboundzero}) by induction; we describe first the inductive step, since the base cases will become clear once $M$ is established. To this end, we fix $n\in\N$, and assume that (\ref{Misaboundzero}) holds for all $j = 0,\ldots,n - 1$. We will show that (\ref{Misaboundzero}) holds for $N := n + \ell$. This is a sort of ``jump induction''. The appropriate base cases for this type of induction are the cases $n = 0,\ldots,\ell - 1$.

Since $T_j\in\BBB$ for $j = n,\ldots,N - 1$, we have $T_n^N\in\BBB^\ell\implies\BBB_\ell$, so there exists $S_n\in\AA^\ell$ such that $T_n^N\in\BBB_{S_n}$. By Lemma \ref{lemmaspiderperturbation}, there is a simply connected hyperbolic open set $U_N\implies\C$ satisfying conditions (i) and (ii) of Lemma \ref{lemmaspiderperturbation}, with $S = T_n^N$. Let $m = d_{S_n} = \deg(T_n^N)$. For each $i = 1,\ldots,m$ there is a holomorphic map $\xi_i:U_N\rightarrow\B$ satisfying (iii) - (v) of Lemma \ref{lemmaspiderperturbation}, with $S = T_n^N$ and $C = C_5$.

For each $i=1,\ldots,m$, let $\zeta_i := \zeta_i^{(S_n)}$. We have now defined all objects required for the inverse branch formalism (Definition \ref{definitionZAB}). Apply this formalism to get the operator $A_0^n$.

By the inductive hypothesis, (\ref{Misaboundzero}) holds for $j=0,\ldots,n$; by the backwards invariance of $X$, it holds for $\wtilde{X} := T_n^j(X)$ as required in Corollary \ref{corollaryZAB}. Thus we may apply Corollary \ref{corollaryZAB}. Note that the $r$ in this corollary is not the same as our $r$, but by (\ref{multzetabounds}) we may substitute $r = 9r$ into (\ref{C5def}), and (\ref{ZAB}) becomes a true formula.

For each $i=1,\ldots,m$, we have that $\zeta_i$ is $\sigma$-locally injective, and so we may apply Lemma \ref{lemmahyperbolic}. Thus (\ref{Aislessthan}) holds for all $x,y\in\xi_i(U_N)$. Furthermore, for each $x,y\in\xi_i(U_N)$, (iv) of Lemma \ref{lemmaspiderperturbation} gives that $\dist_h(x,y)\leq C_5$. Thus if we let $g = \one$, then we have
\begin{equation}
\label{Aislessthannew}
\sup_{\xi_i(U_N)}(A_0^n[\one])_i \leq e^{C_4 C_5^\alpha} \inf_{\xi_i(U_N)}(A_0^n[\one])_i.
\end{equation}
By (v) of Lemma \ref{lemmaspiderperturbation}, we have $T_N^n(p) = \{\zeta_1(x_1),\ldots,\zeta_n(x_n)\}$ for all $p\in U_N$. Thus we can rewrite (\ref{Ldef}) on $U_N$:
\begin{equation*}
L_0^N[\one] = \sum_{i=1}^m\left(e^{\phi_n^N}L_0^n[\one]\right)\circ\zeta_i\circ\xi_i
\end{equation*}
\begin{comment}
The idea now is to compare $L_0^N[\one]$ with the expression
\[
C_0^N[\one] := \sum_{i=1}^m\left(e^{\phi_n^N\circ\zeta_i}(A_0^n[\one])_i\right)\circ\xi_i
\]
and to show that the latter is less variable. We see immediately that
\begin{align*}
L_0^N[\one] - C_0^N[\one] &= \sum_{i=1}^m\left(e^{\phi_n^N\circ\zeta_i}[L_0^n[\one]\circ\zeta_i - (A_0^n[\one])_i]\right)\circ\xi_i\\
\sup_{U_N\cap X}(L_0^N[\one] - C_0^N[\one]) &\leq \sum_{i=1}^m e^{\sup(\phi_n^N)}\sup_{\xi_i(U_N\cap X)}\left([L_0^n[\one]\circ\zeta_i - (A_0^n[\one])_i]\right)
\end{align*}
Now, the backwards invariance of $X$ implies that $\xi_i(U_N\cap X) \implies \zeta_i^{-1}(X)$. Combining with (\ref{ZAB}) [with $f = \one$] and (\ref{tauisabound}) yields
\begin{align*}
\sup_{U_N\cap X}(L_0^N[\one] - C_0^N[\one]) &\leq e^{\sup(\phi_n^N)} C_3 e^M\inf_X L_0^n[\one]\\
&\leq \tau^\ell\inf(L_n^N[\one])C_3 e^M\inf_X L_0^n[\one]\\
&\leq C_3 e^M\tau^\ell\inf_X(L_0^N[\one])
\end{align*}

DOTDOTDOT
\end{comment}
By (i) of Lemma \ref{lemmaspiderperturbation}, we have $\cl{U_N\cap X} = X$. Thus taking extrema over $U_N\cap X$ gives
\begin{align*}
\sup_X L_0^N[\one] &\leq \sum_{i=1}^m e^{\sup(\phi_n^N)}\sup_{\xi_i(U_N\cap X)}\left(L_0^n[\one]\circ\zeta_i\right)\\
\inf_X L_0^N[\one] &\geq \sum_{i=1}^m e^{\inf(\phi_n^N)}\inf_{\xi_i(U_N\cap X)}\left(L_0^n[\one]\circ\zeta_i\right)
\end{align*}
Multiplying (\ref{ZAB}) [with $f = \one$] by (\ref{tauisabound}) yields
\begin{align*}
\sum_{i=1}^m e^{\sup(\phi_n^N)}\sup_{\zeta_i^{-1}(X)}[L_0^n[f]\circ\zeta_i - (A_0^n[f])_i] &\leq  \tau^\ell\inf(L_n^N[\one])C_3 e^M\inf_X L_0^n[f]\\
&\leq C_3 e^M \tau^\ell\inf_X L_0^N[\one].
\end{align*}
Now, the backwards invariance of $X$ together with (iii) of Lemma \ref{lemmaspiderperturbation} imply that $\xi_i(U_N\cap X) \implies \zeta_i^{-1}(X)$. Thus we have
\begin{equation}
\label{leavingoffinadifferenttheorem}
\sup_X L_0^N[\one] \leq C_3 e^M \tau^\ell\inf_X L_0^N[\one] + \sum_{i=1}^m e^{\sup(\phi_n^N)}\sup_{\xi_i(U_N\cap X)}(A_0^n[\one])_i.
\end{equation}
We concentrate on the last term. By (\ref{C1isabound}) and (\ref{Aislessthannew}),
\begin{align*}
\sum_{i=1}^m e^{\sup(\phi_n^N)}\sup_{\xi_i(U_N\cap X)}(A_0^n[\one])_i
&\leq \sum_{i=1}^m e^{\ell C_1 (\pi/2)^\alpha}e^{\inf(\phi_n^N)}e^{C_4 C_5^\alpha}\inf_{\xi_i(U_N\cap X)}(A_0^n[\one])_i\\
&\leq e^{\ell C_1 (\pi/2)^\alpha + C_4 C_5^\alpha}\sum_{i=1}^m e^{\inf(\phi_n^N)}\inf_{\xi_i(U_N\cap X)}L_0^n[\one]\circ\zeta_i\\
&\leq e^{\ell C_1 (\pi/2)^\alpha + C_4 C_5^\alpha}\inf_X L_0^N[\one].
\end{align*}
Recombining with (\ref{leavingoffinadifferenttheorem}), we have
\[
\sup_X L_0^N[\one] \leq \left(C_3 e^M \tau^\ell + e^{\ell C_1 (\pi/2)^\alpha + C_4 C_5^\alpha}\right)\inf_X L_0^N[\one].
\]
Dividing both sides by $\inf_X(L_0^N[\one])$ yields
\begin{equation*}
e^{\|\ln(L_0^N[\one])\|_{\osc,X}} \leq C_3 e^M \tau^\ell + e^{\ell C_1 (\pi/2)^\alpha + C_4 C_5^\alpha}.
\end{equation*}
Rearranging yields
\[
e^{\|\ln(L_0^N[\one])\|_{\osc,X}} - \frac{e^{\ell C_1 (\pi/2)^\alpha + C_4 C_5^\alpha}}{1 - \tau^\ell C_3} \leq \tau^\ell C_3 \left[e^M - \frac{e^{\ell C_1 (\pi/2)^\alpha + C_4 C_5^\alpha}}{1 - \tau^\ell C_3}\right]
\]
If $\ell$ is sufficiently large, we have $\tau^\ell C_3 < 1$. In particular, this choice of $\ell$ can be made using only the information of $r$, $D$, and $\tau$. (In fact, $C_3$ depends only on the variables stated plus $c$, which depends only on $\tau$, and $C_2$ and $\beta$, which are in fact constants in this proof.) It is crucial here that $C_3$ does not depend on $C_5$ or on $\sigma$, since each of these depends indirectly on $\ell$.

Thus we have
\[
e^{\|\ln(L_0^N[\one])\|_{\osc,X}} - \frac{e^{\ell C_1 (\pi/2)^\alpha + C_4 C_5^\alpha}}{1 - \tau^\ell C_3} \leq \max\left(0,e^M - \frac{e^{\ell C_1 (\pi/2)^\alpha + C_4 C_5^\alpha}}{1 - \tau^\ell C_3}\right).
\]
Solving for $\|\ln(L_0^N[\one])\|_{\osc,X}$ yields
\begin{equation*}
\|\ln(L_0^N[\one])\|_{\osc,X}
\leq \max\left(\ell C_1 (\pi/2)^\alpha + C_4 C_5^\alpha - \ln(1 - \tau^\ell C_3),M\right)
\end{equation*}
Let
\[
M = \ell C_1 (\pi/2)^\alpha + C_4 C_5^\alpha - \ln(1 - \tau^\ell C_3),
\]
so that we have completed the inductive step. As promised, $M$ depends only on the parameters.

To prove the base case, we fix $n = 0,\ldots,\ell - 1$. Now
\begin{equation*}
\|\ln(L_0^n[\one])\|_\osc \leq \|\phi_0^n\|_\osc \leq \ell C_1 (\pi/2)^\alpha \leq M,
\end{equation*}
completing the proof of (\ref{Misaboundzero}).

\end{proof}
\begin{proof}[Proof of (\ref{gammaisaboundzero}):]

We next want to construct $\gamma$ satisfying (\ref{gammaisaboundzero}) for all $n\in\N$.

Fix $\varepsilon > 0$. We claim that there exists $\delta_\varepsilon > 0$ depending only on $\varepsilon$ and the parameters such that for all $n\in\N$,
\begin{equation}
\label{deltaepsilon}
\rho_{\ln(L_0^n[\one])}^{(X)}(\delta_\varepsilon)\leq \varepsilon.
\end{equation}
This claim will show that
\[
\gamma(\delta) := \inf\{\varepsilon:\delta_\varepsilon\geq\delta\}
\]
is a modulus of continuity. By construction, $\gamma$ satisfies (\ref{gammaisaboundzero}). Thus the proof of (\ref{deltaepsilon}) is sufficient to complete the proof of Theorem \ref{theoremcondition}.

\begin{claim}
\label{claimsigma}
There exists $\varepsilon_3 > 0$ depending only on $\inc$, $C_1$, and $D$ such that for all $k\in\N$, for all $n\in\N$, for all $a\in\C$ and for all $p\in X$, if
\begin{align*}
N &:= n + \inc k\\
\delta_k &:= H^{-k}\delta_3/2,
\end{align*}
then
\begin{equation}
\label{smallpart}
\sum_{x\in T_N^n(p)\cap B_s(a,\delta_k)}\frac{e^{\phi_n^N(x)}}{L_n^N[\one](p)} \leq (1 + \varepsilon_3)^{-k}.
\end{equation}
\end{claim}
\begin{proof}
By induction on $k$. $n$ and $a$ will remain fixed, so the inductive claim is that (\ref{smallpart}) holds for all $p\in X$.

Base case $k = 0$: This is clear from (\ref{Ldef}).

Inductive step: Assume that (\ref{smallpart}) holds for $k$, for all $p\in\Omega$. Let $\wtilde{k} = k + 1$, so that $\wtilde{N} = N + \inc$. Fix $p\in X$. By the backwards invariance of $X$, $T_{\wtilde{N}}^N(p)\implies X$, so by the inductive hypothesis (\ref{smallpart}) holds for all $x\in T_{\wtilde{N}}^N(p)$. Now
\[
T_n^N(B_s(a,\delta_{\wtilde{k}})) \implies B_s(T_n^N(a),\delta_3/2)
\]
which cannot contain every element of $T_{\wtilde{N}}^N(p)$, because of (\ref{equationL}).

We have
\begin{align*}
&\sum_{z\in T_{\wtilde{N}}^n(p)\cap B_s(a,\delta_{\wtilde{k}})}\frac{e^{\phi_n^{\wtilde{N}}(z)}}{L_n^{\wtilde{N}}[\one](p)}\\
&= \sum_{x\in T_{\wtilde{N}}^N(p)\cap B_s(T_n^N(a),\delta_3/2)}\frac{e^{\phi_N^{\wtilde{N}}(x)}L_n^N[\one](x)}{L_n^{\wtilde{N}}[\one](p)}\sum_{z\in T_N^n(x)\cap B_s(a,\delta_{\wtilde{k}})}\frac{e^{\phi_n^N(z)}}{L_n^N[\one](x)}\\
&\leq \sum_{x\in T_{\wtilde{N}}^N(p)\cap B_s(T_n^N(a),\delta_3/2)}\frac{e^{\phi_N^{\wtilde{N}}(x)}L_n^N[\one](x)}{L_n^{\wtilde{N}}[\one](p)}(1 + \varepsilon_3)^{-k}\\
&= (1 + \varepsilon_3)^{-k} \left(1 + \frac{\displaystyle\sum_{x\in T_{\wtilde{N}}^n(p)\butnot B_s(T_n^N(a),\delta_3/2)}e^{\phi_N^{\wtilde{N}}(x)}}{\displaystyle\sum_{x\in T_{\wtilde{N}}^N(p)\cap B_s(T_n^N(a),\delta_3/2)}e^{\phi_N^{\wtilde{N}}(x)}}\right)^{-1}\\
&\leq (1 + \varepsilon_3)^{-k} \left(1 + \frac{1}{D^\inc - 1}\frac{\inf(e^{\phi_N^{\wtilde{N}}})}{\sup(e^{\phi_N^{\wtilde{N}}})}\right)^{-1}\\
&\leq (1 + \varepsilon_3)^{-k} \left(1 + \frac{1}{D^\inc - 1}e^{-\inc C_1 (\pi/2)^\alpha}\right)^{-1},
\end{align*}
which yields (\ref{smallpart}) for $\wtilde{k}$ if we set
\[
\varepsilon_3 := \frac{1}{D^\inc - 1}e^{-\inc C_1 (\pi/2)^\alpha} > 0.
\]
As promised, $\varepsilon_3$ depends only on $\inc$, $C_1$, and $D$.
\QEDmod\end{proof}

Fix $k\in\N$; $k$ will be specified later and will depend only on $\varepsilon_3$, $C_3$, $M$, and $\varepsilon$.

Let $\{a_1,\ldots,a_m\}$ be a maximal $\delta_k$-separated subset of $\C$. For each $i=1,\ldots,m$, let
\[
U_i = B_s(a_i,2\delta_k),
\]
and let $\zeta_i:U_i\rightarrow\C$ be the identity map. Apply the inverse branch formalism (Definition \ref{definitionZAB}) to get the operator $A_0^n$.

Now (\ref{Misaboundzero}) has already been proven. As in the proof of (\ref{Misaboundzero}), the backwards invariance of $X$ implies that (\ref{Misaboundzero}) holds on the appropriate domain $\wtilde{X} = T_n^j(X)$ in order to apply Corollary \ref{corollaryZAB}. Thus we may apply Corollary \ref{corollaryZAB}. The value $r$ is no longer relevant to $\mult(\zeta_i)_{i=1}^m$; nonetheless we have the following bound:

\begin{claim}
\[
\mult(\zeta_i)_{i=1}^m \leq 25.
\]
\end{claim}
\begin{proof}
Fix $x\in\C$. For $i=1,\ldots,m$, if $x\in B_s(a_i,2\delta_k)$, then $B_s(a_i,0.5\delta_k)\implies B_s(x,2.5\delta_k)$. The collection $(B_s(a_i,0.5\delta_k))_i$ is disjoint, so we have
\begin{equation*}
\mult(B_s(a_i,2\delta_k))_{i=1}^m
\leq \frac{\lambda_s(B_s(0,2.5\delta_k))}{\lambda_s(B_s(0,0.5\delta_k))}
= \frac{\sin^2(2.5\delta_k)}{\sin^2(0.5\delta_k)}
< \frac{2.5^2}{0.5^2} = 25.
\end{equation*}
\QEDmod\end{proof}

Thus (\ref{ZAB}) is true with $r$ replaced by $25$.

Fix $i=1,\ldots,m$. Since $\zeta_i$ is injective, it is $\sigma$-locally injective with $\sigma = \infty$. Thus we may apply Lemma \ref{lemmahyperbolic}. Plugging $g = \one$ into (\ref{Aislessthan}) yields
\[
\rho_{\ln(A_0^n[\one])_i}(\varepsilon) \leq C_4 \varepsilon^\alpha.
\]
Let
\[
P_i = B_s(a_i,\delta_k)\butnot \bigcup_{j < i}B_s(a_j,\delta_k),
\]
so that $(P_i)_{i=1}^m$ is a partition of $\C$, and $P_i\implies B_s(a_i,\delta_k)$ for $i=1,\ldots,m$.

Fix $p\in X$. We weaken (\ref{smallpart}):
\begin{align*}
\sum_{x\in T_N^n(p)\cap P_i} e^{\phi_n^N(x)} &\leq (1 + \varepsilon_3)^{-k} L_n^N[\one](p)\\
\sum_{x\in T_N^n(p)\cap P_i} e^{\phi_n^N(x)}[L_0^n[\one](x) - (A_0^n[\one])_i(x)]
&\leq (1 + \varepsilon_3)^{-k} L_n^N[\one](p) \sup_{U_i\cap X}(L_0^n[\one] - (A_0^n[\one])_i)
\end{align*}
Summing over $i=1,\ldots,m$ and combining with (\ref{ZAB}) gives
\begin{align*}
L_0^N[\one](p) - \sum_{i=1}^m\sum_{x\in T_N^n(p)\cap P_i} e^{\phi_n^N(x)}(A_0^n[\one])_i(x)
&\leq (1 + \varepsilon_3)^{-k} L_n^N[\one](p) \sum_{i=1}^m \sup_{U_i\cap X}(L_0^n[\one] - (A_0^n[\one])_i)\\
&\leq (1 + \varepsilon_3)^{-k} L_n^N[\one](p) C_3 e^M \inf_X(L_0^n[\one])\\
&\leq (1 + \varepsilon_3)^{-k} C_3 e^M L_0^N[\one](p).
\end{align*}
Let
\[
k = \left\lceil\frac{\ln(C_3 e^M) - \ln(1 - e^{\varepsilon/2})}{\ln(1 + \varepsilon_3)}\right\rceil \in \N,
\]
so that solving for $L_0^N[\one](p)$ yields
\begin{equation}
\label{algebraisfun}
L_0^N[\one](p) \leq e^{\varepsilon/2}\sum_{i=1}^m\sum_{x\in T_N^n(p)\cap P_i} e^{\phi_n^N(x)}(A_0^n[\one])_i(x).
\end{equation}
Fix $\delta_2 > 0$; $\delta_2$ will depend only on $\varepsilon$ and the parameters. Let $\delta_\varepsilon > 0$ be the constant guaranteed by Lemma \ref{lemmacomplex} for $D^{\inc k}$, $H^{\inc k}$, and $\delta_2$; as promised, $\delta_\varepsilon$ depends only on $\varepsilon$ and the parameters. Fix $q\in X$ with $\dist_s(p,q)\leq\delta_\varepsilon$. We will be done if
\begin{equation}
\label{proofcomplete}
\frac{L_0^N[\one](p)}{L_0^N[\one](q)} \leq e^\varepsilon.
\end{equation}
Let $\Phi:T_N^n(p)\rightarrow T_N^n(q)$ be the bijection guaranteed by Lemma \ref{lemmacomplex}.

We introduce the following notational convention: By a subscript of $\IXY$, we mean that the sum, maximum, or minimum is to be taken over all $i=1,\ldots,m$ and over all $x\in T_N^n(p)\cap P_i$. For shorthand we write $y := \Phi(x)$. Thus (\ref{algebraisfun}) becomes
\[
L_0^N[\one](p) \leq e^{\varepsilon/2}\sum_\IXY e^{\phi_n^N(x)}(A_0^n[\one])_i(x).
\]
Now clearly
\[
L_0^N[\one](q) \geq \sum_\IXY e^{\phi_n^N(y)}(A_0^n[\one])_i(y).
\]
We continue, applying (\ref{Aislessthan}):
\begin{align*}
\frac{L_0^N[\one](p)}{L_0^N[\one](q)}
&\leq e^{\varepsilon/2}\frac{\displaystyle\sum_\IXY e^{\phi_n^N(x)} (A_0^n[\one])_i(x)}{\displaystyle\sum_\IXY e^{\phi_n^N(y)} (A_0^n[\one])_i(y)}\\
&\leq e^{\varepsilon/2}\max_\IXY\frac{e^{\phi_n^N(x)} (A_0^n[\one])_i(x)}{e^{\phi_n^N(y)} (A_0^n[\one])_i(y)}\\
&\leq e^{\varepsilon/2}\exp\left(\max_\IXY[\phi_n^N(x) - \phi_n^N(y)] + C_4 \max_\IXY\dist_{U_i}(x,y)^\alpha\right)
\end{align*}
Fix $\IXY$. We have
\begin{align}
\dist_s(x,y) &\leq \delta_2 \\
\dist_s(a_i,x) &\leq \delta_k.
\end{align}
For each $j = n,\ldots,N - 1$, the Lipschitz continuity of $T_n^j$ gives
\begin{align*}
\dist_s(T_n^j(x),T_n^j(y)) &\leq H^{j - n}\delta_2\\
\phi_j(T_n^j(x)) - \phi_j(T_n^j(y)) &\leq C_1 (H^{j - n}\delta_2)^\alpha \\
\phi_n^N(x) - \phi_n^N(y) &\leq C_1 \delta_2^\alpha \sum_{j = n}^N H^{\alpha(j - n)}
= C_7\delta_2^\alpha,
\end{align*}
where
\[
C_7 := C_1 \sum_{j = 0}^{\inc k} H^{j\alpha}.
\]
We now need to bound $\dist_{U_i}(x,y)$ in terms of $\delta_2$. Without loss of generality suppose that $\delta_2 \leq \delta_k/2$, so that
\[
\dist_s(a_i,y) \leq 3\delta_k/2.
\]
Without loss of generality suppose that $\delta_k\leq 1/4$. A simple calculation shows that
\[
\|\id\|_{h_{U_i}}^s \geq \frac{1 - (c_1/c_2)^2}{1 + c_1^2},\hspace{.5 in}[\text{on }B_s(a_i,3\delta_k/2)]
\]
where
\begin{align*}
c_1 &:= \tan(3\delta_k/2)\\
c_2 &:= \tan(2\delta_k)
\end{align*}
Further calculation shows that
\[
\|\id\|_{h_{U_i}}^s \geq \frac{1 - (3/4)^2}{1 + 1^2} = \frac{7}{32}.\hspace{.5 in}[\text{on }B_s(a_i,3\delta_k/2)]
\]
(Here we have used the upper convexity of the tangent function.)

Integrating along a hyperbolic geodesic gives
\begin{equation}
\dist_{U_i}(x,y) \leq \frac{32}{7}\dist_s(x,y) \leq \frac{32}{7}\delta_2.
\end{equation}
Thus
\begin{align*}
\frac{L_0^N[\one](p)}{L_0^N[\one](q)}
&\leq e^{\varepsilon/2}\exp\left(\max_\IXY[\phi_n^N(x) - \phi_n^N(y)] + C_4 \max_\IXY\dist_{U_i}(x,y)^\alpha\right)\\
&\leq e^{\varepsilon/2}\exp((C_7 + C_4 (32/7)^\alpha)\delta_2^\alpha)
\end{align*}
We set the right hand side equal to $e^\varepsilon$, and solve for $\delta_2$. Since $\delta_2$ depends only on $\varepsilon$ and the parameters, the proof is complete.
\end{proof}
\end{section}
\begin{section}{Proof of Theorem \ref{theoremmanepartition}} \label{sectionmanepartition}
\begin{remark}
This proof is essentially based off of the proof in the deterministic case, given by Ma\~n\'e [\cite{Ma1}, Lemma II.4 p.33]. The biggest difference is the following: In the proof of Ma\~n\'e's Lemma II.5 [p.37], which gives a lower bound for the derivative of a point in terms of its distance from the critical points, Ma\~n\'e uses the compactness of $\C$ to get a non-explicit lower bound. However, this would not suffice for our purposes, because we need the bound to not only have the right behavior near critical points, but also to be uniform in a sense across all possible rational maps $T$. This is made explicit in Lemma \ref{lemmaC7C8} below.
\end{remark}
To prove Theorem \ref{theoremmanepartition}, we need several lemmas:
\begin{lemma}
\label{lemmaC7C8}
Fix $D<\infty$. Then there exists a constant $C_8 < \infty$ depending only on $D$ such that for all $T\in\RR_d$ with $1\leq d\leq D$ and for all $x\in\C$, we have
\begin{equation}
\label{C7C8}
h_0\geq \frac{\delta^{2D}}{C_8 H^{4D}},
\end{equation}
where
\begin{align} \label{hdef}
h_0 &:= T_*(x)\\ \label{deltadef}
\delta &:= \dist_s(x,\RP_T)\\ \label{Hdef}
H &:= \sup(T_*).
\end{align}
\end{lemma}
\begin{proof}
Fix $T\in\RR_d$ and $x\in\C$, and let $h_0$, $\delta$ and $H$ be as defined in (\ref{hdef}) - (\ref{Hdef}). By composing with a spherical isometry, we may without loss of generality suppose that $x = T(x) = 0$; this does not change the value of $h_0$, $\delta$, $H$, and therefore does not change the truth value of (\ref{C7C8}).

Write $T = f/g$ the quotient of two polynomial functions. Since $T(0) = 0$, we have $f(0) = 0$; we may without loss of generality suppose that $g(0) = 1$. We can write $g$ in the form
\begin{equation}
\label{multiplicativerepresentationg}
g(z) = \prod_{i = 1}^d \left(1 - \frac{z}{\beta_i}\right),
\end{equation}
where $\beta_1,\ldots,\beta_d$ are the roots of $g$ (possibly with repetition).

Let $h = f'g - g'f$, so that
\begin{align*}
m := \deg(h) &\leq 2d - 2\\
T_*(z) &= |h(z)|\frac{1 + |z|^2}{|f(z)|^2 + |g(z)|^2}\\
h_0 = T_*(0) &= |h(0)|.
\end{align*}
We can write $h$ in the form
\begin{equation}
\label{multiplicativerepresentationh}
h(z) = h(0)\prod_{i = 1}^m \left(1 - \frac{z}{\gamma_i}\right),
\end{equation}
where $\gamma_1,\ldots,\gamma_m$ are the roots of $h$ (possibly with repetition).

Note that $\gamma_1,\ldots,\gamma_m$ are also the ramification points of $T$ (other than $\infty$). Thus by (\ref{deltadef}), we have $\tan(\delta)\leq\gamma_i$ for $i=1,\ldots,m$.

Fix $\varepsilon > 0$, and let
\[
A := B_e(0,1/\varepsilon)\butnot\bigcup_{i=1}^d B_e(\beta_i,\varepsilon\beta_i).
\]
For all $z\in A$, we have the following bound for $T_*(z)$:
\begin{equation*}
T_*(z)
\leq h_0\prod_{i = 1}^m \left|1 - \frac{z}{\gamma_i}\right|\frac{1 + 1/\varepsilon^2}{\displaystyle\prod_{i = 1}^d\left|1 - \dfrac{z}{\beta_i}\right|^2}
\leq h_0\left(1 + \frac{1/\varepsilon}{\tan(\delta)}\right)^m\frac{1 + 1/\varepsilon^2}{\varepsilon^{2d}}.
\end{equation*}
Let us assume now that $\varepsilon\leq 1/2$. By definition, $\delta \leq \diam(\C) = \pi/2$. Calculus gives $\delta\leq\tan(\delta)$. Thus
\begin{align*}
T_*(z)
&\leq h_0\left(\frac{\pi/2}{\delta} + \frac{1/\varepsilon}{\delta}\right)^m\frac{1 + 1/\varepsilon^2}{\varepsilon^{2d}}\\
&\leq 1.25(\pi/4 + 1)h_0\left(\frac{1}{\delta\varepsilon}\right)^m\frac{1}{\varepsilon^{2d + 2}}\\
&\leq 3\frac{h_0}{\delta^m\varepsilon^{2d + m + 2}}\\
&\leq 3(\pi/2)^{2D}\frac{h_0}{\delta^{2D}\varepsilon^{4D}}
\end{align*}
Now by the change of variables formula,
\begin{align*}
d &= \int T_*(z)^2\d\lambda_s(z)\\
&= \int_A T_*(z)^2\d\lambda_s(z) + \int_{\C\butnot A} T_*(z)^2\d\lambda_s(z)\\
&\leq \sup_A(T_*(z))^2\lambda_s(\C) + \sup(T_*(z))^2\lambda_s(\C\butnot A)\\
&\leq 9(\pi/2)^{4D}\frac{h_0^2}{\delta^{4D}\varepsilon^{8D}} + H^2\lambda_s(\C\butnot A).
\end{align*}
We concentrate on this last term:
\begin{align*}
\lambda_s(\C\butnot A)
&\leq \lambda_s(\C\butnot B_e(0,1/\varepsilon)) + \sum_{i = 1}^d\lambda_s(B_e(\beta_i,\varepsilon\beta_i))\\
\lambda_s(\C\butnot B_e(0,1/\varepsilon))
&= \lambda_s(B_e(0,\varepsilon))\\
&= 1 - \frac{1}{1 + \varepsilon^2} \leq \varepsilon^2.
\end{align*}
For each $i = 1,\ldots,d$, consider the map $Q_i:B_e(1,\varepsilon)\rightarrow B_e(\beta_i,\varepsilon\beta_i)$ defined by $Q_i(z) = \beta_i z$. We have
\begin{align*}
\|(Q_i)_*(z)\|_e^s
&= |\beta_i|\frac{1}{1 + |\beta_i z|^2}\\
&\leq \frac{1}{2|z|} \leq \frac{1}{2(1 - \varepsilon)} \leq 1\\
\lambda_s(B_e(\beta_i,\varepsilon\beta_i))
&= \int_{B_e(1,\varepsilon)} [\|(Q_i)_*(z)\|_e^s]^2\d\lambda_e(z)\\
&\leq \lambda_e(B_e(1,\varepsilon)) = \pi\varepsilon^2
\end{align*}
Thus
\begin{equation*}
\lambda_s(\C\butnot A)
\leq (1 + d\pi)\varepsilon^2
\leq 2d\pi\varepsilon^2.
\end{equation*}
Let
\[
\varepsilon = \frac{1}{H\sqrt{4\pi}} \leq 1/2,
\]
so that
\begin{align*}
d &\leq 9(\pi/2)^{4D}\frac{h_0^2}{\delta^{4D}\varepsilon^{8D}} + H^2 (2d\pi\varepsilon^2)\\
&= 9(2\pi^2)^{4D}\frac{h_0^2 H^{8D}}{\delta^{4D}} + d/2\\
1/2 \leq d/2
&\leq 9(16\pi^8)^D\frac{h_0^2 H^{8D}}{\delta^{4D}}\\
\end{align*}
Rearranging yields (\ref{C7C8}).
\end{proof}

\begin{lemma}
\label{lemmapartitions}
There exists a sequence of partitions $(\Par_k)_{k\in\N}$ of $\C$ such that for all $k\in\N$,
\begin{enumerate}[A)]
\item $\diam(\Par_k)\leq 2^{-k}$
\item For all $x\in \C$ and for all $\delta > 0$,
\begin{equation}
\label{lnnumberbound}
\ln\#(\Par_k\on B_s(x,\delta))\leq 6\ln(2) + 2\max(0,\ln(\delta/2^{-(k-1)})).
\end{equation}
\end{enumerate}
\end{lemma}
\begin{proof}
Fix $k\in\N$. Let $(x_i)_{i = 1}^m$ be a maximal $2^{-(k + 1)}$-separated sequence in $\C$. Let
\[
\Par_k = \left\{B(x_i,2^{-(k + 1)})\butnot \bigcup_{j < i}B(x_i,2^{-(k + 1)}):i=1,\ldots,m\right\};
\]
(A) follows easily.

Now,
\begin{align*}
\ln\#(\Par_k\on B_s(x,\delta))
&\leq \ln\#(i = 1,\ldots,m: B_s(x_i,2^{-(k + 2)}) \in B_s(x,\delta + 2^{-k}))\\
&\leq \ln(\lambda_s(B_s(0,\delta + 2^{-k}))) - \ln(\lambda_s(B_s(0,2^{-(k + 2)}))),
\end{align*}
since the balls $(B_s(x_i,2^{-(k + 2)}))_{i = 1}^m$ are disjoint, as $(x_i)_{i = 1}^m$ is $2^{-(k + 1)}$-separated. We continue:
\begin{align*}
\ln\#(\Par_k\on B_s(x,\delta))
&\leq 2\ln\sin(\delta + 2^{-k}) - 2\ln\sin(2^{-(k + 2)})\\
&\leq 2\ln(\delta + 2^{-k}) - 2\ln(2^{-(k + 2)})\\
&= 4\ln(2) + 2\ln(1 + \delta/2^{-k})\\
&\leq 6\ln(2) + 2\max(0,\ln(\delta/2^{-k})).
\end{align*}
\end{proof}

\begin{lemma}
\label{lemmaUVW}
Fix $0 < h \leq H < \infty$ and $\delta \leq \pi/4$. Suppose that $T\in\RR$ satisfies
\begin{align}
\label{hHbounds}
T_* &\leq H \\
T_* &\geq h. \hspace{.5 in}[\text{on }B_s(a,\delta)]
\end{align}
Let $W = B_s(a,h\delta/H)$. Then $T\on W$ is injective; furthermore, for all $x,y\in W$,
\begin{equation}
\label{UVW}
\dist_s(T(x),T(y)) \geq \dist_s(x,y)\left[T_*(x) - \frac{\wtilde{H}}{2}\dist_s(x,y)\right],
\end{equation}
where $\wtilde{H}$ is as in Lemma \ref{lemmaC10}.
\end{lemma}

\begin{figure}
\centerline{\mbox{\includegraphics[scale=.5]{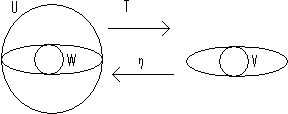}}}
\caption{Construction of the set $W$.}
\label{figureUVW}
\end{figure}

\begin{proof}
Let $U = B_s(a,\delta)$. Let $\delta_2 > 0$ be the largest number such that there exists an inverse branch $\eta$ of $T$ on $B_s(T(a),\delta_2)$ sending $T(a)$ to $a$ such that $\eta(B_s(T(a),\delta_2)) \implies U$. Since $\delta_1$ is maximal, there exists $p\in \del B_s(T(a),\delta_2)$ such that $\eta$ cannot be extended to any neighborhood of $p$. We claim that $\delta_2 \geq h\delta$. By contradiction, suppose otherwise; note that by the inverse chain rule (\ref{hHbounds}) becomes
\begin{equation}
\label{Hhbounds}
\frac{1}{H} \leq \eta_* \leq \frac{1}{h}.
\end{equation}
Thus $\eta$ is Lipschitz continuous with a corresponding constant of $1/h$; we have
\[
\eta(B_s(T(a),\delta_2)) \Kin U.
\]
This implies that there exists $x\in \cl{\eta(B_s(T(a),\delta_2))} \implies U$ such that $T(x) = p$.

In particular, since $x\in U$, we have $T_*(x) \geq h > 0$, so $T$ is injective on some neighborhood $B_s(x,\varepsilon)$ of $x$. Let $\varepsilon_2 > 0$ be small enough so that
\[
B_s(p,\varepsilon_2) \implies T(B_s(x,\varepsilon)).
\]
Then there exists
\[
\wtilde{\eta}: B_s(p,\varepsilon_2) \rightarrow B_s(x,\varepsilon) \implies U
\]
an inverse branch of $T$ such that $\wtilde{\eta}(p) = x$. Since $x\in\cl{\eta(B_s(T(a),\delta_2))}$, it follows that $\eta$ and $\wtilde{\eta}$ agree in a neighborhood of some point. Basic spherical geometry shows that the set
\[
B_s(T(a),\delta_2) \cap B_s(p,\varepsilon_2)
\]
is connected. Thus $\eta$ and $\wtilde{\eta}$ agree on this intersection, so $\eta$ can be extended, contradicting that $\delta_2$ is maximal.

Thus $\delta_2 \geq h\delta$. Thus there exists an inverse branch of $T$
\[
\eta: V := B_s(T(a),h \delta) \rightarrow U
\]
such that $\eta(T(a)) = a$. But by (\ref{Hhbounds}), the exact same argument can be applied to $\eta$, yielding an inverse branch of $\eta$
\[
\zeta: W := B_s\left(\eta(T(a)),\frac{1}{H}h\delta\right) \rightarrow V
\]
such that $\zeta(\eta(T(a))) = T(a)$. But an inverse branch of an inverse branch of $T$ must just be $T$, so $\zeta = T\on W$. Since $\zeta$ is an inverse branch of $\eta$, $\zeta$ is injective. Fix $x,y\in W$; we will show (\ref{UVW}). Now $T(x),T(y)\in V$ with $\eta(T(x)) = x$, $\eta(T(y)) = y$. Let $\gamma:[0,\dist_s(T(x),T(y))]\rightarrow V$ be the geodesic connecting $T(x)$ and $T(y)$, parameterized at unit speed. Then $\eta\circ\gamma$ connects $x$ with $y$. Define $f:[0,\dist_s(T(x),T(y))]\rightarrow[0,\pi/2]$ by
\[
f(t) := \dist(\eta\circ\gamma(t),x).
\]
Then
\begin{align*}
f(0) &= 0\\
f(\dist_s(T(x),T(y))) &= \dist_s(x,y).
\end{align*}
We have
\begin{align*}
f'(t)
&\leq \|(\eta\circ\gamma)_*(t)\|_e^s\\
&= \eta_*\circ\gamma(t)\\
&= \frac{1}{T_*(\eta\circ\gamma(t))}\\
&\leq \frac{1}{(T_*(x) - \wtilde{H} f(t))_+},
\end{align*}
the last inequality coming from Lemma \ref{lemmaC10}. Let
\[
\Psi(r) = \begin{cases}
T_*(x)r - \dfrac{\wtilde{H}}{2}r^2 & r \leq \dfrac{T_*(x)}{\wtilde{H}} \\
\dfrac{T_*(x)^2}{2\wtilde{H}} & r > \dfrac{T_*(x)}{\wtilde{H}}
\end{cases}
\]
Note that $\Psi$ is $C^1$, nondecreasing, and that $\Psi(0) = 0$.

To show (\ref{UVW}), it clearly suffices to show that
\begin{equation}
\label{ODEs}
t \geq \Psi(f(t))
\end{equation}
for all $t\in[0,\dist_s(T(x),T(y))]$. We already know that this inequality holds at at least one point, namely $t = 0$. If $f(t) < \frac{T_*(x)}{\wtilde{H}}$, the chain rule gives
\[
(\Psi\circ f)'(t) \leq \frac{T_*(x) - \wtilde{H} f(t)}{T_*(x) - \wtilde{H} f(t)} = 1;
\]
if $f(t) \geq \frac{T_*(x)}{\wtilde{H}}$, it gives
\[
(\Psi\circ f)'(t) = 0 \leq 1.
\]
The mean value inequality yields (\ref{ODEs}).
\end{proof}

\begin{proof}[Proof of Theorem \ref{theoremmanepartition}]
By the random Ruelle inequality [\cite{BB} Theorem 1, p.248],\footnote{There is a minor error in the referenced paper. On page 250, the statement ``Without loss of generality, ...'' is incorrect because this change would significantly affect the bounds from Lemma 3 [p.249]. One solution is to generalize Lemma 1(b) [p.247] to the case where the sequence of partitions is not assumed to be increasing. This is not too difficult; for the proof in the deterministic case see [\cite{PU} Corollary 1.8.10, p.57].}

\begin{equation*}
0 < h_\sigma(\T\on\theta) \leq 2\max\left(0,\int\ln(T_*(x))\d\sigma(\omega,x)\right),
\end{equation*}
and thus
\begin{equation}
\label{ruelle}
\int\ln(T_*(x))\d\sigma(\omega,x) > 0.
\end{equation}
Thus by the Birkhoff ergodic theorem, there exists $\varepsilon > 0$ such that $\sigma(\X\butnot A_\varepsilon) = 0$, where
\begin{equation}
\label{Aepsilondef}
A_\varepsilon := \left\{(\omega,x)\in\X:\sum_{j = 0}^{n - 1}\ln(T_j)_*\circ T_0^j(x) - n\varepsilon \tendston \infty\right\}.
\end{equation}
Note than $A_\varepsilon\cap \RP_T = \emptyset$. (By abuse of notation, we use $\RP_T$ to mean the set $\{(\omega,x):x\in\RP_{T_\omega}\}$.)

Let $(\Par_k)_{k\in\N}$ be as in Lemma \ref{lemmapartitions}. Fix $\omega\in\Omega$. For each $k\in\N$, let
\begin{equation}
\label{deltakdef}
\delta_k := \max\left(2\left(2^{-k}C_8 H^{4D+1}\right)^{1/(2D+1)},\left(2^{-k + 1}\frac{C_8 H^{4D}\wtilde{H}}{1 - e^{-\varepsilon}}\right)^{1/(2D)}\right),
\end{equation}
and let
\[
\delta_{-1} := \max(\delta_0,\pi/2).
\]
Then $(\delta_k)_{k\geq -1}$ is a nonincreasing sequence whose limit is zero. Furthermore, for $k\in\N$
\begin{align} \label{2kislessthan1}
2^{-k} &\leq 2\frac{1 - e^{-\varepsilon}}{\wtilde{H}}\frac{\delta_k^{2D}}{C_8 H^{4D}}\\ \label{2kislessthan2}
2^{-k} &\leq \frac{(\delta_k/2)^{2D+1}}{C_8 H^{4D + 1}}.
\end{align}
Let
\begin{equation}
\label{equationbk}
B_k := \cl{B_s}(\RP_T,\delta_k),% see appendix
\end{equation}
so that
\begin{align} \label{pardefone}
\ParB &:= \{B_{k-1}\butnot B_k:k\in\N\}\cup\{\RP_T\}\\ \label{pardeftwo}
\Par &:= \bigcup_{k\in\N}\Omega\times\Par_k\on(B_{k - 1}\butnot B_k)\cup\{\RP_T\}% see appendix
\end{align}
are partitions of $\X$, with $\Par$ a refinement of $\ParB$. Lt $B_{0,k}$, $\ParB_0$, and $\Par_0$ denote the $\omega$th fibers of $B_k$, $\ParB$, and $\Par$, respectively.
\begin{claim}
\label{claimTxTy}
Fix $\omega\in\Omega$. If $x,y\in\C$ with $y\in\Par_\omega(x)$, then
\begin{equation}
\label{TxTy}
\dist_s(T(x),T(y)) \geq \dist_s(x,y)e^{-\varepsilon}T_*(x).
\end{equation}
\end{claim}
\begin{proof}
If $x = y \in\RP_T$, (\ref{TxTy}) is trivial. Otherwise, there exists $k\in\N$ such that $x,y\in B_{0,k-1}\butnot B_{0,k}$ and $y\in\Par_k(x)$. Now (A) of Lemma \ref{lemmapartitions} gives
\begin{equation}
\label{dsxy2k}
\dist_s(x,y)\leq 2^{-k},
\end{equation}
and (\ref{equationbk}) gives
\begin{equation}
\dist_s(x,\RP_T)\geq\delta_k.
\end{equation}
Thus (\ref{C7C8}) gives
\[
\frac{(\delta_k/2)^{2D}}{C_8 H^{4D}} \leq T_* \leq H; \hspace{.5 in}[\text{on }B_s(x,\delta_k/2)]
\]
so that Lemma \ref{lemmaUVW} applies on the disk
\[
W := B_s\left(x,\frac{\delta_k}{2}\frac{\left(\dfrac{(\delta_k/2)^{2D}}{C_8 H^{4D}}\right)}{H}\right)
= B_s\left(x,\frac{(\delta_k/2)^{2D + 1}}{C_8 H^{4D + 1}}\right).
\]
By (\ref{2kislessthan2}) and (\ref{dsxy2k}), $y\in W$. Thus we have (\ref{UVW}), which yields (\ref{TxTy}) as long as
\[
T_*(x) - \frac{\wtilde{H}}{2}\dist_s(x,y) \geq e^{-\varepsilon}T_*(x),
\]
or equivalently
\[
\dist_s(x,y) \leq 2T_*(x)\frac{1 - e^{-\varepsilon}}{\wtilde{H}}.
\]
But this follows from (\ref{dsxy2k}), (\ref{2kislessthan1}), and (\ref{C7C8}).
\QEDmod\end{proof}

Suppose that $\sigma\in\M_e(\X,\T,\basemeasure)$ with $h_\sigma(\T\on\theta) > 0$. We are done if (A) and (B) of Theorem \ref{theoremmanepartition} hold.

To show (A), we will show that
\begin{align} \label{entropyB}
H_\sigma(\ParB\on\pi^{-1}\epsilon_\Omega) &< \infty\\ \label{entropypar}
H_\sigma(\Par\on\ParB\vee\pi^{-1}\epsilon_\Omega) &< \infty.
\end{align}
Then by a well-known formula for conditional entropy,
\[
H_\sigma(\Par\on\pi^{-1}\epsilon_\Omega) = H_\sigma(\Par\on\ParB\vee\pi^{-1}\epsilon_\Omega) + H_\sigma(\ParB\on\pi^{-1}\epsilon_\Omega) < \infty.
\]
To show (\ref{entropyB}), define the map $k:\X\rightarrow\N$ by
\[
k(\omega,x) :=
\begin{cases}
\min(k\in\N:\dist_s(x,\RP_T)\geq \delta_k) & \text{if $x\notin\RP_T$}\\
\infty & \text{if $x\in\RP_T$}
\end{cases}
\]
Note that $\ParB = k^{-1}\epsilon_{\what{\N}}$.
\begin{claim}
\begin{equation}
\label{kboundintegral}
\int k\d\sigma < \infty.
\end{equation}
\end{claim}
\begin{proof}
By (\ref{deltakdef}), there exists $C_{10} < \infty$ depending only on $D$ and $\varepsilon$ such that
\begin{equation}
\label{deltaklessthan}
\delta_k \leq C_{10} H^3 2^{-k/(2D + 1)}
\end{equation}
for all $k\in\N$. Fix $(\omega,x)\in\X$; if we let
\[
\wtilde{k} := \left\lceil -(2D + 1)\frac{1}{\ln(2)}\ln\left(\frac{\dist_s(x,\RP_T)}{C_{10}H^3}\right)\right\rceil \in \N,
\]
then algebra shows $\dist_s(x,\RP_T)\geq \delta_{\wtilde{k}}$; thus
\begin{equation}
\label{kbound}
k(\omega,x) \leq \wtilde{k}
\leq 1 + (2D + 1)\frac{1}{\ln(2)}\left[\ln(C_{10}) + 3\ln(H) - \ln(\dist_s(x,\RP_T))\right].
\end{equation}
Now, Lemma \ref{lemmaC10} gives
\[
(T_\omega)_*(x) \leq \wtilde{H_\omega}\dist_s(x,\RP_{T_\omega}).
\]
Taking logs, integrating against $\d\sigma(\omega,x)$, and combining with (\ref{ruelle}) and (\ref{lnH}) yields
\[
\int\ln\dist_s(x,\RP_{T_\omega})\d\sigma(\omega,x) > -\infty.
\]
Combining with (\ref{kbound}) and (\ref{lnH}) yields (\ref{kboundintegral}).
\QEDmod\end{proof}

By an elementary calculation [\cite{PU} Lemma 10.3.1, p.314], it follows that $H_\sigma(\ParB) = H_\sigma(k^{-1}\epsilon_{\what{\N}}) < \infty$; we have shown (\ref{entropyB}).

To show (\ref{entropypar}), first note that
\[
H_\sigma(\Par\on\ParB\vee\pi^{-1}\epsilon_\Omega) = \int H_\sigma(\Par_\omega \on B_{\omega,k(\omega,x) - 1}\butnot B_{\omega,k(\omega,x)}) \d\sigma(\omega,x).
\]
Fix $k\in\N$ and $\omega\in\Omega$. Now
\[
H_\sigma(\Par_\omega \on B_{\omega,k - 1}\butnot B_{\omega,k})
\leq \ln\#(\Par_\omega \on B_{\omega,k - 1}\butnot B_{\omega,k})
= \ln\#(\Par_k \on B_{\omega,k - 1}\butnot B_{\omega,k}).
\]
Combining with (\ref{equationbk}) and (\ref{deltaklessthan}),
\[
H_\sigma(\Par_\omega \on B_{\omega,k - 1}\butnot B_{\omega,k})
\leq \ln\left(\sum_{p\in\RP_T}\#\left[\Par_k \on B_s(p,C_{10} H^3 2^{-(k - 1)/(2D + 1)})\right]\right);
\]
combining with (\ref{lnnumberbound}),
\begin{align*}
&H_\sigma(\Par_\omega \on B_{\omega,k - 1}\butnot B_{\omega,k})\\
&\leq \ln\#(\RP_T) + 6\ln(2) + 2\max\left(0,\ln\left(\frac{C_{10} H^3 2^{-(k - 1)/(2D + 1)}}{2^{-(k-1)}}\right)\right) \\
&\leq \ln(2D - 2) + 6\ln(2) + 2\max\left(0,\ln(C_{10}) + 3\ln(H) + \ln(2)\left(1 - \frac{1}{2D + 1}\right)(k - 1)\right),
\end{align*}
which is integrable by (\ref{lnH}) and (\ref{kboundintegral}). Thus we have shown (\ref{entropypar}), completing the proof of (A).

To show (B), let $A_\varepsilon\implies\X$ be defined by (\ref{Aepsilondef}). We claim that (\ref{partitionsequal}) holds on $A_\varepsilon$. To this end, fix $(\omega,x)\in A_\varepsilon$; we must show that
\[
\left(\bigvee_{j\in\N} \T^{-j}\Par \vee \pi^{-1}\epsilon_\Omega\right)(\omega,x)
= \epsilon_\X(\omega,x),
\]
i.e.
\[
\bigvee_{j\in\N} T_j^0\Par_j(T_0^j(x)) = \{x\}.
\]
By contradiction, suppose that $y\in\bigvee_{j\in\N} T_j^0\Par_j(T_0^j(x))\butnot\{x\}$. Then for all $j\in\N$, $T_0^j(y)\in\Par_j(T_0^j(x))$. By Claim \ref{claimTxTy},
\[
\dist_s(T_0^{j+1}(x),T_0^{j+1}(y)) \geq \dist_s(T_0^j(x),T_0^j(y))e^{-\varepsilon}(T_j)_*\circ T_0^j(x);
\]
iterating yields
\[
\dist_s(T_0^n(x),T_0^n(y)) \geq \dist_s(x,y)\exp\left(\sum_{j = 0}^{n - 1}\ln (T_j)_*\circ T_0^j(x) - n\varepsilon\right) \tendston \infty,
\]
which is a contradiction since $\diam(\C) = \pi/2 < \infty$. Thus $\Par$ $\sigma$-almost generates $\X$ over $\Omega$.
\end{proof}
\end{section}
\begin{section}{Proof of Theorem \ref{theoremequilibrium}} \label{sectionequilibrium}
Let $\points_\Omega$ and $\points_\X$ be the partition of $\Omega$ and $\X$ into points, respectively.

Fix $\sigma\in \M_e(\X,\T,\basemeasure)$. We will show that
\begin{equation}
\label{varprinciplelessthan}
h_\sigma(\T\on\theta) + \int \phi\d\sigma \leq \int \ln(\lambda)\d\basemeasure,
\end{equation}
with equality if and only if $\sigma = \mu$.

Since $\Omega\times B,\SS\implies\X$ are forward invariant, the ergodicity of $\sigma$ implies that each of them has measure zero or one. If $\SS$ has measure one, then $h_\sigma(\points_\X\on\pi^{-1}\points_\Omega)\leq\ln(2) < \infty$, so $h_\sigma(\T\on\theta) = 0$. We deal with this case below. Suppose that $\sigma(\Omega\times B) = 1$. For each $\omega\in\Omega$, the containment
\[
T_0(B)\Kin B
\]
together with the Schwarz-Pick lemma imply that
\[
\diam_B(\Supp(\sigma_1)) \leq \diam_B(\Supp(\sigma_0))
\]
with equality if and only if $\Supp(\sigma_0)$ is a singleton. By ergodicity, we have equality almost surely; thus $\sigma_0$ is almost surely a point measure. Thus $h_\sigma(\points_\X\on\pi^{-1}\points_\Omega) = 0$, so again $h_\sigma(\T\on\theta) = 0$.

Suppose that $\sigma\in \M_e(\X,\T,\basemeasure)$ satisfies $h_\sigma(\T\on\theta) = 0$. (\ref{tauisaboundintegral}) gives
\begin{equation*}
\int \phi\d\sigma
\leq \int \sup(\phi_\omega)\d\basemeasure(\omega)
< \int \ln\inf(L_\omega[\one])\d\basemeasure(\omega)
\leq \int\ln(\lambda(\omega))\d\basemeasure(\omega).
\end{equation*}
(The last inequality comes from integrating (\ref{measures}) against $\one$, taking logarithms, and integrating against $\d\basemeasure(\omega)$.

This completes the proof of (\ref{varprinciplelessthan}) in the case where $h_\sigma(\T\on\theta) = 0$. Thus the only case which remains is the case $h_\sigma(\T\on\theta) > 0$. In this case we also know that $\sigma((\Omega\times X)\butnot\SS) = 1$, and that there exists a partition of $(\X,\T,\sigma)$ which is generating relative to $(\Omega,\theta,\basemeasure)$ and which has finite relative entropy.
By Theorems \ref{theoremmanepartition} and \ref{propositionkolmogorovsinai}, we have (\ref{hequalsH}), where $(\sigma_{\omega,p})_{(\omega,p)\in\X}$ is the Rohlin decomposition of $\sigma$ relative to $\T^{-1}(\points_\X)$. Thus
\begin{equation*}
h_\sigma(\T\on\theta) + \int\psi\d\sigma
= \int \left[H_{\sigma_{\omega,p}}(\points_\X) + \int \psi\d\sigma_{\omega,p}\right]\d\sigma(\omega,p)
\end{equation*}
By the finite variational principle, the integrand is bounded above by
\begin{equation}
\label{finitevariationalprinciple}
\ln\left(\sideset{}{^*}\sum_{(\theta^{-1}\omega,x)\in\T^{-1}(\omega,p)}e^{\psi(\theta^{-1}\omega,x)}\right).
\end{equation}
(Recall that $*$ indicates that the sum is taken without multiplicity.) Since all terms are positive, the same sum with multiplicity is at least as large. Thus
\begin{align*}
H_{\sigma_{\omega,p}}(\points_\X) + \int \psi\d\sigma_{\omega,p}
&\leq \ln\left(\sum_{(\theta^{-1}\omega,x)\in\T^{-1}(\omega,p)}e^{\psi(\theta^{-1}\omega,x)}\right)\\
&= \ln(\L[\one](\omega,p))\\
&= \ln(1) = 0.
\end{align*}
Integrating against $\d\sigma(\omega,p)$ , we find that
\begin{equation*}
h_\sigma(\T\on\theta) + \int\psi\d\sigma \leq 0.
\end{equation*}
Expanding $\psi$ and rearranging,
\begin{equation*}
h_\sigma(\T\on\theta) + \int\phi\d\sigma \leq \int \ln(\lambda)\d\basemeasure.
\end{equation*}
Equality is achieved if and only if for $\sigma$-almost every $(\omega,p)\in\X$,
\begin{itemize}
\item[A)] The maximum in the finite variational principle is achieved i.e.
\[
\sigma_{\omega,p} = \frac{1}{C} \sideset{}{^*}\sum_{(\theta^{-1}\omega,x)\in\T^{-1}(\omega,p)}e^{\psi(\theta^{-1}\omega,x)}\delta_{(\theta^{-1}\omega,x)}
\]
where $C$ is the appropriate normalization constant.
\item[B)] The sum in (\ref{finitevariationalprinciple}) is the same whether or not multiplicity is counted. This happens if and only if $p$ is not a branch point of $T_{\theta^{-1}\omega}$.
\end{itemize}
We claim that (A) and (B) occur for $\sigma$-almost every $(\omega,p)$ if and only if $\sigma = \mu$. This will complete the proof due to the remarks below.

In the presence of (B), (A) is equivalent to the simpler equation
\begin{itemize}
\item[C)]
\begin{align*}
\sigma_{\omega,p} &= \sum_{(\theta^{-1}\omega,x)\in\T^{-1}(\omega,p)}e^{\psi(\theta^{-1}\omega,x)}\delta_{(\theta^{-1}\omega,x)}\\
&= \L^*[\delta_{(\omega,p)}].
\end{align*}
\end{itemize}

Thus, (A) and (B) hold if and only if (B) and (C) hold. We show that this occurs if and only if $\sigma = \mu$:
\begin{itemize}
\item[($\Rightarrow$)] Integrating (C) against $\d\sigma(\omega,p)$ yields
\begin{equation*}
\sigma
= \L^*\left[\int \delta_{(\omega,p)}\d\sigma(\omega,p)\right]
= \L^*[\sigma].
\end{equation*}

Since $\sigma$ is supported on $X\butnot\SS$, for all $\epsilon > 0$ there exists $\kappa > 0$ such that $\pr(\sigma_0(B_s(\SS_0,\kappa))\geq\epsilon) < 1$. Fix $\omega\in\Omega$ and assume that there exists a sequence $(n_k)_{k\in\N}$ of positive integers so that
\[
\sigma_{n_k}(B_s(\SS_{n_k},\kappa))\leq 2^{-k}.
\]
for all $k\in\N$. By Lemma \ref{lemmapoincare}, this assumption is almost certainly valid. Now, Remark \ref{remarkweakstar} gives
\begin{equation*}
\mu_0 = \lim_{k\rightarrow\infty}\L_{n_k}^0[\sigma_{n_k}]
= \lim_{k\rightarrow\infty}\sigma_0 = \sigma_0.
\end{equation*}
Since this is true for $\basemeasure$-almost every $\omega\in\Omega$, we have that $\sigma = \mu$.

\item[($\Leftarrow$)] Proposition \ref{propositionatomless} implies (B), since each rational map has only finitely many branch points. Disintegrating (\ref{measuresnu}) yields (C).
\end{itemize}
However, we are not done quite yet. We know that $\mu$ is supported on $\JJ\implies(\Omega\times X)\butnot\SS$, but we do not yet know that $h_\mu(\T\on\theta) > 0$. Nevertheless, the calculations made under this assumption still hold, since (\ref{hequalsH}) is always true (for any relative dynamical system) when $h_\sigma(\T\on\theta) = 0$. Thus we are done.
\end{section}
\begin{section}{Appendix: A logical test for measurability} \label{sectionappendix}
\begin{subsection}{Statements} \label{sectionstatements}
The purpose of this appendix is to explore the issue of measurability in relative dynamical systems. The idea is a simple one; namely that many statements can be seen to be measurable simply from the way that they are written. In fact, it is clear that in any language in which atomic propositions correspond to measurable sets and in which only quantification over countable sets is allowed, every proposition corresponds to a measurable set. However this is insufficient for our purposes because we would like to quantify over uncountable sets. For example, consider the following event:
\begin{event}
\label{eventexample}
There exists $p\in \JJ_0$ such that $T_0^n(\cl{B_s}(p,\delta_1))\supseteq \C\butnot B_s(\SS_n,\kappa)$. (Here $\JJ_0\implies \C$ is the random Julia set and $\SS_0\implies\C$ is the random exceptional set, defined in Section \ref{sectionbasic}.)
\end{event}
(The measurability of this event is used in the proof of (\ref{equationexample}), in order to be able to apply a continuity of measures argument.)

The domain of quantification $\JJ_0$ is almost certainly uncountable. However, this event will turn out to be measurable Corollary \ref{corollaryconventions}. Let us think about how to prove directly that Event \ref{eventexample} is measurable. Consider the random set
\[
\K_0 := \{p\in\C:T_0^n(\cl{B_s}(p,\delta_1))\supseteq \C\butnot B_s(\SS_n,\kappa)\}
\]
The question that Event \ref{eventexample} asks is whether $\JJ_0\cap\K_0\neq\emptyset$. If this were an open set, we could quantify over a countable dense subset of $\C$, and answer the same question. However $\JJ_0$ is a closed set, not an open set. It seems probable that most of the time $\JJ_0$ will not even intersect the countable dense set, rendering its detection power invalid. We solve this problem by noticing that $\JJ_0$ has the property of being ``strongly measurable", meaning that the map $(\omega,x)\mapsto\dist_s(x,\JJ_\omega)$ is jointly measurable (Remark \ref{remarkmeasurable}). This means that we can look at $\varepsilon$-neighborhoods of $\JJ_0$, where the countable dense set has detection power. Taking the limit as $\varepsilon$ goes to zero, we retain measurability since $\varepsilon$ can be quantified countably.

This solves the problem of detecting $\JJ_0$, but there is still the issue of detecting $\K_0$. Since $\K_0$ is linguistically non-atomic, we would like to have some way of checking that $\K_0$ is strongly measurable based on its subformulas. In fact, we will do this by inductively showing that its subformulas are continuous functions of $p$.

We begin with the following definitions. By ``locally compact'', we always mean that all closed and bounded subsets of $X$ are compact; any locally compact metrizable space has a compatible metric satisfying this condition.
\begin{itemize}
\item If $X$ is a locally compact separable metric space, let $\K(X)$ be the set of all compact subsets of $X$, endowed with the Vietoris topology (also known as the narrow topology, or the topology induced by the Hausdorff metric). Let $\F(X)$ be the set of all closed subsets of $X$, endowed with the Fell topology (also known as the vague topology). Both of these are Polish topologies [\cite{Mo} Theorem B.2(iii) p.399, Theorem C.8 p.405]. Precise definitions are given below in the second paragraph of the proof of Lemma \ref{lemmacontinuous}.
\item If $X$ is a locally compact separable metric space, and if $Y$ is a Polish space, let $\CC(X,Y)$ be the set of all continuous functions from $X$ to $Y$, endowed with the compact-open topology. This topology is Polish; in fact it is induced by the collection of pseudometrics $(\max_K(\dist))_{K\in\K(X)}$. In other words, this topology is the topology of locally uniform convergence.
\item If $X$ is a locally compact separable metric space, let $\M(X)$ be the set of all locally finite measures on $X$, endowed with the weak-* topology.
\item If $X$ is a compact Riemann surface with a Riemannian metric, let $\Div(X)$ be the set of all effective divisors on $X$, endowed with the quasimetric
\[
\dist(D_1,D_2) := \inf\{\delta:\exists \Phi:D_1\rightarrow D_2\text{ a bijection such that }\dist(x,\Phi(x))\leq\delta\all x\in D_1\}
\]
Note that $\Div(X)$ is a locally compact separable metric space.
\item If $X$ and $Y$ are Riemann surfaces, let $\AA(X,Y)$ be the set of all holomorphic maps from $X$ to $Y$, endowed with the compact-open topology. This topology is Polish, being a closed subspace of $\CC(X,Y)$. We write $\AA(X) := \AA(X,X)$.
\end{itemize}
The following lemmas are a collection of some well-known results, together with some (possibly) new ones. Proofs are given in Section \ref{sectionproofs}.
\begin{lemma}
\label{lemmacontinuous}
Suppose that $X$, $Y$, $Z$ are locally compact separable metric spaces. The maps \textup{(\ref{mapsstart}) - (\ref{mapsend})} are continuous, except for the starred maps which are only Borel measurable. For \textup{(\ref{MOC})} and \textup{(\ref{extend})}, assume that $X$ is a geodesic metric space i.e.
\[
\cl{B}(\cl{B}(x,\delta_1),\delta_2) = \cl{B}(x,\delta_1 + \delta_2)
\]
for all $x\in X$ and $\delta_1,\delta_2 > 0$. For \textup{(\ref{complexanalysisstart}) - (\ref{mapsend})}, assume that $X$, $Y$, and $Z$ are compact Riemann surfaces.
\end{lemma}
\begin{lemma}
\label{lemmaimplicitization}
Suppose that $(\Omega,\Par)$ is a measurable space, suppose that $X,Y,Z$ are locally compact separable metric spaces, and suppose that
\[
f:\Omega\times X\times Y\rightarrow Z.
\]
Then $f$ is measurable and fiberwise continuous (m.f.c.) if and only if the induced map
\[
\wtilde{f}:\Omega\times X\rightarrow \CC(Y,Z)
\]
is m.f.c.
\end{lemma}
\begin{lemma}
\label{lemmaclosed}
Suppose that $X$ is a locally compact separable metric space. For each of the expressions \textup{(\ref{atomicstart}) - (\ref{atomicend})}, the set of all tuples satisfying the quoted condition is closed.
\end{lemma}
\begin{lemma}
\label{lemmaquantifiers}
Suppose that $X$ and $Y$ are locally compact separable metric spaces. The maps
\begin{align}
\label{forallK2} (K,P) &\mapsto \forall_K(P) := \{x\in X:(x,y)\in P\all y\in K(x)\}\in\F(X)\\
\label{existsK2} (K,P) &\mapsto \exists_K(P) := \{x\in X:\exists y\in K(x)\;\;(x,y)\in P\}\in\F(X)
\end{align}
\[
[K\in\CC(X,\K(Y)),P\in\F(X\times Y)]
\]
are Borel measurable. (These correspond to \textup{(\ref{forallK})} and \textup{(\ref{existsK})}.)
\end{lemma}
\begin{description}
\item[Set-theoretic operations]
\begin{align} \label{mapsstart}
K&\mapsto K&&\in\F(X) & [K\in\K(X)] && \\
(f,x)&\mapsto f(x)&&\in Y & [x\in X,f\in\CC(X,Y)] && \\% Corollary \ref{corollarysixquantitative} [strange], Event \ref{eventinternalstandard} [exponent map]
(f_1,f_2)&\mapsto f_2\circ f_1&&\in \CC(X,Z) & [f_1\in\CC(X,Y),f_2\in\CC(Y,Z)] && \\% Corollary \ref{corollarysixquantitative} [strange], Event \ref{eventinternalstandard} [exponent map]
\label{xx} x&\mapsto \{x\}&&\in\K(X) & [x\in X] && \\% Remark \ref{remarkmeasurable}
\label{unionK} (K_1,K_2)&\mapsto K_1\cup K_2&&\in\K(X) & [K_1,K_2\in\K(X)] && \\% Event \ref{eventinternalstandard}
\label{unionF} (F_1,F_2)&\mapsto F_1\cup F_2&&\in\F(X) & [F_1,F_2\in\F(X)] && \\
(F,K)&\mapsto F\cap K&&\in\K(X) & [F\in\F(X),K\in\K(X)] && * \\% Corollary \ref{corollarymeasurable}, Event \ref{eventinternalstandard}
\label{intersectionF} (F_1,F_2)&\mapsto F_1\cap F_2&&\in\F(X) & [F_1,F_2\in\F(X)] && * \\
\label{forward} (f,K)&\mapsto f(K)&&\in\K(Y) & [K\in\K(X), f\in\CC(X,Y)] && \\% Lemma \ref{lemmamultiexact}
\label{backward} (f,F)&\mapsto f^{-1}(F)&&\in\F(X) & [F\in\F(Y), f\in\CC(X,Y)] && * 
\end{align}
\item[Arithmetic operations]
\begin{align}
(a,b)&\mapsto a + b&&\in\R & [a,b\in\R] && \\
(a,b)&\mapsto a - b&&\in\R & [a,b\in\R] && \\
(a,b)&\mapsto ab&&\in\R & [a,b\in\R] && \\% Event \ref{eventinternalstandard}
(a,b)&\mapsto a/b&&\in\R & [a\in\R,b\in\R\butnot\{0\}] && \\% Event \ref{eventinternalstandard}
(a,b)&\mapsto a^b&&>0 & [a > 0,b\in\R] && \\
(a,b)&\mapsto \min(a,b)&&\in\R & [a,b\in\R] && \\
(a,b)&\mapsto \max(a,b)&&\in\R & [a,b\in\R] && \\
a&\mapsto |a|&&\in\R & [a\in\R] && 
\end{align}
\item[Topological and metric operations]
\begin{align}
\label{molchanovstart} F&\mapsto \cl{X\butnot F}&&\in\F(X) & [F\in\F(X)] && * \\
F&\mapsto \del{F}&&\in\F(X) & [F\in\F(X)] && * \\
\label{molchanovend} F&\mapsto \cl{\text{int}(F)}&&\in\F(X) & [F\in\F(X)] && * \\
(x,y)&\mapsto d(x,y)&&\in[0,\infty) & [x,y\in X] && \\
\label{distK} (F,K)&\mapsto \dist(F,K)&&\in [0,\infty] & [F\in\F(X),K\in\K(X)] && \\% Lemma \ref{lemmamultiexact}
\label{distF} (F_1,F_2)&\mapsto \dist(F_1,F_2)&&\in [0,\infty] & [F_1,F_2\in\F(X)] && * \\
\label{coextend} (K,\delta)&\mapsto X\butnot B(K,\delta)&&\in\F(X) & [K\in\F(X),\delta\geq 0] && * \\% Corollary \ref{corollarysixquantitative}, Event \ref{eventinternalstandard}
\label{extend} (K,\delta)&\mapsto \cl{B}(K,\delta)&&\in\K(X) & [K\in\K(X),\delta\in(0,\infty),X\text{ is a g.m.s.}] && \\% Remark \ref{remarkmeasurable}
f&\mapsto\lim_{n\rightarrow\infty}f(n)&&\in X & [f\in\CC(\N,X),\text{ assuming the limit exists}] && * \\
\label{diamm} (K,m)&\mapsto \diam_m(F)&&\geq 0 & [K\in\K(X),m\in\N] && \\
\label{card} F&\mapsto \#(F)&&\in\N\cup\{\infty\} & [F\in\F(X)] && * % Corollary \ref{corollarymeasurable}
\end{align}
\item[Functional analysis operations]
\begin{align}
\label{supK} (K,f)&\mapsto \max_K(f)&&\in\R & [K\in\K(X),f\in\CC(X,\R)] && \\% Remark \ref{remarkmeasurable}, Lemma \ref{lemmamultiexact}
(K,f)&\mapsto \min_K(f)&&\in\R & [K\in\K(X),f\in\CC(X,\R)] && \\
(K,f)&\mapsto \|f\|_{\osc,K}&&\geq 0 & [K\in\K(X),f\in\CC(X,\R)] && \\% Event \ref{eventinternalstandard}
(K,f)&\mapsto \|f\|_{\infty,K}&&\geq 0& [K\in\K(X),f\in\CC(X,\R)] && \\
\label{supF} (F,f)&\mapsto \sup_F(f)&&\in\R & [F\in\F(X),f\in\CC(X,\R)] && * \\
\label{MOC} (f,\delta)&\mapsto \rho_f(\delta)&&\geq 0 & [f\in\CC(X,Y),\delta\geq 0,X\text{ is a compact g.m.s.}] && \\
%\label{MOCF} (F,f,\delta)&\mapsto \rho_f^{(F)}(\delta)&&\geq 0 & [F\in\F(X),f\in\CC(X,Y),\delta\geq 0] && \\
\label{fKalpha} (K,f,\alpha)&\mapsto \|f\|_{\alpha,K}&&\geq 0 & [K\in\K(X),f\in\CC(X,Y),\alpha > 0] && * \\
x&\mapsto \delta_x&&\in\M(X) & [x\in X] && \\
(f,\mu)&\mapsto f\mu&&\in\M(X) & [f\in\CC(X,\R),\mu\in\M(X)] && \\
\label{integral} (f,\mu)&\mapsto \int f\d\mu&&\in\R & [f\in\CC(X,\R),\mu\in\M(X),X\text{ is compact}] && \\% Proposition \ref{propositionequivalent}
(K,f,\mu)&\mapsto \int_K f\d\mu&&\in\R & [K\in\K(X),f\in\CC(X,\R),\mu\in\M(X)] && * 
%\mu&\mapsto \Supp(\mu)&&\in\F(X) & [\mu\in\M(X)] && * 
\end{align}
\item[Complex analysis operations]
\begin{align}
\label{complexanalysisstart} T&\mapsto T&&\in\CC(X,Y) & [T\in\AA(X,Y)] && \\
(T_1,T_2)&\mapsto T_2\circ T_1&&\in\AA(X,Z) & [T_1\in\AA(X,Y), T_2\in\AA(Y,Z)] && \\% Remark \ref{remarkmeasurable}
T&\mapsto \deg(T)&&\in\N & [T\in\AA(X,Y)] && \\% Corollary \ref{corollarymeasurable}
(T,x)&\mapsto \mult_T(x)&&\in\N & [T\in\AA(X,Y),x\in X] && * \\
T&\mapsto \|T_*\|&&\in\CC(X,[0,\infty)) & [T\in\AA(X)] && \\% Remark \ref{remarkmeasurable}, Lemma \ref{lemmamultiexact}
x&\mapsto [x] &&\in\Div(X) & [x\in X] && \\% Remark \ref{remarkmeasurable}, Corollary \ref{corollarysixquantitative}
T&\mapsto \RP_T &&\in\Div(X) & [T\in\AA(X,Y)] && \\
T&\mapsto \BP_T &&\in\Div(Y) & [T\in\AA(X,Y)] && \\
T&\mapsto \FP_T &&\in\Div(X) & [T\in\AA(X)] && \\
(D_1,D_2)&\mapsto D_1 + D_2 &&\in\Div(X) & [D_1,D_2\in\Div(X)] && \\
%D&\mapsto -D &&\in\Div(X) & [D\in\Div(X)] && \\
D&\mapsto \Supp(D)&&\in\F(X) & [D\in\Div(X)] && \\% Remark \ref{remarkmeasurable}, Corollary \ref{corollarysixquantitative}
D&\mapsto \deg(D)&&\in\N & [D\in\Div(X)] && \\
(D,f)&\mapsto\sum_D f&&\in\R & [D\in\Div(X),f\in\CC(X,\R)] && \\% Event \ref{eventinternalstandard}
(D,T)&\mapsto T^*D&&\in\Div(X) & [D\in\Div(Y),T\in\AA(X,Y)] && \\% Remark \ref{remarkmeasurable}, Corollary \ref{corollarysixquantitative}
(D,f)&\mapsto f_*D&&\in\Div(Y) & [D\in\Div(X),f\in\CC(X,Y)] && \label{mapsend}
\end{align}
\item[Implicitization]
\begin{align}% this is align not equation for formatting reasons
\label{implicitization} y(x) &\mapsto (x\mapsto y(x))&&\in \CC(X,Y) & [y\in Y\text{ with free variable }x\in X] && 
\end{align}
\item[Atomic propositions]
\begin{align}
\label{atomicstart} K &\mapsto ``K \neq \emptyset" & [K\in\K(X)]\\% Corollary \ref{corollarymeasurable}
F &\mapsto ``F = \emptyset" & [F\in\F(X)]\\
\label{equalsend} F &\mapsto ``F = X" & [F\in\F(X)]\\
(x,y) &\mapsto ``x = y" & [x,y\in X]\\% Corollary \ref{corollarymeasurable}
\label{inclusion} (x,F) &\mapsto ``x \in F" & [x\in X,F\in\F(X)]\\
(a,b) &\mapsto ``a \leq b" & [a,b\in\R]\\% Remark \ref{remarkmeasurable}, Corollary \ref{corollarysixquantitative}, Event \ref{eventinternalstandard}
\label{atomicend} (F_1,F_2) &\mapsto ``F_1 \implies F_2" & [F_1,F_2\in\F(X)]% Remark \ref{remarkmeasurable}
\end{align}
\item[Non-atomic propositions]
\begin{align}
\label{nonatomicstart} (P_1,P_2)&\mapsto ``P_1\text{ and }P_2" & && \\
\label{nonatomicor} (P_1,P_2)&\mapsto ``P_1\text{ or }P_2" & && \\
\label{nonatomicnot} P&\mapsto ``\text{not }P" & && * \\
\label{forallK} (K,P(x))&\mapsto ``\forall x\in K\;\; P(x)" & [K\in\K(X)] &&  \\% Remark \ref{remarkmeasurable}, Corollary \ref{corollarysixquantitative}
\label{existsK} (K,P(x))&\mapsto ``\exists x\in K\;\; P(x)" & [K\in\K(X)] &&  \\% Remark \ref{remarkmeasurable}, Lemma \ref{lemmamultiexact}
\label{forallF} (F,P(x))&\mapsto ``\forall x\in F\;\; P(x)" & [F\in\F(X)] && \dag \\% Remark \ref{remarkmeasurable}, Corollary \ref{corollarysixquantitative}
\label{nonatomicend} (F,P(x))&\mapsto ``\exists x\in F\;\; P(x)" & [F\in\F(X)] && * 
\end{align}
\end{description}

Inspired by this list, we define a language $\LL$. $\LL$ will consist of a pair $(\E_\LL,\mathbb{P}_\LL)$, where $\E_\LL$ and $\mathbb{P}_\LL$ are subsets of the set of all finite strings over the alphabet $\A$ consisting of all symbols in LaTeX. An element of $\E_\LL$ will be called an \emph{expression} in $\LL$, and an element of $\mathbb{P}_\LL$ will be called a \emph{proposition} of $\LL$.

Fix two disjoint subsets $\V,\mathbb{C}\implies \A$ which are disjoint from the set of symbols needed to do the operations (\ref{mapsstart}) - (\ref{mapsend}). (We distinguish between the syntax, e.g. $+$, $\deg$, $B(\cdot,\cdot)$, which are not allowed in $\V$ or $\mathbb{C}$, from the mere placeholders e.g. $F$, $T$, $x$, which are allowed.) An element of $\V$ is called a \emph{formal variable}, and an element of $\mathbb{C}$ is called a \emph{formal constant}. If a particular formula is given that you are trying to test the measurability of, then you should generally let $\V$ be the set of all variables which have been bound by quantifiers or implicitization, and let $\mathbb{C}$ be the set of all remaining variables used in the formula, For example, in Event \ref{eventexample},
\begin{align*}
\V &:= \{``p"\}\\
\mathbb{C} &:= \{``\JJ_0",``\SS_n",``T_0^n",``\delta_1",``\kappa"\}
\end{align*}

We define the set $\E_\LL$ by induction: A string is in $\E_\LL$ if and only if:
\begin{itemize}
\item It is a formal variable or constant
\item It is obtained by concatenating previously existing elements of $\E_\LL$ according to the rules (\ref{mapsstart}) - (\ref{mapsend}), with the qualification that the starred rules can only be used if each of the strings being concatenated contains no free variables.
\item It is obtained by ``implicitizing'' a formal variable according to rule (\ref{implicitization}). Specifically, if $e_1\in \E_\LL$ and $v_1\in \V$, then $e = (v_1\mapsto e_1(v_1))$ is a new element of $\E_\LL$.
\end{itemize}
For example, the string
\[
``(x\mapsto \cl{B}(x,\delta))"
\]
is proved to be in $\E_\LL$ in four steps:
\begin{itemize}
\item $``x"\in\V\implies\E_\LL$
\item $``\delta"\in\mathbb{C}\implies\E_\LL$
\item $``\cl{B}(x,\delta)"\in\E_\LL$ by rule (\ref{extend}) (here we assume that $X$ is a g.m.s.)
\item $``(x\mapsto \cl{B}(x,\delta))"$ by rule (\ref{implicitization}).
\end{itemize}

The set $\mathbb{P}_\LL$ is defined similarly: A string is in $\mathbb{P}_\LL$ if and only if either it is obtained by concatenating elements of $\E_\LL$ according to the rules (\ref{atomicstart}) - (\ref{atomicend}), or it is obtained by concatenating previously existing elements of $\mathbb{P}_\LL$ according to the rules (\ref{nonatomicstart}) - (\ref{nonatomicend}), again with the qualification that the starred rules can only be used if each of the strings being concatenated contains no free variables. The daggered rule can be used as long as the input string taking the place of the domain of quantification (i.e. $F$) contains no free variables.

Denote the set of free variables of an expression or proposition by $F(e)$ or $F(p)$; i.e. variables which occur in the string but are neither bound to a quantifier nor implicitized.

\begin{remark}
There are many subtleties in the language $\LL$. For example, the statement ``$F_1\implies F_2$" has different rules of construction than the logically equivalent statement ``$\forall x\in F_1,x\in F_2$". In fact, the latter has the $\dag$ restriction ($F_1$ cannot have free variables), whereas the former has no restriction. The reason for this is because of the precise implementation of measurability and continuity concepts in the Theorems \ref{theoremexpressions} and \ref{theorempropositions} below. They are not the only possible choice of inductive claims, and another choice could possibly yield a different language.
\end{remark}

\begin{remark}
Although this list cannot possibly be exhaustive, we have included expressions which are not directly relevant to this paper, on the grounds that they could be useful in the future.
\end{remark}

We now come to the issue of interpretation. Fix a measurable space $(\Omega,\Par)$, and for each $c\in \mathbb{C}$ fix a topological space $X_c$ and a Borel measurable map $c_*:\Omega\rightarrow X_c$. $c_*$ is called the \emph{interpretation} of $c$. For each $v\in \V$ fix a locally compact separable metric space $X_v$. ($X_c$ or $X_v$ may happen to equal a constructed space e.g. $X_c = \K(X)$ for some $X$.)

If $e\in\E_\LL$, we informally define the \emph{interpretation} of $e$ to be the map
\[
e_*: \Omega\times\prod_{v\in F(e)}X_v\rightarrow X_e
\]
which inputs a tuple $(\omega,(x_v)_v)$ and outputs the thing that you get when you plug in $(c_*(\omega))_c$ and $(x_v)_v$ into the string $e$. For example, if
\[
e = ``\dist(x,y)",
\]
then
\begin{align*}
X_e &:= [0,\infty)\\
e_*(\omega,(x,y)) &:= \dist(x,y) \in X_e;
\end{align*}
here we use the convention that $(x,y)_{``x"} = x$ and $(x,y)_{``y"} = y$. With a little work, $e_*$ can be defined inductively in a rigorous manner.

Similarly, for each proposition $p\in \mathbb{P}_\LL$, we informally define the \emph{interpretation} of $p$ to be the map
\[
p_*:\Omega\rightarrow \powerset\left(\prod_{v\in F(p)}X_v\right)
\]
which inputs $\omega$ and outputs the set
\[
p_*(\omega) := \{(x_v)_v:p\text{ is true when you plug in $(c_*(\omega))_c$ and $(x_v)_v$}\}
\]
For example, if
\[
p = ``x\in\JJ_0",
\]
then
\[
p_*(\omega) = \{x\in\C:x\in\JJ_0\on_\omega\} = \JJ_0\on_\omega.
\]
Again, with more work we could define $p_*$ in a rigorous manner by induction.

Note that if $p$ has no free variables then $p_*:\Omega\rightarrow\powerset(\{()\})\equiv \{\text{True},\text{False}\}$ can be reinterpreted as a subset of $\Omega$. (Here $()$ denotes the empty tuple.)

We have the following results, which are the only motivation for constructing so idiosyncratic a language.
\begin{theorem}
\label{theoremexpressions}
Fix $(\Omega,\Par)$, $(X_c)_{c\in \mathbb{C}}$, $(X_v)_{v\in \V}$, and $c_*:\Omega\rightarrow X_c$. For each $e\in \E_\LL$, we have that $e_*$ is m.f.c.
\end{theorem}
\begin{theorem}
\label{theorempropositions}
Fix $(\Omega,\Par)$, $(X_c)_{c\in \mathbb{C}}$, $(X_v)_{v\in \V}$, and $c_*:\Omega\rightarrow X_c$. For each $p\in \mathbb{P}_\LL$, then
\begin{enumerate}[A)]
\item For all $\omega\in\Omega$, $p_*(\omega)$ is closed i.e. $p_*(\omega)\in \F(\prod_{v\in F(p)}X_v)$
\item $p_*:\Omega\rightarrow\F(\prod_{v\in F(p)}X_v)$ is Borel measurable
\end{enumerate}
\end{theorem}
We call a map $p_*$ satisfying (A) and (B) \emph{strongly measurable} or \emph{s.m.}

The case $F(p) = \emptyset$ gives the following corollary:
\begin{corollary}
\label{corollaryconventions}
\textup{(Measurable Conventions)} In Theorem \ref{theorempropositions}, if $p$ has no free variables, then $p_*$, interpreted as a subset of $\Omega$ as above, is a measurable set.
\end{corollary}
\end{subsection}
\begin{subsection}{Proofs} \label{sectionproofs}
We omit the majority of the proofs. Those which are omitted are either obvious or well-known.

The proofs build on each other, in an order which is inconsistent with the order in which they are listed. In addition, some proofs demonstrate continuity and/or measurability by using prototype versions of Theorems \ref{theoremexpressions} and \ref{theorempropositions}, compiled using only functions which were already known to satisfy the requirements. However, there is no circular reasoning.

We follow the notation found in \cite{Mo}. Suppose that $G$ is open, $F$ is closed, and $K$ is compact. Let
\begin{align*}
\F_G &:= \{F\in\F(X):F\cap G\neq\emptyset\}\\
\K_G &:= \{K\in\K(X):K\cap G\neq\emptyset\}\\
\F^K &:= \{F\in\F(X):F\cap K = \emptyset\}\\
\K^F &:= \{K\in\K(X):K\cap F = \emptyset\}\\
\CC_K^G &:= \{f\in\CC(X,Y):f(K)\implies G\}.
\end{align*}
By definition, sets of the forms $\F_G,\F^K$ form a subbasis for $\F(X)$, sets of the form $\K_G,\K^F$ form a subbasis for $\K(X)$, and sets of the form $\CC_K^G$ form a subbasis for $\CC(X,Y)$. $\dist_H$ denotes the Hausdorff metric on $\K(X)$, which induces the Vietoris topology [\cite{Mo} Corollary C.6 p.404].

In some cases, we will use the existence of a sequence $K_n := \cl{B}(a,n)\in\K(X)$ such that $X = \bigcup_{n\in\N}K_n$, and of a countable dense set $Q\implies X$.

\begin{proof}[Proof of Lemma \ref{lemmacontinuous}:]~
\begin{description}
\item[(\ref{xx})]
\begin{align*}
\{x:\{x\}\in \K_G\} &= G\\
\{x:\{x\}\in \K^F\} &= X\butnot F
\end{align*}
are open.
\item[(\ref{unionK})]
\begin{align*}
\{(K_1,K_2):K_1 \cup K_2\in\K_G\} &= \K_G \times \K \cup \K\times \K_G\\
\{(K_1,K_2):K_1 \cup K_2\in\K^F\} &= \K^F\times \K^F
\end{align*}
are open.
\item[(\ref{forward})]
Fix $f_0\in\CC(X,Y)$, $K_0\in\K(X)$, and $\varepsilon > 0$. Fix $\delta_1 > 0$ such that $\cl{B}(K_0,\delta_1)$ is compact. Fix $\delta_2 > 0$ so that
\[
\rho_{f_0}^{(\cl{B}(K_0,\delta_1))}(\delta_2) \leq \varepsilon/2.
\]
Let $\delta_3 = \min(\delta_1,\delta_2) > 0$. Suppose that $f\in\CC(X,Y)$ and $K\in\K(X)$ are close enough to $f_0$ and $K_0$ so that
\begin{align*}
\max_{\cl{B}(K_0,\delta_1)}\dist(f,f_0) &\leq \varepsilon/2\\
\dist_H(K,K_0) &\leq \delta_3.
\end{align*}
Now
\begin{align*}
\dist_H(f(K),f_0(K_0))
&\leq \dist_H(f(K),f_0(K)) + \dist_H(f_0(K),f_0(K_0))\\
&\leq \varepsilon/2 + \varepsilon/2 \leq \varepsilon.
\end{align*}
\item[(\ref{backward})]
For each $K\in \K(X)$,
\[
f^{-1}(F)\cap K = \emptyset \Leftrightarrow f(K)\cap F = \emptyset.
\]
Now, the map $(f,F)\mapsto f(K)\cap F$ is measurable by Theorem \ref{theoremexpressions}. Thus the set
\[
((f,F)\mapsto f^{-1}(F))^{-1}(\F^K) = ((f,F)\mapsto f(K)\cap F)^{-1}(\F(Y)\butnot \F_Y)
\]
is Borel measurable. Since the sets $(\F^K)_K$ form a basis for the $\sigma$-algebra of Borel sets, the map $(f,F)\mapsto f^{-1}(F)$ is Borel measurable.
\item[(\ref{supK})]
\[
\max_K(f) = \max(f(K))
\]
Noting that the map $K\mapsto\max(K)\hspace{.15 in}[K\in\K(\R)]$ is Lipschitz $1$-continuous completes the proof.
\item[(\ref{supF})]
\[
\sup_F(f) = \sup_{n\in\N}\max_{F\cap K_n}(f)
\]
\item[(\ref{MOC})]
\[
\rho_f(\delta) = \max_{x\in X}\max_{y\in \cl{B}(x,\delta)}|f(x) - f(y)|
\]
\item[(\ref{fKalpha})]
\[
\|f\|_{\alpha,K} = \sup_{\substack{\delta > 0 \\ \text{rational}}}\max_{x\in K}\max_{y\in K\butnot B(x,\delta)}\frac{f(y) - f(x)}{q^\alpha}
\]
\item[(\ref{integral})]
\[
\int_K f\d\mu = \inf_{\substack{\delta > 0 \\ \text{rational}}}\int (x\mapsto \phi_\delta(\dist(x,K))) f\d\mu,
\]
where
\[
\phi_\delta(t) := \begin{cases}
0 & t \geq \delta\\
1 - t/\delta & 0\leq t < \delta\\
1 & t < 0
\end{cases}
\]
\item[(\ref{molchanovstart}) - (\ref{molchanovend})]
[\cite{Mo} Theorem 2.25 p.37 (iii)]
\item[(\ref{distK})]
\[
\dist(F,K) = \min_{x\in K}\dist(x,F)
\]
To see that the map $(x,F)\mapsto\dist(x,F)$ is continuous, fix $x_0\in X$, $F_0\in\F(X)$, and $\varepsilon > 0$, and let $d_0 = \dist(x_0,F_0)$. If $x\in X$ and $F\in\F(X)$ are close enough to $x_0$ and $F_0$ so that
\begin{align*}
\dist(x,x_0) &\leq \varepsilon/2\\
F &\in \F_{B(x_0,d_0 + \varepsilon/2)}\cap \F^{\cl{B}(x_0,d_0 - \varepsilon/2)},
\end{align*}
then
\[
|\dist(x,F) - \dist(x_0,F_0)| \leq \varepsilon.
\]
Here we have used the fact that all closed and bounded subsets of $X$ are compact. Without this condition we could only show that $(K_1,K_2)\mapsto \dist(K_1,K_2)$ is continuous for $K_1,K_2\in\K(X)$.
\item[(\ref{distF})]
\[
\dist(F_1,F_2) = \inf_{n\in\N}\dist(F_1\cap K_n,F_2\cap K_n)
\]
\item[(\ref{coextend})] If $\wtilde{K}\in\K(X)$,
\[
(X\butnot B(K,\delta))\in \F^{\wtilde{K}} \Leftrightarrow \dist(K,\wtilde{K}) \geq \delta.
\]
\item[(\ref{extend})]
The map is Lipschitz $1$-continuous with respect to each input.
\item[(\ref{diamm})]
\[
\diam_m(K) = \max_{x_1\in K}\cdots\max_{x_m\in K}\min_{\substack{i,j \\ i\neq j}}\dist(x_i,x_j).
\]
(See Section \ref{sectionexact} for definition of $\diam_m$.)
\item[(\ref{card})]
\[
\#(F)\geq m \Leftrightarrow \exists n\text{ such that } \diam_m(F\cap K_n) > 0.
\]
\end{description}
\end{proof}
\begin{proof}[Proof of Lemma \ref{lemmaimplicitization}:]
For the backwards direction, note that $f(\omega,x,y) = \wtilde{f}(\omega,x)(y)$. (Indeed, this is the definition of $\wtilde{f}$.) We prove the forward direction:

Fix $K\implies Y$ compact and $U\implies Z$ open, so that $\CC_K^U$ is an arbitrary basic open subset of $\CC(Y,Z)$. We will show that the set $\wtilde{f}^{-1}(\CC_K^U)$ is measurable and fiberwise open.

Let $Q\implies K$ be countable dense. Then
\[
\wtilde{f}^{-1}(\CC_K^U) = \bigcap_{y\in Q}(f\circ i_y)^{-1}(U),
\]
where $i_y:\Omega\times X\rightarrow \Omega\times X\times Y$ is the obvious inclusion. This proves measurability.

Fix $\omega\in\Omega$ and $x\in\wtilde{f}\on_\omega^{-1}(\CC_K^U)$. Since $f\on_\omega$ is continuous, for each $y\in K$ there exist $V_y\implies X$ and $W_y\implies Y$ open neighborhoods of $x$ and $y$ respectively so that
\[
f(\omega,V_y,W_y)\implies U.
\]
Let $(W_y)_{y\in F}$ be a finite subcover i.e. $K\implies \bigcup_{y\in F}W_y$. Then
\begin{align*}
f\left(\omega,\bigcap_{y\in F}V_y,K\right)&\implies U\\
x\in \bigcap_{y\in F}V_y&\implies\wtilde{f}\on_\omega^{-1}(\CC_K^U).
\end{align*}
Since $x\in\wtilde{f}\on_\omega^{-1}(\CC_K^U)$ was arbitrary, $\wtilde{f}\on_\omega^{-1}(\CC_K^U)$ is open.\
\end{proof}
\begin{proof}[Proof of Lemma \ref{lemmaclosed}:]~
\begin{description}
\item[(\ref{atomicstart}) - (\ref{equalsend})]
\begin{align*}
\{K\in\K(X):K\neq\emptyset\} &= \K(X)\butnot \K^X \\
\{F\in\F(X):F=\emptyset\} &= \F(X)\butnot \F_X \\
\{F\in\F(X):F=X\} &= \bigcap_{x\in X}\F(X)\butnot \F^{\{x\}} \\
\end{align*}
\item[(\ref{inclusion})]
\[
\{(x,F):x\in F\} = \{(x,F):\dist(x,F) = 0\}
\]
which is closed by Theorem \ref{theorempropositions}.
\item[(\ref{atomicend})]
\[
\{(F_1,F_2): F_1\implies F_2\} = \bigcap_{x\in X}\{(F_1,F_2):\dist(x,F_2)\leq \dist(x,F_1)\}
\]
which is closed by Theorem \ref{theorempropositions}.
\end{description}
\end{proof}
\begin{proof}[Proof of Lemma \ref{lemmaquantifiers}:]
We rewrite (\ref{forallK2}) - (\ref{existsK2}):
\begin{align*}
\{x\in X:(x,y)\in P\all y\in K(x)\} &= \left(x\mapsto\max_{y\in K(x)}\dist((x,y),P)\right)^{-1}(0) \in\F(X)\\
\{x\in X:\exists y\in K(x)\text{ such that }(x,y)\in P\} &= \left(x\mapsto\min_{y\in K(x)}\dist((x,y),P)\right)^{-1}(0)
\in\F(X),
\end{align*}
both of which are measurable by Corollary \ref{corollaryconventions}. (Here we let $\Omega = \CC(X,\K(Y))\times\F(X\times Y)$, and let $``K"_*$ and $``P"_*$ be the first and second projections.)
\end{proof}
\begin{remark}
This proof is deceptively simple, since it depends on all the theory which has been developed so far. Lemma \ref{lemmaquantifiers} is probably the most important result in this paper.
\end{remark}
\begin{proof}[Proof of Theorem \ref{theoremexpressions}:]~
\begin{description}
\item[Base case]~
\begin{description}
\item[$e\in \mathbb{C}$] $e_*$ is measurable by hypothesis. Since $e$ has no free variables $e_*$ is automatically fiberwise continuous.
\item[$e\in \V$] In this case $e_*$ is just projection onto the $e$th coordinate, which is clearly m.f.c.
\end{description}
\item[Inductive step]~
\begin{description}
\item[Case 1]
Suppose that $e_1,\ldots,e_n$ are expressions of $\LL$ with m.f.c. interpretations $e_{1*},\ldots,e_{n*}$, and suppose that
\[
e = ``f(e_1,\ldots,e_n)",
\]
where $f$ is one of the functions (\ref{mapsstart}) - (\ref{mapsend}).
\begin{description}
\item[1A]
If $f$ is continuous (i.e. is unstarred), then 
\[
e_* := f\circ (e_{1*},\ldots,e_{n*})
\]
is Borel measurable, and for each $\omega\in\Omega$
\[
e_*\on_\omega := f\circ (e_{1*}\on_\omega,\ldots,e_{n*}\on_\omega)
\]
is continuous.
\item[1B] If $f$ is only Borel measurable (i.e. is starred), then $e_*$ is still Borel measurable. Since $e$ has no free variables, for each $\omega\in\Omega$ the domain of $e_*\on_\omega$ is a singleton, and thus $e_*\on_\omega$ cannot fail to be continuous.
\end{description}
\item[Case 2]
Alternatively, suppose that rule (\ref{implicitization}) is being used i.e. suppose that $e_1$ is an expression of $\LL$, suppose that $v_1\in F(e_1)$, and suppose that
\[
e = ``(v_1\mapsto e_1(v_1))".
\]
By assumption
\[
e_{1*}:\Omega\times \prod_{v\in F(e_1)\butnot\{v_1\}}X_v \times X_{v_1}\rightarrow X_{e_1}
\]
is m.f.c., so by Lemma \ref{lemmaimplicitization},
\[
e_*:\Omega\times \prod_{v\in F(e_1)\butnot\{v_1\}}X_v \rightarrow X_e := \CC(X_{v_1},X_{e_1})
\]
is m.f.c. Clearly $F(e) = F(e_1)\butnot\{v_1\}$, so we are done.
\end{description}
\end{description}
\end{proof}
\begin{proof}[Proof of Theorem \ref{theorempropositions}]~
\begin{description}
\item[Base case]
Suppose that $e_1,\ldots,e_n$ [$n = 1,2$] are expressions of $\LL$ with m.f.c. interpretations $e_{1*},\ldots,e_{n*}$, and suppose that
\[
p = s(e_1,\ldots,e_n)
\]
where $s$ is one of the maps (\ref{atomicstart}) - (\ref{atomicend}). (Note that $p$ is therefore a string, since the maps (\ref{atomicstart}) - (\ref{atomicend}) have quotation marks.) In particular, by Lemma \ref{lemmaclosed} the set $s_*$ consisting of all tuples $x_1,\ldots,x_n$ such that $s(x_1,\ldots,x_n)$ is a true statement is closed i.e.
\[
s_*\in\F\left(\prod_{i=1}^n X_{e_i}\right).
\]
Interpreted as a map from $\Omega$ to $\F\left(\prod_{i=1}^n X_{e_i}\right)$, $s_*$ is measurable since its range is a singleton. Let
\[
e_*:\Omega\times\prod_{v\in F(p)}X_v\rightarrow\prod_{i=1}^n X_{e_i}
\]
be the product of $e_{1*},\ldots,e_{n*}$; $e_*$ is m.f.c. By definition
\[
p_*(\omega) := \{(x_v)_{v\in F(p)}:e_*(\omega,(x_v)_v)\in s_*\}.
\]
By Lemma \ref{lemmaimplicitization}, the map
\[
\wtilde{e_*}:\Omega\rightarrow\CC\left(\prod_{v\in F(p)}X_v,\prod_{i=1}^n X_{e_i}\right)
\]
is measurable. Now
\[
p_*(\omega) = (\wtilde{e_*}(\omega))^{-1}(s_*) = ((f,F)\mapsto f^{-1}(F))\circ(\wtilde{e_*},s_*)(\omega)
\]
By Lemma \ref{lemmacontinuous}, in particular (\ref{backward}),
\[
\omega \mapsto p_*(\omega) \in \F\left(\prod_{v\in F(p)}X_v\right)
\]
is measurable.

\item[Inductive step]~
\begin{description}
\item[(\ref{nonatomicstart}),(\ref{nonatomicor})]
Suppose that $p_1,p_2$ are propositions of $\LL$ with s.m. interpretations $p_{1*},p_{2*}$, and suppose that
\[
p = ``p_1 \cap p_2".
\]
Then
\[
p_*(\omega) = p_{1*}(\omega) \cap p_{2*}(\omega) = \cap(p_{1*}(\omega),p_{2*}(\omega))
\]
which is clearly closed; the map $\omega\mapsto p_*(\omega)$ is measurable by (\ref{intersectionF}). (\ref{nonatomicor}) is proved similarly.

Note that we do not need a star here, despite the fact that (\ref{intersectionF}) has a star. This is because the domain of $p_*$ is just $\Omega$, which has no topology, rendering continuity unnecessary.
\item[(\ref{nonatomicnot})] Suppose that $p_1$ is a proposition of $\LL$ with s.m. interpretation $p_{1*}$ and with no free variables, and suppose that
\[
p = ``\text{not }p_1".
\]
Then
\[
p_*(\omega) := \left(\prod_{v\in F(p)}X_v\right)\butnot p_{1*}(\omega).
\]
Since $p$ has no free variables, 
\[
\prod_{v\in F(p)}X_v = \{()\},
\]
so clearly (A) is satisfied. The map
\[
i: F\mapsto \{()\}\butnot F
\]
is measurable because it is a permutation of the finite set $\powerset(\{()\})$. Thus 
\[
p_* = i \circ p_{1*}
\]
is measurable, so (B) is satisfied.
\item[(\ref{forallK}),(\ref{existsK})]
Suppose that $e_1,p_1$ are an expression and a proposition of $\LL$, respectively, with interpretations $e_{1*},p_{1*}$, with $e_{1*}$ m.f.c. and $p_{1*}$ s.m. Suppose that
\[
p = ``\forall v\in e_1, p_1(v_1)",
\]
where $v_1\in F(p_1)\butnot F(e_1)$. By definition
\[
F(p) = F(e_1)\cup F(p_1)\butnot\{v_1\}.
\]
Let
\begin{align*}
X &:= X_{v_1}\\
Y &:= \prod_{v\in F(p)}X_v,
\end{align*}
so that
\[
X\times Y = \prod_{v\in F(e_1)\cup F(p_1)}X_v.
\]
Now
\[
p_*(\omega) = \{y\in Y:\forall x\in e_{1*}(\omega), (x,y)\in p_{1*}(\omega)\} = \forall_{e_{1*}(\omega)}(p_{1*}(\omega));
\]
by Lemma \ref{lemmaquantifiers}, $p_*$ is s.m. (\ref{existsK}) is proved similarly.
\item[(\ref{forallF}),(\ref{nonatomicend})]
\begin{align*}
\forall x\in F\;\; P(x) &\Leftrightarrow \forall n\in\N\;\; \forall x\in F\cap K_n\;\; P(x)\\
\exists x\in F\;\; P(x) &\Leftrightarrow \exists n\in\N\;\; \exists x\in F\cap K_n\;\; P(x).
\end{align*}
Since a countable union of closed sets may not be closed unless they are subsets of $\{()\}$, the $*$ requirement must be imposed on (\ref{nonatomicend}), whereas only the $\dag$ requirement is needed on (\ref{forallF}).
\end{description}
\end{description}
\end{proof}
\end{subsection}
\end{section}
\begin{section}{Notational conventions}\label{sectionnotation}

We make the following miscellaneous notational conventions:

% Set theory
We consider $\N = \{0,1,\ldots\}$; in particular, $0\in\N$.

Unless explicitly stated, variables are allowed to take on the value $\infty$. However, they are nonnegative unless otherwise stated.

$\points_S$ is the partition of $S$ into points.

All measurable spaces are assumed to be standard Borel.

% Metric/topological spaces
If $U\implies\C$, we denote the collection of connected components of $U$ by $\connected(U)$.

$K \Kin U$ means that $K$ is relatively compact in $U$.

% Complex analysis
The set of ramification points of a rational map $T$ is denoted $\RP_T$; the set of branch points is denoted $\BP_T$, so that $\BP_T = T(\RP_T)$. The set of fixed points is denoted $\FP_T$.

We denote the (local) spherical metric on $\C$ by $s$, the corresponding distance function (global metric) by $\dist_s$, and the corresponding area measure by $\lambda_s$. Similarly, $e$ is the local Euclidean metric on $\complexplane$. We assume that the spherical area measure is normalized so that $\lambda_s(\C)=1$; however the metric remains standard. Although these normalizations are inconsistent with each other, no contradiction will occur since we will not move between them. If $T$ is a rational function, then $T_*$ denotes the map $\|T_*\|:\C\rightarrow\R$ which sends a point to the operator norm of its derivative interpreted as a map on tangent spaces.

If $U$ is a hyperbolic Riemann surface, we denote the Poincar\'e metric on $U$ by $h_U$, and the corresponding distance function by $\dist_U$.

By $\B$ we mean the Poincar\'e disk $B_e(0,1)$; often we care only about the fact that it is a simply connected hyperbolic Riemann surface, and not the embedding in $\C$. The Poincar\'e metric on $\B$ we shorten $h := h_\B$, and the corresponding distance function we denote $\dist_h$.

As a result of these conventions we have
\begin{align*}
\tan(\dist_s(0,x)) &= \dist_e(0,x) = \tanh(\dist_h(0,x))\\
B_s(0,\delta) &= B_e(0,\tan(\delta))\\
B_h(0,\delta) &= B_e(0,\tanh(\delta))\\
\diam_s(\C) &= \pi/2\\
\lambda_s(B_s(x,\delta)) &= \sin(\delta)
\end{align*}
A useful inequality in connection with the last equation is the inequality
\[
\frac{\sin(b)}{\sin(a)} < \frac{b}{a},
\]
valid whenever $0 < a < b \leq \pi/2$. (It follows from the fact that $\sin$ is strictly concave down on this interval.)

% Functional analysis/measure theory
$\delta_x$ denotes the point measure centered at $x$; $\one_A$ denotes the characteristic function of the set $A$.

If $X$ is a topological space, then $\CC(X)$ denoted the space of all continuous (real-valued) functions on $X$, and $\M(X)$ denotes the space of all nonnegative measures on $X$. If $X$ is a compact metric space, then $\M(X)\implies\CC^*(X)$ is given the weak-* topology.

If $f\in\CC(\C)$, and $K\implies \C$, we denote the modulus of continuity of $f$ relative to $K$ by
\[
\rho_f^{(K)}(\varepsilon):=\sup_{\substack{x,y\in K \\ \dist_s(x,y)\leq\varepsilon}}|f(y) - f(x)|.
\]
The absolute modulus of continuity we denote $\rho_f := \rho_f^{(\C)}$. Similarly, we define the relative and absolute oscillation
\[
\|f\|_{\osc,K} := \sup_K(f) - \inf_K(f) = \rho_f^{(K)}(\pi/2)
\]
and
\[
\|f\|_\osc := \|f\|_{\osc,\C}.
\]

We define a \emph{modulus of continuity} to be a nondecreasing function $\gamma:(0,\infty)\rightarrow(0,\infty)$ so that $\gamma(\varepsilon)\tendstoeps 0$. It follows that $\rho_f^{(K)}$ is a modulus of continuity, since $f$ is uniformly continuous on $K$.

We define the \emph{local $\alpha$ norm} on a function $\phi\in\CC(\C)$ to be $\|\phi\|_\AL := \sup_{\varepsilon>0}\frac{\rho_\phi(\varepsilon)}{\varepsilon^\alpha}$.

If $(f_n)_n$ is a sequence of functions and $(K_n)_n$ is a sequence of sets, we say that $(f_n)_n$ tends to a constant $C$ uniformly on $(K_n)_n$ if $\|f_n - C\|_{\infty,K_n} \tendston 0$.

Operator norms will be notated in the following way: If $R$ is an operator, then $\|R\|_A^B$ means the operator norm of $R$ where norms in the domain are taken according to the (pseudo)norm $\|\cdot\|_A$ and norms in the range are taken according to the (pseudo)norm $\|\cdot\|_B$. (In other words, $\|R\|_A^B=\sup\{\|R[v]\|_B:\|v\|_A\leq 1\}$.) If both norms are the same, we write $\|R\|_A=\|R\|_A^A$.

If $f:X\rightarrow Y$ is a function, then we denote the forward image of a measure $\mu\in\M(X)$ under $f$ by $f_*[\mu]$; i.e. $f_*[\mu](A) = \mu(f^{-1}(A))$. (The use of brackets rather than parentheses is because $f$ is linear when interpreted as a map on measures.)

% Formatting
By a \emph{potential function} we mean a continuous function from $\C$ to $\R$.

A sequence of objects $(S_n)_n$ is \emph{$(T_n)_n$-invariant} if $S_m = T_n^m(S_n)$ for all $m,n\in\Z$ with $m\leq n$. (vice-versa if the objects move in the opposite direction)
\end{section}


\begin{thebibliography}{999999}

\bibitem[AR]{AR} L. Abramov and V. Rohlin, \emph{Entropy of a skew product of mappings with invariant measure}, (Russian), Vestnik Leningrad. Univ. {\bf 17} (1962) no. 7, 5-13. 

\bibitem[Ar98]{Ar} L. Arnold, \emph{Random dynamical systems}, Springer Monographs in Mathematics. Springer-Verlag, Berlin, 1998

\bibitem[BB95]{BB} J. Bahnm\"uller and T. Bogensch\"utz, \emph{A Margulis-Ruelle inequality for random dynamical systems}, Arch. Math. (Basel) {\bf 64} (1995), no. 3, 246-253.

\bibitem[Bo92]{Bo} T. Bogensch\"utz, \emph{Entropy, pressure, and a variational principle for random dynamical systems}, Random and Computational Dynamics {\bf 1} (1992/93), no. 1, 99-116.

\bibitem[CG93]{CG} L. Carleson and T. Gamelin, \emph{Complex dynamics}, Universitext: Tracts in Mathematics. Springer-Verlag, New York, 1993.

\bibitem[DU91]{DU} M. Denker and M. Urba\'nski, \emph{Ergodic theory of equilibrium states for rational maps}, Nonlinearity {\bf 4} (1991), no. 1, 103-134

% \bibitem[FS91]{FS} J. Fornaess and N. Sibony, \emph{Random iterations of rational functions}, Ergodic Theory and Dynamical Systems {\bf 11} (1991), no. 4, 687-708.

\bibitem[Gr03]{Gr} M. Gromov, \emph{On the entropy of holomorphic maps}, Enseign. Math. (2) {\bf 49} (2003), no. 3-4, 217-235

\bibitem[Jo00]{Jo} M. Jonsson, \emph{Ergodic properties of fibered rational maps}, Arkiv f\"or Matematik {\bf 38} (2000), no. 2, 281-317.

\bibitem[Ki86]{Ki} Y. Kifer, \emph{Ergodic theory of random transformations}, Progress in Probability and Statistics, 10. Birkh\"auser Boston, Inc., Boston, MA, 1986.

\bibitem[Ly83]{Ly} M. Lyubich, \emph{Entropy properties of rational endomorphisms of the Riemann sphere}, Ergodic Theory Dynam. Systems {\bf 3} (1983), no. 3, 351-385. 

% \bibitem[LW77]{LW} F. Ledrappier and P. Walters, \emph{A Relativised Variational Principle for Continuous Transformations}, J. London Math. Soc. (2) {\bf 16} (1977), no. 3, 568-576.

\bibitem[Ma83]{Ma1} R. Ma\~n\'e, \emph{On the uniqueness of the maximizing measure for rational maps}, Bol. Soc. Brasil. Mat. {\bf 14} (1983), no. 1, 27-43.

\bibitem[Ma85]{Ma2} R. Ma\~n\'e, \emph{On the Bernoulli property of rational maps}, Ergodic Theory Dynam. Systems {\bf 5} (1985), no. 1, 71-88.

\bibitem[Mil99]{Mil} J. Milnor, \emph{Dynamics in one complex variable}, Introductory lectures. Friedr. Vieweg \& Sohn, Braunschweig, 1999.

\bibitem[Mir95]{Mir} R. Miranda, \emph{Algebraic curves and Riemann surfaces}, Graduate Studies in Mathematics, 5. American Mathematical Society, Providence, RI, 1995.

\bibitem[Mo05]{Mo} I. Molchanov, \emph{Theory of random sets}, Probability and its Applications (New York). Springer-Verlag London, Ltd., London, 2005.

%\bibitem[MSU]{MSU} V. Mayer, B. Skolruski, M. Urba\'nski, \emph{Distance Expanding Random Mappings, Thermodynamic Formalism, Gibbs Measures, and Fractal Geometry}, submitted

\bibitem[Pr90]{Prz} F. Przytycki, \emph{On the Perron-Frobenius-Ruelle operator for rational maps on the Riemann sphere and for H\"older continuous functions}, Bol. Soc. Brasil. Mat. (N.S.) {\bf 20} (1990), no. 2, 95-125. 

\bibitem[PU10]{PU} F. Przytycki and M. Urba\'nski, \emph{Conformal fractals: ergodic theory methods}, London Mathematical Society Lecture Note Series, 371. Cambridge University Press, Cambridge, 2010.

\bibitem[Si]{SG} D. Simmons, \emph{A lower bound for the size of the spectral gap of transfer operators}, preprint 2010.

\end{thebibliography}
\end{document}